\documentclass[11pt, twoside]{amsart}

\usepackage[all]{xy}
\usepackage{amsmath}
\usepackage{hyperref}
\usepackage{amsfonts,graphics,amsthm,amsfonts,amscd,latexsym}
\usepackage{epsfig}
\usepackage{flafter}
\usepackage{mathtools}

\hypersetup{
    colorlinks=true,    
    linkcolor=blue,          
    citecolor=blue,      
    filecolor=blue,      
    urlcolor=blue           
}
\usepackage{tikz}
\usetikzlibrary{graphs,positioning,arrows,shapes.misc,decorations.pathmorphing}

\tikzset{
    >=stealth,
    every picture/.style={thick},
    graphs/every graph/.style={empty nodes},
}

\tikzstyle{vertex}=[
    draw,
    circle,
    fill=black,
    inner sep=1pt,
    minimum width=5pt,
]
\usepackage[position=top]{subfig}
\usepackage[backrefs, alphabetic, initials]{amsrefs}
\usepackage{amssymb}
\usepackage{color}

\usepackage{geometry}

\usetikzlibrary{decorations.pathmorphing}
\tikzstyle{printersafe}=[decoration={snake,amplitude=0pt}]

\newcommand{\codim}{\operatorname{codim}}

\newcommand{\mult}{\operatorname{mult}}

\newcommand{\Spec}{\operatorname{Spec}}

\renewcommand{\qq}{\mathbb{Q}}
\newcommand{\zz}{\mathbb{Z}}
\newcommand{\nn}{\mathbb{N}}
\newcommand{\rr}{\mathbb{R}}
\newcommand{\cc}{\mathbb{C}}

\def\O#1.{\mathcal {O}_{#1}}			
\def\pr #1.{\mathbb P^{#1}}				
\def\af #1.{\mathbb A^{#1}}				
\def\ses#1.#2.#3.{0\to #1\to #2\to #3 \to 0}		
\def\xrar#1.{\xrightarrow{#1}}			
\def\K#1.{K_{#1}}						
\def\bA#1.{\mathbf{A}_{#1}}				
\def\bM#1.{\mathbf{M}_{#1}}				
\def\bL#1.{\mathbf{L}_{#1}}				
\def\bB#1.{\mathbf{B}_{#1}}				
\def\bK#1.{\mathbf{K}_{#1}}				
\def\subs#1.{_{#1}}						
\def\sups#1.{^{#1}}						

\DeclareMathOperator{\Supp}{Supp}		

\DeclareMathOperator{\Nklt}{Nklt}	

\newcommand{\rar}{\rightarrow}		
\newcommand{\drar}{\dashrightarrow}	

\usepackage{tikz}
\usetikzlibrary{matrix,arrows,decorations.pathmorphing}

  \newtheorem*{ks.princ}{
  Connectedness Principle}
  \newtheorem{theorem}{Theorem}[section]
  \newtheorem{lemma}[theorem]{Lemma}
  \newtheorem{proposition}[theorem]{Proposition}
  \newtheorem{corollary}[theorem]{Corollary}

\theoremstyle{definition}
  
  \newtheorem{definition}[theorem]{Definition}
  \newtheorem{example}[theorem]{Example}

\newtheorem{remark}[theorem]{Remark}

\theoremstyle{remark}

\numberwithin{equation}{section}

\newcommand{\strutt}{\mathcal{O}}
\newcommand{\AF}{\lceil A \rceil - \lfloor F \rfloor}

\makeatletter
\@namedef{subjclassname@2020}{%
  \textup{2020} Mathematics Subject Classification}
\makeatother

\begin{document}
\title{On the connectedness principle and dual complexes for generalized pairs}

\author[S. Filipazzi]{Stefano Filipazzi}
\address{EPFL, SB MATH-CAG, MA C3 625 (B\^{a}timent MA), Station 8, CH-1015 Lausanne, Switzerland}
\email{stefano.filipazzi@epfl.ch}

\author[R. Svaldi]{Roberto Svaldi}
\address{Universit\`a degli Studi di Milano, Dipartimento di Matematica ``F. Enriques'', Via Saldini 50, 20133 Milano Mi, Italy}
\address{EPFL, SB MATH-GE, Station 8, CH-1015 Lausanne, Switzerland}
\email{roberto.svaldi@unimi.it}

\subjclass[2020]{Primary 14E30.}
\keywords{Non-klt locus, generalized pairs, dual complex.}
\thanks{
Part of this work was completed during a visit of RS to the University of Utah. 
RS would like to thank the University of Utah for the hospitality and the nice working environment, and Christopher Hacon for funding his visit.
SF acknowledges support from a Graduate Research Fellowship awarded by the University of Utah. 
He was also partially supported by NSF research grants no: DMS-1801851, DMS-195252 and by a grant from the Simons Foundation; Award Number: 256202. 
RS was partially supported by Churchill College, Cambridge,  by the European Union's Horizon 2020 research and innovation programme under the Marie Sk\l{}odowska-Curie grant agreement No. 842071, and by the {\it Programma per giovani ricercatori ``Rita Levi Montalcini''}.
}

\begin{abstract}
Let $(X,B)$ be a pair, and let $f \colon X \rar S$ be a contraction with $-(\K X. + B)$ nef over $S$.
A conjecture, known as the Shokurov--Koll\'ar connectedness principle, predicts that $f \sups -1. (s) \cap \Nklt(X,B)$ has at most two connected components, where $s \in S$ is an arbitrary schematic point and $\Nklt(X,B)$ denotes the non-klt locus of $(X,B)$.
In this work, we prove this conjecture, characterizing those cases in which $\Nklt(X,B)$ fails to be connected, and we extend these same results also to the category of generalized pairs.
Finally, we apply these results and the techniques to the study of the dual complex for generalized log Calabi--Yau pairs, generalizing results of Koll\'ar--Xu and Nakamura.
\end{abstract}

\maketitle

\setcounter{tocdepth}{1}
\tableofcontents

\section{Introduction}

In birational geometry, one of the fundamental and most studied problems is the classification of singularities.
The study of singularities can be carried out either from a local viewpoint, that is, considering a germ of a normal singularity $x \in X$, or from a global one, considering a normal proper variety $Y$.
More often than not, though, we are led to consider more general frameworks: namely, in the local case we consider germs of a normal singularity $x\in X$ and an effective Weil divisor $B$ with coefficients in $[0,1]$, while in the global one, we consider log pairs $(Y, B)$, see \S~\ref{sect.gen.pairs.sings} for the precise definitions. 
The importance of this generalization is evident if, for example, one considers the Riemann--Hurwitz formula for a proper finite map, cf.~\cite{KM98}*{Prop. 5.20}, or when attempting to extend the adjunction formula to the non-Gorenstein case, cf.~\cite{Kol92}*{Chapter 16}.

From the point of view of the birational classification of algebraic varieties, and in particular, the Minimal Model Program, log canonical singularities are the broader class of interest.
Roughly speaking, log canonical singularities can be characterized by the requirement that the pullback of a locally generating top-dimensional differential form may only have poles of order at most one along the exceptional divisors of a log resolution, cf.~\S~\ref{sect.gen.pairs.sings} for a more precise definition.
It has been clear since the 1990's, through the work of Nadel and many others, that the locus of maximal singularities, that is, the set of points that are dominated by exceptional valuations along which poles of order one appear when pulling back a locally generating top-dimensional form, features very important geometrical and cohomological properties that can be used, for example, to construct and lift sections of log divisors from lower-dimensional subvarieties: this type of technique is one of the fundamental tools in the Minimal Model Program and birational geometry, at large.
Such locus where the singularities are maximal is called the non-klt locus and, therefore, it is of particular interest to study its structure.

\subsection*{Connectedness of the non-klt locus}
In this work, we prove an optimal and general structure theorem for the non-klt locus of positively curved pairs, which falls in the framework of the following connectedness principle.
\begin{ks.princ}
Let $(X,B)$ be a log canonical pair.
Let $f \colon X \rar S$ be a contraction. 
If $-(\K X.+ B)$ is $f$-nef and $f$-big, then $\Nklt(X,B)$ is connected in a neighborhood of any fiber of $f$.
\end{ks.princ}
The original version of the connectedness principle dates back  to~\cite{Kol92}*{Theorem~17.4}, which generalized a result of Shokurov,~\cite{MR1162635}*{5.7}, who proved the principle for anti-ample log canonical divisors. Many more instances and generalizations of the principle have appeared throughout the years.
Despite its simplicity, this statement has many powerful applications: for example, inversion of adjunction (see~\cite{Kol92}*{Theorem~17.6},~\cite{kawakita}, and~\cite{MR3165017} for a more recent and general version) or the fact that log canonical singularities are Du-Bois (see~\cite{KK10}), or yet again the study of the geometry and boundedness of varieties of Fano-type and complements (see~\cites{KM99, HM07, Bir16a}).
Perhaps more surprisingly, the connectedness principle has also been used to study hyperbolicity questions related to the positivity of log pairs and even foliations (see~\cites{Sva, SS}).

We work with the following setup: we consider log pairs $(X,B)$ together with a contraction $f \colon X \rar S$ such that $-(\K X. + B)$ $f$-nef. 
Defining the auxiliary class $\mathbf M _X \coloneqq -(\K X. + B)$ and considering the Cartier closure $\mathbf M$ of $\mathbf M _X$, we obtain a generalized pair $(X, B, \mathbf M)$.
This new pair has the advantage of being a Calabi--Yau generalized pair, i.e. $K_X+B+\mathbf M_X \sim_\mathbb{Q} 0$, a condition that is maintained when taking a birational contraction of the space $X$;
moreover, passing to the framework of generalized pairs does not alter the non-klt locus.
Using this data, we are reduced to studying connectedness properties of the non-klt locus of generalized pairs of Calabi--Yau type.
In this framework, we are able to fully and explicitly describe the extent to which failure of the connectedness of $f \sups -1.(s) \cap \Nklt(X,B)$ may be realized.

\begin{theorem} 
\label{main theorem}
Let $(X,B,\mathbf{M})/S$ be a generalized pair, and $f \colon X \rar S$ be a projective morphism such that $\K X. + B + \mathbf{M}_X \sim \subs \qq,f. 0$. 
Fix $s \in S$ and assume that $f \sups -1. (s)$ is connected but $f \sups -1. (s) \cap \Nklt(X,B,\mathbf{M})$ is disconnected (as $k(s)$-schemes). Then,
$f \sups -1. (s) \cap \Nklt(X,B,\mathbf M )$ has exactly two connected components. 
Moreover, taking a dlt model $\overline f \colon (\overline X, \overline B , \mathbf M) \rar S$ of $(X, B, \mathbf{M})$, then
\begin{itemize}
    \item [(1)] $(\overline X,\overline B,\mathbf{M})/S$ is generalized plt in a neighborhood of $\overline{f} \sups -1.(s)$; and
    \item[(2)] there are an \'etale morphism $(s' \in S') \rar (s \in S)$ and a projective morphism $T' \rar S'$ such that $k(s)=k(s')$ and  $(\overline X, \overline B,\mathbf M ) \times_S S'$ is birational to a standard $\pr 1.$-link over $T'$.
\end{itemize}
\end{theorem}

After the completion of this work, we learnt that Birkar has also obtained a similar statement to Theorem~\ref{main theorem} using different techniques,~\cite{B20}.
\\
In dimension 2, Shokurov proved that Theorem~\ref{main theorem} for log Calabi--Yau pairs, i.e., when $\K X. + B$ is numerically trivial,~\cite{Kol92}*{Proposition~12.3.1}; Prokhorov then settled the theorem in general for log canonical of dimension 2,~\cite{Pro01}*{\S~3.3}.
In the log Calabi--Yau case, a version of the theorem was proven by Fujino in dimension 3,~\cite{Fuj00}, and later Koll\'ar and Kov\'acs extended it in any dimension,~\cite{KK10}. 
More recently, Hacon and Han,~\cite{HH19}, proved a weaker version of Theorem~\ref{main theorem} conditionally to termination of flips or to the dimension of $X$ being at most 4.
Our approach is rather different than the most recent results of~\cite{HH19}, and our main insight is the adoption of the language of generalized pairs, cf.~\S~\ref{sect.gen.pairs.sings}, together with the establishment of a canonical bundle formula for these pairs, see Theorem~\ref{generalized canonical bundle formula}.

The statement of Theorem~\ref{main theorem} is sharp, in that none of its hypotheses may be weakened.
We remark that we do not assume any hypothesis on the singularities of the generalized pair $(X, B, \mathbf M)$.
Properties (1)-(2) in the above statement imply that, when $\Nklt(X,B,\mathbf M )$ is disconnected in a neighborhood of $f^{-1}(s)$, then $(X, B, \mathbf M)$ is actually generalized log canonical and, \'etale locally around $s$, the pair admits exactly two disjoint lc centers that coincide with the two connected components of $\Nklt(X,B,\mathbf M )$ in a neighborhood of $f^{-1}(s)$.
Furthermore, easy examples show that, in order to conclude that properties (1)-(2) hold in the statement of Theorem~\ref{main theorem}, we may be forced to pass to a dlt model of $(X, B, \mathbf M)$, see~\S~\ref{gen.dlt.mod.ssect} for the definition and existence of dlt models.

\begin{example}
Let us consider $X \coloneqq \mathbb{P}^2$ (while $S$ will be a point in this example and $M=0$) with homogeneous coordinates, $X_0, X_1, X_2$ and let $B \coloneqq \frac 2 3 L_0 + \frac 2 3 L_1 + \frac 2 3 D + L_2$, where $L_i \coloneqq \{X_i=0\}$ and $D \coloneqq \{X_1+X_2=0\}$.
Then $(\mathbb P^2, B)$ is log canonical, but not dlt, $K_{\mathbb P^2} +B \sim_\mathbb Q 0$, and its lc centers are the point $[0:0:1]$ and the line $\{X_2=0\}$. 
Hence, $\Nklt(X, B)$ is disconnected, but it is not divisorial -- which shows that $(X, B)$ is not dlt.
To obtain a dlt model $(\overline X, \overline B)$ of $(X, B)$, it suffices to blow up the point $[0:0:1]$. 
Thus, in this case, $\overline X=F_1$ and $\overline{B} = \{0\} \times \mathbb P^1 + \{\infty\} \times \mathbb P^1 + \frac 2 3 (\mathbb P^1 \times \{0\} + \mathbb P^1 \times \{1\} + \mathbb P^1 \times \{\infty\})$, which immediately shows how the projection of $X'$ to the second factor endows the pair $(\overline X, \overline B)$ with a $\mathbb{P}^1$-link structure over $\mathbb{P}^1$.  
\end{example}

The notion of standard $\pr 1.$-link mentioned in (2) of the above theorem is an adaptation to the framework of generalized pairs of the following fundamental example.
It is recalled and defined precisely in the context of generalized pairs in \S~\ref{subsect P1 link}.
\begin{example}
\label{example:p1.link.intro}
Let $(T,\Delta)$ be a klt pair.
Then, the pair $(\pr 1. \times T,\lbrace 0 \rbrace \times T + \lbrace \infty \rbrace \times T + \pr 1. \times \Delta)$ together with the morphism $f \colon \pr 1. \times T \rar T$ is a standard $\pr 1.$-link.
Notice that $f \sups -1. (t) \cap \Nklt(\pr 1. \times T,\lbrace 0 \rbrace \times T + \lbrace \infty \rbrace \times T + \pr 1. \times \Delta)$ has two connected components for every $t \in T$, each one corresponding to one of the two distinguished sections of $f$.
\end{example}

\subsection*{$\mathbb{P}^1$-links} The content of Theorem~\ref{main theorem} can be used inductively to study the combinatorics of the log canonical centers of a dlt pair $(X,B)$ with $-(\K X. + B)$ $f$-nef for some contraction $f \colon X \rar S$.
More precisely, we obtain the following statement, generalizing the content of~\cite{Kol13}*{Theorem~4.40}.

\begin{theorem} 
\label{Theoremp1 link}
Let $(X,B, \mathbf{M})/S$ be a generalized dlt pair, and $f \colon X \rar S$ be a projective morphism such that $\K X. + B + \mathbf M _X \sim \subs \qq,f. 0$.
Fix $s \in S$ such that $f \sups -1. (s)$ is connected.
Let $Z \subset X$ be minimal (with respect to inclusion) among the generalized log canonical centers such that $s \in f (Z)$.
Let $W$ be a generalized log canonical center such that $s \in f (W)$.
Then, there exists a generalized log canonical center $Z_W \subset W$ such that $Z$ and $Z_W$ are $\pr 1.$-linked and $s \in f(Z_W)$.
In particular, all the minimal (with respect to inclusion) generalized log canonical centers $Z_i \subset Z$ such that $s \in f (Z_i)$ are $\pr 1.$-linked to each other.
\end{theorem}

In~\cite{HH19}, Hacon and Han proved a similar statement for log pairs, conditionally to termination of flips or to the condition $\dim (X) \leq 4$.

We illustrate the idea behind Theorem~\ref{Theoremp1 link} with an example.
\begin{example}
Consider the pair $(X, B)$, where $X=\pr 1 . \times \pr 1.$ and $B \coloneqq B_1 + B_2 + B_3 + B_4$ is a toric invariant divisor whose irreducible components are the $B_i$.
Then, the pair $(X,B)$ is dlt with $\K X. + B \sim 0$ and Theorem~\ref{Theoremp1 link} applies.
The minimal log canonical centers of the pair $(X,B)$ are the four toric invariant points given by the intersections of the $B_i$.
For every $i= 1, \ldots, 4$, $(B_i,(B-B_i)| \subs D_i.) = (B_i, p_{i, 0} + p_{i, 1}) \simeq (\pr 1.,\lbrace 0 \rbrace + \lbrace \infty \rbrace)$.
For any $i$, we say that the points $p_{i, 0}, p_{i, 1}$ are directly $\pr 1.$-linked, as they lie on the same lc center.
Not all of the $p_{i, j}$, $i=1, \dots, 4$, $j = 0, 1$ are directly $\pr 1.$-linked as we can choose a pair of them that do not lie on the same curve $B_i$.
Nonetheless, the property of being $\pr 1.$-linked is an equivalence relation and so we can partition the set of four points into the orbits of this equivalence relation. 
In the case of $\pr 1 . \times \pr 1.$, and more generally for the case of toric pairs, it is immediate to see that the conclusion of Theorem~\ref{Theoremp1 link} holds at once: namely, all $p_{i, j}$ are $\pr 1.$-linked.
\end{example}

\subsection*{Dual complex for generalized log Calabi--Yau pairs} 

The dual complex of singularities $\mathcal{DMR}(X,B)$ of a log canonical pair $(X,B)$ is a PL-homeomorphism equivalence class of CW-complexes encoding combinatorial information about the strata of $\Nklt(X,B)$.
Given a log resolution $\pi \colon Y \rar X$ of $(X,B)$, it is possible to construct a CW-complex $\mathcal{D}(B_Y \sups =1.)$ whose cells are in correspondence with the intersections of the irreducible components of the simple normal crossing variety $B_Y \sups =1.$ containing all prime divisors of $Y$ along which $K_X+B$ has poles of order one, see \S~\ref{dual.compl.gen.ssect} for a rigorous definition.
By work of de Fernex, Koll\'ar, and Xu,~\cite{dFKX}, the PL-homeomorphism type of $\mathcal{D}(B_Y \sups =1.)$ is independent of the choice of log resolution $\pi \colon Y \rar X$.
In~\cite{KX16}, Koll\'ar and Xu studied the dual complex of log canonical pairs $(X,B)$ with $\K X. + B \sim_\qq 0$, proving that the PL-homeomorphism class $\mathcal{DMR}(X,B)$ of the dual complex of $(X, B)$ admits as a representative an equidimensional complex and it satisfies
\begin{align*}
H^i(\mathcal{DMR}(X,B),\qq)=0, \ \text{for} \ 0 < i < \dim(\mathcal{DMR}(X,B)).
\end{align*}
Furthermore, they described sufficient conditions for the contractibility of $\mathcal{DMR}(X,B)$.

In this paper, we study the dual complex and its topological and cohomological features for log canonical pairs $(X,B)$ with $-(\K X. + B)$ nef, once again, by translating this problem into the analogous one for log Calabi--Yau generalized pairs with log canonical singularities.
Theorems~\ref{main theorem} and~\ref{Theoremp1 link} provide us with powerful tools to extend the results of~\cite{KX16} to this much wider context.
\begin{theorem}
\label{dual.main.thm}
Let $(X, B, \mathbf{M})$ be a generalized pair with log canonical singularities. 
Assume that $K_X+B+M \sim_{\mathbb{Q}} 0$.
Then the dual complex $\mathcal{DMR}(X,B, \mathbf M )$ is an equidimensional pseudomanifold (with boundary). 
Moreover, exactly one of the following condition holds:\begin{enumerate}
    \item $\mathcal{DMR}(X,B, \mathbf M )$ is disconnected and it only contains two points;
    \item $\mathcal{DMR}(X,B, \mathbf M)$ is connected and collapsible to a point;
    \item 
    $\mathcal{DMR}(X,B, \mathbf M)$ is connected, non-collapsible, and 
    \begin{align*}
    H^i(\mathcal{DMR}(X,B, \mathbf M), \mathbb Q)=0, 
    \ \text{for} 
    \ 0 < i < \dim \mathcal{DMR}(X,B, \mathbf M).
    \end{align*}
    \end{enumerate}
\end{theorem}

In order to prove Theorem~\ref{dual.main.thm} we show that the computation of the dual complex of a generalized log canonical pair can be reduced to the classical case of log pairs, in the non-collapsible case.
Under this hypothesis, we show that in general the PL-homeomorphism class $\mathcal{DMR}(X,B, \mathbf M)$ admits as a representative a finite quotient of the dual complex obtained by adjunction to a general fibre of the morphism $\tilde{q} \colon \tilde{X} \rar Z$ constructed in Corollary~\ref{corollary reduce to pairs} -- here, $\tilde X$ is a birational model of $X$ crepant for the generalized pair $(X, B, \mathbf M)$. 
Moreover, using that $\mathcal{DMR}(X,B, \mathbf M)$ is non-collapsible, we show that, upon restricting to the general fibre $F$ of $\tilde q$, the moduli part $\mathbf M$ becomes $0$ on $F$ and we can invoke the results of~\cite{KX16}. 

Recent work of Nakamura,~\cite{Nak19}, extends the construction of dual complex also to the category of log pairs with singularities worse than log canonical.
Nakamura, ~\cite{Nak19}*{Theorem~1.1}, showed that the dual complex of a log pair $(X, B)$ is collapsible provided that $-(\K X. + B)$ is nef and big, without any assumption on the singularities of $(X,B)$.
Using the techniques of~\cite{Nak19} together with the ideas used in the proof of Theorem~\ref{dual.main.thm}, we obtain the following theorem generalizing Nakamura's result.

\begin{theorem} 
\label{thm non lc}
Let $(X, B, \mathbf{M})$ be a generalized pair with singularities worse than log canonical.
Assume that $K_X+B + \mathbf M _X \sim_{\mathbb{Q}} 0$.
Then the dual complex $\mathcal{DMR}(X,B, \mathbf M)$ is collapsible.
\end{theorem}

\subsection*{Strategy of proof.}
The proof of Theorem~\ref{main theorem} consists of several reductions.
For simplicity, we sketch them under the assumption that $(X,B, \mathbf M)$ is generalized log canonical.
The general case is treated similarly but requires some heavy notation.

In \S~\ref{sect.bir.case}, we show that the number of connected components of $\Nklt(X,B, \mathbf M)$ are preserved under birational morphisms.
In particular, this allows us to run certain MMPs while preserving the assumptions of the statement.
For instance, we can assume that $X$ is $\qq$-factorial and that $(X,B, \bM.)$ is generalized dlt.
Thus, we have $\Nklt(X,B, \mathbf M)=B \sups =1.$.

Then, the core of the proof consists in showing that some component of $B \sups =1.$ dominates $S$.
This is done in Proposition \ref{key prop}.
We illustrate the main idea under the assumption that $S$ is a curve and that $s \in S$ is a closed point.
Assume that no component of $B \sups =1.$ dominates $S$.
Then, we can assume that all the connected components of $B \sups =1.$ map to $s$.
For notation's sake, assume that they are two, denoted by $\Delta_1$ and $\Delta_2$.
Let $\tilde X \rar X$ be a log resolution of $(X,B)$, and denote by $(\tilde X,B_{\tilde X}, \mathbf M)$ the trace of $(X,B, \mathbf M)$ on $\tilde X$, see the line before Definition~\ref{def.sings} for the definition of trace of a generalized pair.
By the results of \S~\ref{sect.bir.case}, $B_{\tilde X} \sups =1.$ has two distinct connected components $\Gamma_1$ and $\Gamma_2$, each one mapping to the corresponding $\Delta_i$.
Notice that $\Gamma_1 \cup \Gamma_2 \subset \tilde X _s$, where $\tilde X _s$ denotes the fiber over $s$.
By our assumption, the fiber $\tilde X _s$ contains other irreducible components that connect $\Gamma_1$ and $\Gamma_2$.
Denote the support of these residual components by $E$.
In order to get a contradiction, we would like to contract $E$.
By ideas similar to ones contained in~\cite{FG14}*{proof of Theorem 1.1}, we can run a suitable MMP over $S$ that contracts $E$, while preserving at least one irreducible component of each $\Gamma_i$.
Thus, we reach a model $\overline{X} \rar S$ where the fiber over $s$ consists of the strict transforms of $\Gamma_1$ and $\Gamma_2$, which are now connected.
This contradicts suitable results in \S~\ref{sect.bir.case}, which guarantee that the MMP we just run cannot connect different connected components of the non-klt locus.

When $\dim (S) \geq 2$, this step is more delicate.
Indeed, the components $D_1$ and $D_2$ can dominate different subvarieties $T_1$ and $T_2$, each one containing $s$.
To control this phenomenon, we make use of the generalized canonical bundle formula (see \S~\ref{section cbf}).
For simplicity, assume that $X \rar S$ is a contraction.
Then, the generalized canonical bundle formula allows us to define a generalized pair $(S,B_S,\mathbf N)/S$.
In this way, we can regard $T_1$ and $T_2$ as non-klt centers of $(S,B_S, \mathbf N)$.
Then, the results of \S~\ref{sect.bir.case} guarantee that a generalized dlt model $\hat S$ of $(S,B_S, \mathbf N)$ preserves the connected components of $\Nklt(S,B_S, \mathbf N)$.
To conclude, we show that a similar argument as in the case $\dim(S)=1$ works over $\hat S$.

Once it is established that at least one connected component of $B \sups =1.$ dominates $S$, we can conclude the proof of Theorem~\ref{main theorem}.
Let $D_1,\ldots,D_k$ be the connected components of $B \sups =1.$.
Up to relabelling, we can assume that $D_1$ dominates $S$.
Then, the divisor $\K X. + B \sups < 1. +M$ is not pseduo-effective over $S$.
Therefore, we may run a $(\K X. + B \sups <1. + \mathbf M _X)$-MMP over $S$, which terminates with a Mori fiber space $g \colon \tilde X \rar Z$.
Denote by $(\tilde X,\tilde B , \mathbf M)$ the trace of $(X,B, \mathbf M)$ on $\tilde X$.
Since this is a $(-B) \sups =1.$-MMP, it follows from the results in \S~\ref{sect.bir.case} that
the connected components of $B \sups =1.$ are in one-to-one correspondence with the connected components of $\tilde B \sups =1.$.
Since $\tilde D _1$ is $g$-ample and $\tilde D_1 \cap \tilde D_i = \emptyset$ for $i \geq 2$, it follows that $g$ has relative dimension 1 with general fiber $\pr 1.$.
This forces every $D_i$ to be horizontal over $Z$.
Since $\K \tilde X. + \tilde B + \mathbf M \subs \tilde X . \sim_\qq 0/Z$, it follows that $k =2$.
By direct inspection, we conclude that, when $k=2$, $(\tilde X, \tilde B , \mathbf M) \rar Z$ is a standard $\pr 1.$-link up to an \'etale base change.

The proof of Theorem~\ref{Theoremp1 link} follows the proof of~\cite{Kol13}*{Theorem~4.40}, which deals with the case of log Calabi--Yau pairs.
The argument is by induction on the dimension, and relies on Theorem~\ref{main theorem} to reduce to the case when $\Nklt(X,B, \mathbf M)$ is connected along the fibers of the morphism.

We conclude by sketching the proof of Theorem~\ref{dual.main.thm}.
Theorem~\ref{main theorem} gives an explicit description of the case when $\mathcal{DMR}(X,B, \mathbf M)$ is not connected.
Therefore, we can assume that $\mathcal{DMR}(X,B, \mathbf M)$ is connected.
Then, as a direct combinatorial consequence of Theorem~\ref{Theoremp1 link}, we obtain that the dual complex $\mathcal{DMR}(X,B,\bM.)$ is equidimensional at each point.
In \S~\ref{sect.dual.cplx} we follow ideas of~\cite{KX16} and show that, under certain assumptions, we can construct a morphism $f \colon X \rar Z$ such that $\dim(Z) \geq 1$ and $\mathcal{DMR}(X,B, \mathbf M) \simeq \mathcal{DMR}(X_z,B_z, \mathbf M | \subs X_z.)/G$, where $X_z$ is a general fiber of $f$ and $G$ is a finite group.
Then, in this situation, we can argue by induction on the dimension.
In the leftover cases, we have that $B \sups =1.$ fully supports a big and semi-ample divisor.
This condition on $\Nklt(X,B, \mathbf M)$ allows us to apply a version of Kawamata--Viehweg vanishing for generalized pairs, and conclude the proof of Theorem~\ref{dual.main.thm}.

\subsection*{Acknowledgements} The authors wish to thank  
Tommaso de Fernex, Gabriele Di Cerbo, Christopher Hacon, Mirko Mauri, James M\textsuperscript{c}Kernan, Joe Waldron
for helpful discussions and encouragements.
The authors wish to thank Christopher Hacon and Mirko Mauri also for reading a first draft of this work.
We wish to thank the anonymous referee for useful comments and suggestions that helped the authors improve the clarity of this work.

\section{Preliminaries}\label{prelim.sect}
In this section, we set our notation and collect some definitions and preliminary results that will be useful in the paper.

\subsection{Terminology and conventions}
\label{term.subs}
Throughout this paper, we will work over an algebraically closed field of characteristic 0.
For anything not explicitly addressed in this subsection, we direct the reader to the terminology and the conventions of~\cite{KM98} and~\cite{Kol13}.
\newline
A \emph{contraction} is a projective morphism $f\colon X \rar Z$ of quasi-projective varieties with $f_* \O X. = \O Z.$. 
If $X$ is normal, then so is $Z$ and the fibers of $f$ are connected.

Let $\mathbb{K}$ denote $\zz$, $\qq$, or $\rr$. We say that $D$ is a \emph{$\mathbb{K}$-divisor} on a variety $X$ if we can write $D = \sum \subs i=1. ^n d_i P_i$ where $d_i \in \mathbb{K}$, $n \in \nn$ and $P_i$ is a prime Weil divisor on $X$ for all $i=1, \ldots, n$. 
We say that $D$ is $\mathbb{K}$-Cartier if it can be written as a $\mathbb{K}$-linear combination of $\zz$-divisors that are Cartier.
The \textit{support} of a $\mathbb{K}$-divisor $D=\sum_{i=1}^n d_iP_i$ is the union of the prime divisors appearing in the formal sum $\mathrm{Supp}(D)= \sum_{i=1}^n P_i$.
\newline
In all of the above, if $\mathbb{K}= \zz$, we will systematically drop it from the notation.

Given a prime divisor $P$ in the support of $D$, we will denote by $\mu_P (D)$ the coefficient of $P$ in $D$.
Given a divisor $D = \sum \subs i=1. ^n \mu_{P_i}(D) P_i$, we define its {\it round down} $\lfloor D \rfloor  \coloneqq \sum \subs i=1.^n \lfloor \mu_{P_i}(D) \rfloor P_i$.
The {\it round up} $\lceil D \rceil$ of $D$ is defined analogously.
The {\it fractional part} $\{D\}$ of $D$ is defined as $\{D\} \coloneqq D - \lfloor D \rfloor$.
Let $D_1 = \sum \subs i=1. ^n \mu_{P_i}(D_1) P_i$ and $D_2= \sum \subs i=1. ^n \mu_{P_i}(D_2) P_i$. We define $D_1 \wedge D_2 \coloneqq \sum \subs i=1. ^n \min \lbrace \mu_{P_i}(D_1),\mu_{P_i}(D_2) \rbrace P_i$.
Similarly, we set $D_1 \vee D_2 \coloneqq \sum \subs i=1. ^n \max \lbrace \mu_{P_i}(D_1),\mu_{P_i}(D_2) \rbrace P_i$.
For a divisor $D$, we set $D \sups \geq 0. \coloneqq D \vee 0$, where $0$ denotes the zero divisor.
Similarly, we define $D \sups \leq 0. \coloneqq -(D \wedge 0)$. In particular, we have $D=D \sups \geq 0. - D \sups \leq 0.$.

Given a divisor $D = \sum \mu_{P_i}(D) P_i$ on a normal variety $X$, and a morphism $\pi \colon X \to Z$, we define
\begin{align*}
D^v \coloneqq \sum_{\pi(P_i) \subsetneqq Z} \mu_{P_i}(D) P_i, \
D^h \coloneqq \sum_{\pi(P_i) = Z} \mu_{P_i}(D) P_i.
\end{align*}

\subsection{B-divisors}
\label{b-div.subs}
Let $\mathbb{K}$ denote $\zz$, $\qq$, or $\rr$.
Given a normal variety $X$, a {\it $\mathbb{K}$-b-divisor} $\mathbf{D}$ is a (possibly infinite) sum of geometric valuations $V_i$ of $k(X)$ with coefficients in $\mathbb{K}$,
\begin{align*}
 \mathbf{D}= \sum_{i \in I} b_i V_i, \; b_i \in \mathbb{K},
\end{align*}
such that for every normal variety $X'$ birational to $X$, only a finite number of the $V_i$ can be realized by divisors on $X'$. 
The {\it trace} $\mathbf{D}_{X'}$ of $\mathbf{D}$ on $X'$ is defined as 
\begin{align*}
\mathbf{D}_{X'} \coloneqq \sum_{
\{i \in I \; | \; c_{X'}(V_i)= D_i, \;  
\codim_{X'} D_i=1\}} b_i D_i,
\end{align*}
where $c_{X'}(V_i)$ denotes the center of the valuation on $X'$.
\newline
Given a $\mathbb{K}$-b-divisor $\mathbf{D}$ over $X$, we say that $\mathbf{D}$ is a {\it $\mathbb{K}$-b-Cartier} $\mathbb{K}$-b-divisor if there exists a birational model $X'$ of $X$ such that $\mathbf{D}_{X'}$ is $\mathbb K$-Cartier on $X'$ and for any model $r \colon X''  \rar X', \; \mathbf{D}_{X''} = r^\ast \mathbf{D}_{X'}$.
When that is the case, we will say that $\mathbf{D}$ descends to $X'$ and write $\mathbf{D}= \overline{\mathbf{D}_{X'}}$.
We say that $\mathbf{D}$ is {\it b-effective}, if $\mathbf{D}_{X'}$ is effective for any model $X'$.
We say that $\mathbf{D}$ is {\it b-nef}, if it is $\mathbb{K}$-b-Cartier and, moreover, there exists a model $X'$ of $X$ such that $\mathbf{D}= \overline{\mathbf{D}_{X'}}$ and $\mathbf{D}_{X'}$ is nef on $X'$. 
The notion of b-nef b-divisor can be extended analogously to the relative case.
\newline
In all of the above, if $\mathbb{K}= \zz$, we will systematically drop it from the notation.

\begin{example}
Let $X$ be a normal variety and denote by $K_X$ the choice of a divisor in the canonical class. 
The \emph{canonical b-divisor} $\mathbf{K}$ extending $K_X$ is defined as follows:
its trace $\mathbf K \subs X.$ on $X$ is $K_X$, while the trace $\mathbf K \subs X'.$ on a birational model $\pi \colon X' \rar X$ is given by $\K X'.$, where the divisor $\K X'.$ in the canonical class of $X'$ is chosen so that $\pi_* \K X'.= \K X.$.
The b-divisor $\mathbf{K}$ is not $\qq$-b-Cartier, as it follows easily by blowing up a smooth point.
\end{example}

\begin{example}
\label{discr.div.ex}
Let $(X, B)$ be a log sub-pair.
The \emph{discrepancy b-divisor} $\mathbf{A}(X,B)$ is defined as follows: on a birational model $\pi \colon X' \rar X$, its trace $\mathbf{A}(X,B)_{X'}$ is given by the identity  $\mathbf{A}(X,B)_{X'}\coloneqq \K X'. - \pi^\ast (\K X. + B)$.
Then, the b-divisor $\mathbf{A}^\ast (X,B)$ is defined taking its trace $\mathbf{A}^\ast (X,B)_{X'}$ on $X'$ to be $\sum \subs a_i > -1. a_i D_i$, where $\mathbf{A}(X,B)_{X'} = \sum_i a_i D_i$.
\end{example}

Given a morphism of normal varieties 
$\pi \colon X \to T$, and a
$\mathbb K$-b-Cartier 
$\mathbb K$-b-divisor
$\mathbf M$ 
(resp. 
$\mathbf N$)
on $X$ 
(resp. $T$), 
we will write 
$\mathbf M \sim_{\mathbb K} \pi^* \mathbf N$ 
to indicate that there exists a birational model 
$\pi' \colon X' \rightarrow T'$ of $\pi \colon X \rar T$ such that $\bM X'. \sim_{\mathbb{Q}}(\pi')^*\mathbf{N}_{T'}$, $\bM.= \overline{\bM X'.}$, and $\mathbf{N}=\overline{\mathbf{N}_{T'}}$.

\subsection{Generalized pairs and singularities}
\label{sect.gen.pairs.sings}

We recall the definition of generalized pairs, first introduced in~\cite{BZ16}.
This is a generalization of the classic setting of log pairs.

\begin{definition}
A {\em generalized sub-pair} $(X,B, \mathbf{M})/Z$ over $Z$  is the datum of:
\begin{itemize}
\item a normal variety $X  \rar Z$ projective over $Z$;
\item an $\mathbb R$-Weil divisor $B$ on $X$;
\item a $b$-$\mathbb R$-Cartier b-divisor $\mathbf{M}$ over $X$ which descends to an $\mathbb R$-Cartier divisor $\mathbf{M}_{X'}$ on some birational model $X' \rightarrow X$, and $\mathbf{M}_{X'}$ is relatively nef over $Z$.
\end{itemize}
Moreover, we require that $K_X +B+ \mathbf{M}_X$ is $\mathbb R$-Cartier.
If $B$ is effective, we say that $(X,B,\bM.)/Z$ is a generalized pair.
\end{definition}

In the above definition, we can always replace $X'$ with a higher birational model $X''$ and $\mathbf{M}_{X'}$ with $\mathbf{M}_{X''}$ without changing the generalized pair.
Whenever $\mathbf{M}_{X''}$ descends to $X''$, then the data of the rational map $X'' \drar X$, $B$, and $\mathbf{M}_{X''}$ encode all the information of the generalized pair.

When the setup is clear, we will denote the generalized sub-pair $(X, B, \mathbf{M})/Z$ by $(X,B+\mathbf{M}_X)/Z$ and we will say that $(X,B+\mathbf{M}_X)$ is a generalized pair over $Z$ with datum $\mathbf M$; for the sake of simplifying the notation, we will often replace $\mathbf{M}_X$ by $M$, and write $(X, B+M)$.
When $Z = {\rm Spec}(\mathbb{C})$, we will simply write $(X, B, \mathbf{M})$ and $(X, B +M)$.

Let $(X,B, \mathbf{M})/Z$ be a generalized sub-pair and $\pi \colon Y \rar X$ a projective birational morphism. 
Then, we may write
\begin{align*}
\K Y.+B_Y + \mathbf{M}_{Y}=\pi^\ast (K_X+B+M).
\end{align*}
Given a prime divisor $E$ on $Y$, we define the {\em generalized log discrepancy} of $E$ with respect to $(X,B+M)/Z$  to be $a_E(X,B+M)\coloneqq 1-\mu_{E}(B_Y)$.
In this setup, we call $(Y,B_Y,\bM.)/Z$ the \emph{trace} of $(X,B,\bM.)/Z$ on $Y$.

\begin{definition}
\label{def.sings}
Let $(X,B, \mathbf{M})/Z$ be a generalized sub-pair.
If $a_E(X,B+M) \geq 0$ for all divisors $E$ over $X$, we say that $(X,B+M)$ is \emph{generalized sub-log canonical}.
Similarly, if $a_E(X,B+M) > 0$ for all divisors $E$ over $X$ and $\lfloor B \rfloor \leq 0$, we say that $(X,B+M)$ is \emph{generalized sub-klt}.
When $B \geq 0$, we say that $(X,B+M)$ is \emph{generalized log canonical} or \emph{generalized klt}, respectively.
\end{definition}

\begin{remark} 
\label{remark valuations nef part}
Let $(X,B, \mathbf{M})/Z$ be a generalized sub-pair and let $N$ be an $\mathbb R$-Cartier divisor on $X$ such that $\mathbf{M}_{X'}+f^\ast N$ is nef over $Z$, where $f \colon X' \rar X$ is a birational model of $X$ on which $\mathbf{M}_{X'}$ descends.
Then, $(X, B, \mathbf{M}+\overline{N})/Z$ is a generalized sub-pair with datum $\mathbf{M}+ \bar{N}$ and $a_E(X,B+M)=a_E(X,B+(M+N))$ for every divisor $E$ over $X$.
\end{remark}

\begin{example}
\label{gen.discr.div.example}
Let $(X,B,\mathbf{M})$ be a generalized sub-pair.
The \emph{generalized discrepancy b-divisor} $\mathbf{A}(X,B,\mathbf{M})$ is defined as follows: 
on a birational model $\pi \colon X' \rar X$, its trace $\mathbf{A}(X,B,\mathbf{M})_{X'}$ is given by the identity  $\mathbf{A}(X,B,\mathbf{M})_{X'}\coloneqq \K X'. + \mathbf{M} \subs X'.- \pi^\ast (\K X. + B + \mathbf{M}_X)$.
Then, the b-divisor $\mathbf{A}^\ast (X,B,\mathbf{M})$ is defined taking its trace $\mathbf{A}^\ast (X,B,\mathbf{M})_{X'}$ on $X'$ to be $\mathbf{A}(X,B,\mathbf{M})^\ast_{X'} \coloneqq \sum \subs a_i > -1. a_i D_i$, where $\mathbf{A}(X,B,\mathbf{M})_{X'} = \sum_i a_i D_i$.
Notice that, if $\mathbf{M}$ descends to $X'$, we have identities of b-divisors $\mathbf{A}(X,B,\mathbf{M})=\mathbf{A}(X',B')$, and $\mathbf{A}^\ast (X,B,\mathbf{M})=\mathbf{A}^\ast (X',B')$, where $B'=-\mathbf{A}(X,B,\mathbf{M}) \subs X'.$.
\end{example}

\begin{definition}
\label{lc.center.def}
Let $(X,B,\bM.)/Z$ be a generalized sub-pair and let $E$ be a divisor over $X$.
If $a_E(X,B+M) \leq 0$, we say that $E$ is a {\em non-klt place} for the generalized pair, and $c_X(E)\subset X$ is a {\em non-klt center} for the generalized pair. 
The \emph{non-klt locus} $\Nklt(X,B, \mathbf M)$ is defined as the union of all the non-klt centers of $(X,B+M)$. 
If $a_E(X,B+M) = 0$, we say that $E$ is a {\em generalized log canonical place} for $(X,B+M)$, and $c_X(E)$ is a {\em generalized log canonical center} for $(X,B+M)$, provided that $(X, B+M)$ is generalized log canonical in a neighborhood of $c_X(E)$.
\end{definition}

It is possible to extend the classical results on adjunction for lc pairs, cf.~\cite{Kol92}*{\S~16}, to the context of generalized pairs.

Let $(X,B,\mathbf{M})/Z$ be a generalized pair.
Let $S$ be an irreducible component of $\lfloor B \rfloor$, and denote by $S^\nu$ its normalization.
Let $f \colon X' \rar X$ be a log resolution of $(X,B)$ where $\mathbf{M}$ descends.
Denote by $g  \colon  S' \rar S^\nu$ the induced morphism, where $S'$ represents the strict transform of $S$ on $X'$.
Then, we can write
\begin{align*}
\K X'. + B' + \mathbf{M} \subs X'. = f^\ast  (\K X. + B + \mathbf{M}_X).
\end{align*}
Up to replacing $\mathbf{M}$ in its $\mathbb{K}$-linear equivalence class, we can assume that $S'$ does not appear in the support of $\mathbf{M} \subs X'.$.
Then, we set
\begin{align*}
\K S'. + B_{S'} + \mathbf{N}_{S'} \coloneqq (\K X'. + B' + \mathbf{M} \subs X'.)|_{S'},
\end{align*}
where $B_{S'} \coloneqq (B'-S')|_{S'}$, and $\mathbf{N} \subs S'. \coloneqq \mathbf{M} \subs X'.|_{S'}$. 
Define $B_{S^\nu} \coloneqq g_* B_{S'}$, and $\mathbf{N}_{S^\nu} \coloneqq g_* \mathbf{N}_{S'}$.
By construction, we get
\begin{align*}
\K S^\nu. + B_{S^\nu} + \mathbf N _{S^\nu} = (\K X. + B + \mathbf{M}_X)|_{S^\nu}.
\end{align*}
We refer to such operation as {\it generalized divisorial adjunction}.
By construction, the generalized pair $(S^\nu,B \subs S^\nu.,\mathbf{N})$ is a generalized pair over $Z$.
We may write $\mathbf M | \subs S^\nu.$ for $\mathbf N$ to highlight that indeed $\mathbf N$ comes from the restriction of $\mathbf M$ to $S$.

\subsection{Generalized dlt pairs and dlt models}
\label{gen.dlt.mod.ssect}
In this section, we recall the notion of dlt and plt singularities in the context of generalized pairs and we prove the existence of dlt models.

\begin{definition}
\label{gen.dlt.def}
We say that a generalized pair $(X,B, \mathbf{M})/Z$ is {\em generalized dlt}, if it is generalized log canonical and for the generic point $\eta$ of any generalized log canonical center the following conditions hold:
\begin{itemize}
    \item[(i)] $(X,B)$ is log smooth in a neighborhood of $\eta$; and
    \item[(ii)] $\bM . = \overline{\bM X.}$ over a neighborhood of $\eta$.
\end{itemize}
If, in addition, every connected component of $\lfloor B \rfloor$ is irreducible, we say that $(X,B+M)$ is {\em generalized plt}.
\end{definition}

The following result is a refinement of~\cite{Fil18}*{Theorem~3.2} and proves the existence of generalized dlt models.

\begin{theorem} 
\label{generalized dlt model}
Let $(X,B, \mathbf{M})/Z$ be a generalized pair.  
Then, there exists a $\qq$-factorial model $f^m\colon X^m \rar X$ such that every $f^m$-exceptional divisor has generalized log discrepancy with respect to $(X,B+M)$ at most $0$. 
Let $E^m$ denotes the reduced $f^m$-exceptional divisor.
Then the generalized pair $(X^m,B^m,\mathbf{M})/Z$ is generalized dlt, where $B^m \coloneqq (f^m)^{-1}_* (B \wedge \Supp(B)) + E^m$.
\end{theorem}

\begin{proof}
Let $f\colon X' \rar X$ be a log resolution of $(X,B)$ where $\mathbf M$ descends.
For the sake of simplifying notation, we will use $M'$ to denote $\mathbf{M}_{X'}$.
By Hironaka's theorem, we can assume that $f$ is obtained by blowing up loci of codimension at least two and that there exists an effective $f$-exceptional divisor $C'$ such that $-C'$ is $f$-ample.
We define $B'$ via the identity
\begin{align*}
\K X'. + B' + M' = f^\ast (\K X. + B + M ).
\end{align*}
In view of this definition, we can decompose $B'$ as $B'=f_*^{-1} \lbrace B \rbrace + E^+ + F' - G'$, where 
\begin{itemize}
    \item $E^+$ denotes the (not necessarily $f$-exceptional) divisors with generalized log discrepancy at most $0$ with respect to $(X, B, \mathbf{M})$;
    \item $F'$ the sum of all $f$-exceptional divisors with generalized log discrepancy in $(0, 1]$; and,
    \item $G'$ the sum of all $f$-exceptional divisors with generalized log discrepancy $>1$.
\end{itemize}
We define $E' \coloneqq \Supp (E^+)$.
Letting $H$ be a sufficiently ample divisor on $X$, for all $\epsilon, \nu, \tau \in \mathbb{R}$,
\begin{equation} 
\label{equation with epsilon mu nu}
E' + (1 + \nu)F' + \tau(-C'+f^\ast H)+M' = (1-\epsilon \tau)E' + (1+ \nu)F' + \tau(\epsilon E' - C' + f^\ast H) + M',
\end{equation}
and for any $0< \epsilon \ll \tau$ and $\epsilon \ll 1$, both $\tau(-C'+f^\ast H)+M'$ and $\tau(\epsilon E' - C' + f^\ast H) + M'$ are ample over $X$.
For any such choice of $\epsilon$ and $\tau$, we can choose divisors $H'_{1,\tau} \sim_{\mathbb{R},X} \tau(-C'+f^\ast H)+M'$ and $H'_{2,\tau, \epsilon} \sim_{\mathbb{R},X} \tau(\epsilon E' - C' + f^\ast H) + M'$ such that $B'+H'_{1,\tau}+H'_{2,\tau, \epsilon}$ has simple normal crossing support, and $\lfloor H'_{1,\tau} \rfloor = \lfloor H'_{2,\tau, \epsilon} \rfloor =0$.
\newline
Thus, if $0 < \tau <1 $ and $0 < \nu \ll1$, the pair
\begin{align*}
(X', \Delta'_{\epsilon, \nu, \tau}  \coloneqq  f^{-1}_*\lbrace B \rbrace+(1 - \epsilon \tau)E' + (1 + \nu)F' + H'_{2,\tau, \epsilon})
\end{align*}
is klt.
By~\cite{BCHM}, we can run a $(K_{X'} + \Delta'_{\epsilon, \tau, \nu})$-MMP over $X$ that terminates with a relative $\qq$-factorial minimal model
\begin{align*}
f^m_{\epsilon,\tau,\nu} \colon (X^m_{\epsilon,\tau,\nu}, \Delta^m_{\epsilon,\tau,\nu}) \rar X.
\end{align*}
By~\eqref{equation with epsilon mu nu}, $f^m_{\epsilon,\tau,\nu}$ is also a minimal model over $X$ for the pair
\begin{align*}
(X',\Gamma'_{\tau, \nu} \coloneqq  f^{-1}_*\lbrace B \rbrace+ E' + (1 + \nu)F' + H'_{1,\tau}).
\end{align*}
As the dlt property is preserved under steps of the MMP~\cite{KM98}*{Corollary~3.44}, both $(X^m_{\epsilon,\nu, \tau}, \Delta^m_{\epsilon,\nu, \tau})$ and $(X^m_{\epsilon,\nu, \tau}, \Gamma^m_{\nu, \tau})$ are dlt pairs, where $\Delta^m_{\epsilon,\nu, \tau}$ (respectively, $\Gamma^m_{\nu, \tau}$) is the push-forward of $\Delta'_{\epsilon,\nu, \tau}$ (resp., $\Gamma'_{\nu, \tau}$) on $X^m_{\epsilon, \nu, \tau}$.
Hence, the pair $(X^m_{\epsilon,\nu, \tau},B^m_{\epsilon,\nu, \tau})$ is dlt, where $B^m_{\epsilon,\nu, \tau}$ is the push-forward of $f^{-1}_* \lbrace B \rbrace +E' +F'$ on $X^m_{\epsilon,\nu, \tau}$.
\newline
To simplify the notation, we will denote by $A^m$ the strict transform on $X_{\epsilon,\nu, \tau}^m$ of any divisor $A$ on $X'$.
In particular, we will denote $\mathbf{M}_{X^m_{\epsilon, \nu, \tau}}$ by $M^m$.
Then, we define
\begin{eqnarray*}
&N \coloneqq \K X_{\epsilon,\nu, \tau}^m. + B^m_{\epsilon,\nu, \tau} + \nu F^m + H^m_{1} \sim_{\rr,f^m_{\epsilon,\nu, \tau}} \K X_{\epsilon,\nu, \tau}^m. + \Delta^m_{\epsilon,\nu, \tau},\\
&T \coloneqq \K X_{\epsilon,\nu, \tau}^m. + B^m_{\epsilon,\nu, \tau} + (E^+-E')^m - G^m + M^m \sim_\rr (f^m_{\epsilon,\nu, \tau})^\ast  ( \K X. + B + M).
\end{eqnarray*}
The divisor $N$ is $f^m_{\epsilon,\nu, \tau}$-nef, while $T$ is $f^m_{\epsilon,\nu, \tau}$-trivial, and
\begin{align*}
T - N \sim_{\rr,f^m_{\epsilon,\nu, \tau}} \tau C^m + (E^+ -E')^m - G^m  - \nu F^m  \eqqcolon D,
\end{align*}
so that $-D$ is $f_{\epsilon,\nu, \tau}^m$-nef and $f_{\epsilon,\nu, \tau \ast}^m D \geq 0$. 
Therefore, by the negativity lemma~\cite{KM98}*{Lemma~3.39}, $D$ is effective.
\newline
As $C'$, $E^+-E'$, $F'$ and $G'$ are independent of ${\epsilon,\nu, \tau}$, if we choose $0 < \epsilon \ll \tau \ll \nu \ll 1$, the $(K_{X'} + \Delta'_{\epsilon, \nu, \tau})$-MMP contracts $F'$ and $G'$, as
\begin{align*}
K_{X'} + \Delta'_{\epsilon, \nu, \tau} \sim_{\mathbb{R}, f} G' + \nu F' -\tau C' -(E^+-E').
\end{align*}
Indeed, as the effective divisors $G'$, $F'$ and $E^+-E'$ share no prime components, $\Supp(C') \subset \Supp(E'+F'+G')$, and $\tau \ll \nu$, for every prime divisor $P'$ on $X'$ $\mu \subs P'. (G' + \nu F' -\tau C' -(E^+-E')) > 0$ if and only if $\mu \subs P'. (G' + \nu F' -(E^+-E'))>0$.
We fix once and for all such a choice of the coefficients $\epsilon, \tau, \nu$, and we drop the dependence from ${\epsilon,\nu, \tau}$ in our notation.
\newline
The generalized pair $(X^m,B^m+M^m)/Z$ is generalized log canonical. 
In fact, $f^\ast H-C'$ is ample, as $H$ is assumed to be sufficiently ample; 
picking $0 \leq A' \sim_\qq \tau(f^\ast H - C')$ a general element in its $\mathbb{Q}$-linear equivalence class, so that $(X',f_* \sups -1. \lbrace B \rbrace + E' + (1+\nu)F' + A')$ is dlt by Bertini's theorem, the generalized pair $(X',f_* \sups -1. \lbrace B \rbrace + E' + (1+\nu)F' + A', \mathbf{M})/Z$ is generalized dlt, and each step in the $(K_{X'} + \Delta'_{\epsilon, \nu, \tau})$-MMP leading to $X^m$ is a $(\K X'. + f_* \sups -1. \lbrace B \rbrace + E' + (1+\nu)F' + A'+M')$-negative contraction.
Therefore, the generalized pair $(X^m,B^m+A^m,\mathbf{M})/Z$ is generalized log canonical.
As $X^m$ is $\qq$-factorial, then $(X^m,B^m, \mathbf{M})/Z$ is generalized log canonical.
Thus, $(X^m,B^m+M^m)$ has all the claimed properties, besides the fact that it may not be generalized dlt.
\newline
To conclude, it suffices to substitute $(X^m,B^m+M^m)$ with a generalized dlt model, which exists by~\cite{Bir16a}*{2.13.(2)} as $(X^m,B^m+M^m)$ is log canonical.
Passing to such model only extracts divisors with generalized log discrepancy $0$ with respect to $(X^m,B^m+M^m)$.
This completes the proof of the theorem.
\end{proof}

\begin{definition}
For a generalized pair $(X, B, \mathbf{M})/Z$,
we call the generalized pair constructed in Theorem~\ref{generalized dlt model} and denoted by $(X^m,B^m, \mathbf{M})/Z$ a \emph{generalized dlt model} for $(X,B+M)$.
\end{definition}

Finally, we include here a couple of technical results that will be used in the proof of Theorem~\ref{main technical thm for complexes}.
We first introduce a couple of definitions.

\begin{definition}
\label{def:dlt.away}
Given a generalized pair $(X,B,\bM.)$ and a closed subset $C \subset X$, we say that $(X,B,\bM.)$ is generalized dlt away from $C$ if $(X \setminus C,B,\bM.)$ is generalized dlt.
Here, we restrict $B$ to the open subset $X \setminus C$ as in \cite{Har77}*{Proposition II.6.5}.
Similarly, the preimage of $X \setminus C$ is open in every higher birational model of $X$. Therefore, we can therefore restrict $\bM.$ to $X \setminus C$ by restricting all its traces to the preimages of $X \setminus C$.
\end{definition}

\begin{definition}
\label{fully.supp.def}
Let $F$ be an effective $\mathbb{R}$-divisor on a variety $X$, and let $D$ be an $\mathbb R$-divisor on $X$. We say that $F$ \emph{fully supports} $D$ if $D$ is effective and the support of $D$ coincides with the support of $F$.
Moreover, given a morphism of varieties $f \colon X \to Z$, we say that $F$ {\it fully supports an $f$-ample divisor} if $F$ fully supports an effective $f$-ample $\mathbb{R}$-divisor $H$.
\end{definition}

\begin{lemma} \label{lemma patch}
Let $(X,B,\bM.)$ be a $\qq$-factorial generalized log canonical pair, and let $q \colon X \rar Z$ be a contraction.
Assume that $B \sups =1.$ fully supports a $q$-ample divisor, and that $(X,B,\bM.)$ is generalized dlt away from $B \sups =1.$.
Let $\pi \colon (X',B',\bM.) \rar X$ be a generalized dlt model for $(X,B,\bM.)$.
Then, $(B')\sups =1.$ fully supports a $(q \circ \pi)$-ample divisor.
\end{lemma}

\begin{remark}
In the setup of Lemma~\ref{lemma patch}, the condition on the dlt-ness of the pair away from $B \sups =1.$ together with $\qq$-factoriality of $X$ guarantees that, when passing to a generalized dlt model of $(X,B,\bM.)$, we only extract divisors with center in $B \sups =1.$.
\end{remark}

\begin{proof}
Let $\pi \colon X' \rar X$ be as in the statement.
Let $H$ be a $q$-ample divisor that is fully supported on $B \sups =1.$.
Since $X$ is $\qq$-factorial, by~\cite{KM98}*{Lemma~2.62}, there is an effective divisor $F$ that is fully supported on the $\pi$-exceptional divisors such that $-F$ is $\pi$-ample.
Thus, for $0 < \epsilon \ll 1$, $\pi^\ast H-\epsilon F$ is ample over $Z$.
By assumption, the $\pi$-exceptional divisors all have center on $\Supp (H ) = \Supp (B \sups =1.)$.
Therefore, if $\epsilon$ is small enough, $\pi^\ast H-\epsilon F$ is effective.
By definition of dlt model, $\Supp {\rm Exc}(\pi) \subset \Supp((B')\sups =1.)$.
Thus, $\Supp(\pi^\ast H-\epsilon F) = \Supp ((B')\sups =1.)$, and the claim follows.
\end{proof}

\begin{lemma} 
\label{lemma supporto}
Let $(X,B,\bM.)/Z$ be a $\qq$-factorial generalized log canonical pair.
Assume that for some $0 < \epsilon \leq 1$, the generalized pair $(X,B \sups <1. + (1-\epsilon)B \sups =1.,\bM.)/Z$ is generalized klt.
Fix a generalized dlt model $(X^m,B^m,\bM.)/Z$ of $(X,B,\bM.)/Z$, and let $\pi \colon X^m \rar X$ denote the corresponding morphism.
Then, $\Supp(\pi^\ast (B \sups =1.))=\Supp((B^m)\sups =1.)$.
\end{lemma}

\begin{proof}
This follows at once, since $\pi$ only extracts divisors appearing in $\Supp((B^m)\sups =1.)$ and the condition that $(X,B \sups <1. + (1-\epsilon)B \sups =1.,\bM.)/Z$ is generalized klt implies that $\pi({\rm Exc}(\pi)) \subset B \sups =1.$.
\end{proof}

\subsection{Canonical bundle formula} \label{section cbf}
In this section, we recall the statement of the {\em canonical bundle formula} for generalized log canonical pairs and we extend it to the relative setting. 
We refer the interested reader to~\cites{Amb04, Amb05, FG14} for the notation involved and a more detailed discussion about the topic, in the case of log canonical pairs.

\begin{definition}
\label{lc-trivial.def}
Let $(X, B)$ be a sub-pair.
A contraction $f \colon X \rar T$ of quasi-projective varieties is an \emph{lc-trivial fibration} if
\begin{itemize}
    \item[(i)] $(X,B)$ is a sub-pair with coefficients in $\qq$ that is sub-log canonical over the generic point of $T$;
    \item[(ii)] $\mathrm{rank} f_* \O X. (\lceil \mathbf{A}^\ast (X,B)\rceil)=1$, where $\mathbf{A}^\ast (X,B)$ is the b-divisor defined in Example~\ref{discr.div.ex}; and
    \item[(iii)] there exists a $\qq$-Cartier $\qq$-divisor $L_T$ on $T$ such that $\K X. + B \sim_\qq f^\ast  L_T$.
\end{itemize}
\end{definition}
Condition (ii) above is automatically satisfied if $B$ is effective over the generic point of $T$.

Given a sub-pair $(X,B)$ and an lc-trivial fibration $f \colon X \rar T$,
there exist $\mathbb{Q}$-b-divisors $\mathbf{B}$ and $\mathbf{N}$ over $T$ such that the following linear equivalence relation, known as the {\it canonical bundle formula}, holds
\begin{equation}
    \label{cbf.eqn}
    K_X+B \sim_{\mathbb{Q}} f^\ast(K_T+\mathbf{B}_{T}+\mathbf{N}_{T}).
\end{equation}
The b-divisor $\mathbf{B}$ is called the \emph{boundary part} in the canonical bundle formula; it is a canonically defined b-divisor.
Furthermore, if $B$ is effective, then so is $\mathbf{B}_T$.
The b-divisor $\mathbf{N}$ in turn is called the \emph{moduli part} in the canonical bundle formula, and it is in general defined only up to $\mathbb{Q}$-linear equivalence. 
The linear equivalence~\eqref{cbf.eqn} holds at the level of b-divisors: namely, 
\begin{align*}
\overline{(K_X+B)} \sim_\qq f^\ast (\mathbf{K}+\mathbf{B}+\mathbf{N}),
\end{align*} 
where $\mathbf{K}$ denotes the canonical b-divisor.
\newline
The moduli b-divisor $\mathbf{N}$ is expected to detect the variation of the restriction of the pair induced on fibers of the morphism $f$ by restricting $B$.

\begin{theorem}
\cite{FG14}*{cf.~Theorem 3.6}
\label{classic cbf}
Let $(X, B)/S$ be a sub-pair and let $f \colon (X,B) \rar T$ be an lc-trivial fibration.
Let $\mathbf{B}$ and $\mathbf{N}$ be the boundary and the moduli part of $f$, respectively.
Then, $\mathbf K + \mathbf B$ and $\mathbf N$ are $\qq$-b-Cartier b-divisors.
Furthermore, $\mathbf{N}$ is b-nef over $S$.
\end{theorem}

\begin{remark}
In the setup of Theorem~\ref{classic cbf}, let $T'$ be a model where the nef part $\mathbf N$ descends in the sense of b-divisors.
Then, $\mathbf N \subs T'.$ is nef over $S$. 
In particular, $(T, \bB T.,\mathbf{N})/S$ is a generalized sub-pair.
If $B \geq 0$, then $(T, \bB T.,\mathbf{N})/S$ is a generalized pair.
\end{remark}

A generalization of the canonical bundle formula to the category of generalized pairs was introduced in~\cite{Fil18}*{Theorem~1.4}, where Theorem~\ref{classic cbf} is extended to the case of generalized sub-pairs $(X,B+M)$ endowed with the analog for generalized pairs of an lc-trivial fibration.

\begin{definition}
\label{gen-lc-trivial.def}
Let $(X,B,\mathbf{M})/Z$ be a generalized sub-pair.
A contraction $f \colon X \rar T$ of quasi-projective varieties over $Z$ is a \emph{generalized lc-trivial fibration} if
\begin{itemize}
    \item[(i)] $(X,B,\mathbf{M})$ is a generalized sub-pair with coefficients in $\qq$ that is generalized sub-log canonical over the generic point of $T$;
    \item[(ii)] $\mathrm{rank} f_* \O X. (\lceil \mathbf{A}^\ast (X,B,\mathbf{M})\rceil)=1$, where $\mathbf{A}^\ast (X,B,\mathbf{M})$ is the b-divisor defined in Example~\ref{gen.discr.div.example}; and,
    \item[(iii)] there exists a $\qq$-Cartier divisor $L_T$ on $T$ such that $\K X. + B + \mathbf{M}_X \sim_\qq f^\ast  L_T$.
\end{itemize}
\end{definition}

As in Definition~\ref{lc-trivial.def}, condition (ii) in Definition~\ref{gen-lc-trivial.def} is automatically satisfied if $B$ is effective over the generic point of $T$.

We are able to adapt the proof of~\cite{Fil18}*{Theorem~1.4} with minor changes to further extended the canonical bundle formula for generalized pairs to the relative setting.
Hence, we will work in this broader context, and highlight the relevant modifications that need to occur in the proof of~\cite{Fil18}*{Theorem~1.4}.

Let $(X,B,\mathbf{M})/S$ be a generalized sub-pair over a quasi-projective variety $X$, and let $f \colon X \rar T$ be a generalized lc-trivial fibration over $S$.
Without loss of generality, we can assume that $\dim T > 0$. 
Fix a divisor $L_T$ on $T$ such that $\K X. + B + M \sim_\qq f^\ast L_T$. 
For any prime divisor $D$ on $T$, let $l_D$ be the generalized log canonical threshold of $f^\ast D$ with respect to $(X,B+M)$ over the generic point of $D$. 
Then, we define 
\begin{align*}
B_T \coloneqq \sum b_D D, \; N_T \coloneqq L_T - (\K T. + B_T),
\end{align*}
where $b_D \coloneqq 1-l_D$, so that
\begin{align*}
\K X. + B + M \sim_\qq f^\ast (\K T. + B_T + N_T).
\end{align*}
Given $\tilde X$ and $\tilde T$ higher birational models of $X$ and $T$, respectively, fitting in the following commutative diagram of morphisms
\begin{align*}
\xymatrix{
\tilde X  \ar[r]^{\phi} \ar[d]_{\tilde f}& X \ar[d]^{f} \\
\tilde T  \ar[r]^{\psi} & T
}
\end{align*}
we will denote by $(\tilde X, \tilde B  + \tilde M)$ the trace of the generalized sub-pair $(X,B+M)$ on $\tilde X$. 
Furthermore, we set $L_{\tilde T} \coloneqq \psi^\ast  L_{T}$. 
With this piece of data, we can define divisors $B_{\tilde T}$ and $M_{\tilde T}$ such that
\begin{align*}
\K \tilde X. + \tilde B + \tilde M \sim_\qq \tilde{f}^\ast (\K \tilde T. + B_{\tilde T} + N_{\tilde T}),
\end{align*}
$B_{T}= \psi_* B_{\tilde T}$, and $N_{T}= \psi_* N_{\tilde T}$. 
In this way, Weil b-divisors $\mathbf{B}$ and $\mathbf{N}$ are defined.
We write $\bB \tilde{T}.$ and $\mathbf N \subs \tilde{T}.$ for the traces of $\bB.$ and $\mathbf N$ on any higher model $\tilde{T}$.
When the setup is clear, we shall write $B \subs \tilde T.$ and $N \subs \tilde T.$ in place of $\mathbf{B} \subs \tilde T.$ and $\mathbf{N} \subs \tilde T.$, respectively.

In this setup, we have the following theorem, referred to as \emph{generalized canonical bundle formula}.

\begin{theorem}
\label{generalized canonical bundle formula}
Let $(X,B,\mathbf{M})/S$ be a generalized sub-pair. 
Let $f\colon X \rar T$ be a generalized lc-trivial fibration over $S$.
If $B$ is effective over the generic point of $T$, then the b-divisor $\mathbf N$ is $\qq$-b-Cartier and b-nef over $S$.
\end{theorem}

Below, we shall summarize the relevant changes to the proof of~\cite{Fil18}*{Theorem~1.4} in order to drop the assumption on the projectivity of the pairs and on $S = \Spec(\cc)$.
We refer to~\cite{Fil19} for a detailed proof.

\begin{proof}[Sketch of proof]
Throughout the proof, $S$ will be a quasi-projective variety, with no further assumption.
For the reader's convenience, we subdivide the proof into several steps.

\medskip

{\bf Step 1:} 
{\it In this step we show that the statement of the theorem holds if we assume that $\mathbf{M}$ is b-semi-ample}.
\newline
Let $X'$ be a model where $\mathbf{M}$ descends.
For brevity, we set $M' \coloneqq \mathbf{M} \subs X'.$.
Let $h \colon X' \rar T$ be the induced morphism.
Let $\mathcal{U} \subset |M'|_\qq$ the set consisting of $\qq$-divisors $0 \leq \Delta' \sim_\qq M'$ such that $(X',B'+\Delta')$ is sub-log canonical over the generic point of $T$.
As $M'$ is semi-ample, $\mathcal{U}$ is non-empty.
Given $\Delta' \in \mathcal{U}$, we can apply Theorem~\ref{classic cbf} to $(X',B'+\Delta') \rar T$, thus obtaining a $\qq$-b-Cartier b-divisor $\mathbf N \sups \Delta'.$ that is b-nef over $S$.
As discussed in~\cite{Fil18}*{Remark 4.8}, we have
\begin{align*}
\quad \bB .= \inf \subs \Delta' \in \mathcal{U}. \bB .\sups \Delta'., \; 
\mathbf N= \sup \subs \Delta' \in \mathcal{U}. \mathbf N \sups \Delta'..
\end{align*}
We wish to apply weak semi-stable reduction to argue that $\mathbf N$ is a $\qq$-b-Cartier b-divisor that descends to a model satisfying explicit properties, see the proof of~\cite{Fil18}*{Theorem~4.13}.
The proof of the b-nefness of $\mathbf N$ over $S$ goes then through as explained in the proof of~\cite{Fil18}*{Theorem~4.15}.
One first shows that the statement is true if $\dim T = 2$, cf.~\cite{Fil18}*{Remark 4.14}.
To conclude, we reduce to the case of dimension $2$ by taking general hyperplane cuts on $T$, cf. the proof of~\cite{Fil18}*{Theorem~4.15}.

\medskip

{\bf Step 2:} 
{\it 
In this step, we show that the theorem holds when $\mathbf{M}$ is b-semi-ample over $T$}.
\newline
Let $X'$ and $M'$ be the objects constructed in Step 1.
Let $A$ be a divisor on $T$ which is ample over $S$. 
Since $M'$ is nef over $S$, and hence over $T$, and it is semi-ample over $T$, $M'+ \epsilon h^\ast  A$ is semi-ample over $S$ for any $\epsilon > 0$, cf.~\cite{Fil18}*{Proposition~4.7}.
Fix $\epsilon > 0$ with $\epsilon \in \qq$.
Let $H$ be an ample divisor on $S$ and write $p \colon X' \rar S, \; s\colon T \rar S$. 
Hence, for $n=n(\epsilon) \gg 0, \; M'+\epsilon h^\ast A + n p^\ast  H$ is semi-ample.
We consider the generalized sub-pairs $(X',B',\mathbf M +\epsilon \overline{ g^\ast  A})/S$ and $(X',B',\mathbf M +\epsilon \overline{g^\ast  A} + n \overline{p^\ast H})/S$.
Let $\bB .^\epsilon$, $\mathbf N ^\epsilon$ and $\bB .^{\epsilon,n}$, $\mathbf N ^{\epsilon,n}$ the b-divisors respectively induced on $T$.
By Remark~\ref{remark valuations nef part}, the generalized discrepancies of $(X,B,\mathbf M)/S$, $(X',B',\mathbf M +\epsilon \overline{g^\ast  A})/S$, and $(X',B',\mathbf M +\epsilon \overline{ g^\ast  A }+ n \overline{p^\ast H})/S$ agree by construction.
Thus, we have
\begin{align*}
\bB . = \bB . ^\epsilon = \bB . \sups \epsilon,n.,
\qquad 
\mathbf N \sups \epsilon,n.= \mathbf N ^\epsilon + n \overline{s^\ast H}, 
\qquad 
\mathbf N ^\epsilon = \mathbf N + \epsilon \overline{A}.
\end{align*}
Therefore, if for some $(\epsilon,n)$ the b-divisor $\mathbf N \sups \epsilon,n.$ is $\qq$-b-Cartier, then so are $\mathbf N$ and $\mathbf N ^{\epsilon'}$ for any $\epsilon'>0$ and all these b-divisors descend to the same model of $T$.
We can then apply Step 1 to the generalized pair $(X',B',\mathbf M + \epsilon \overline{g^\ast  A} + n\overline{ p^\ast H})/S$ together with the morphism $h \colon X' \rar T$ and obtain $\qq$-b-Cartier b-divisor $\mathbf N \sups \epsilon,n.$ which is b-nef$/S$. 
Since we have $\mathbf N \sups \epsilon,n. = \mathbf N \sups \epsilon. + n\overline{s^\ast H}$ and $\overline{s^\ast H}$ is trivial over $S$, $\mathbf N ^\epsilon$ is a $\qq$-b-Cartier b-divisor that is b-nef$/S$.
We let $\epsilon>0$ vary and approach 0.
Thus, as $\mathbf N$ is a limit of $\qq$-b-Cartier b-divisors that all descend to the same model of $T$ and are all b-nef over $S$, then $\mathbf N$ is a $\qq$-b-Cartier b-divisor that is b-nef over $S$.

\medskip

{\bf Step 3:} 
{\it
In this step we show that the statement of the theorem holds when $X$ is $\qq$-factorial klt, $\rho(X/T)=1$ and {$\mathbf{M}_X$} is relatively ample over $T$}.
\newline
The proof of this case goes through as in~\cite{Fil18}*{Lemma~5.2} and reduces to Step 2.
More precisely, since $\bM.$ is b-nef over $S$ and $\bM X.$ is ample over $T$, we can approximate $\bM.$ with b-divisors that are b-nef over $S$ and b-semi-ample over $T$.
The fact that $X$ is a $\qq$-factorial klt variety allows us to moving the difference between $\bM.$ and its approximation to the boundary part of a generalized pair.
So, we can regard $(X,B,\bM.)$ as being approximated by generalized sub-pairs to which Step 2 can be applied.
Since the boundary parts of these approximations can be arranged to have the same support, we can apply the observation in Step 1 about the explicit description of the model where the b-divisors descend.
This allows us to conclude, as $\mathbf{N}$ is realized as limit of $\qq$-b-Cartier b-divisors that descend to the same model, and hence it is $\qq$-b-Cartier itself.

\medskip

{\bf Step 4:} 
{\it 
In this step we show that the statement of the theorem holds in its full generality}.
\newline
By Theorem~\ref{generalized dlt model}, we can assume that $X$ is $\qq$-factorial and $(X,B^h)$ is dlt.
If $\mathbf{M}_X$ is numerically trivial along the generic fiber of $X \rar T$, we can reduce to Theorem~\ref{classic cbf}.
If $\mathbf{M}_X$ is not numerically trivial along the generic fiber of $X \rar T$, $(X,B^h)$ is not pseudo-effective over $T$.
Thus, we can run a $(\K X. + B^h)$-MMP over $T$ with scaling, which terminates with a Mori fiber space $X'' \rar U$.
Thus, we can apply Step 3 to $(X'',B'',\mathbf{M})/S$ and $X'' \rar U$.
By~\cite{Fil18}*{Lemma~5.1}, we can then conclude inductively.
\end{proof}

\subsection{Standard $\pr 1.$-links and $\mathbb P^1$-linkage} 
\label{subsect P1 link}
In this section we recall the notion of standard $\pr 1.$-link and extend it to generalized pairs.
The definition of pull-back for b-divisors can be found at the end of Section~\ref{b-div.subs}.

\begin{definition}
\label{def p1 link}
A generalized pair 
$(X,D_1+D_2+\Delta,\mathbf{M})/S$
endowed with a morphism 
$\pi \colon X \rar T$ 
over 
$S$ 
is a {\em standard $\pr 1.$-link}
if
\begin{enumerate}
    \item[(0)]
$D_1$ and $D_2$ are distinct reduced prime divisors and $\lfloor \Delta \rfloor =0$;
    \item
$\K X. + D_1 + D_2 + \Delta \sim \subs \qq,\pi. 0$;
    \item
there exists a $\mathbb{Q}$-b-Cartier $\mathbb{Q}$-b-divisor $\mathbf{N}$ on $T$ such that $\bM. \sim_{\mathbb Q} \pi^* \mathbf{N}$;
    \item
$\pi\vert_{D_i} \colon D_i \rar T$ is an isomorphism for $i=1,2$;
    \item
$(X,D_1+D_2+\Delta,\mathbf{M})/S$ is generalized plt; and
    \item
every reduced fiber of $\pi$ is isomorphic to $\pr 1.$.
\end{enumerate}
\end{definition}

The following lemma shows that condition (2) in Definition~\ref{def p1 link} is implied by the condition $\bM X. \sim_{\mathbb Q, T} 0$.
We note, though, that the two conditions are not equivalent, since condition (2) in Definition~\ref{def p1 link} only restricts the behavior of $\mathbf M$ along the general fiber of the contraction $\pi$.

\begin{lemma}
\label{condition.2.lemma}
Let $f \colon X \rightarrow T$ be a projective contraction over a base scheme $S$ and $\bM.$ be a $\mathbb Q$-b-Cartier $\mathbb Q$-b-divisor on $X$ that is b-nef over $S$.
If $\bM X. \sim_{\mathbb Q,T} 0$, then there is a $\mathbb{Q}$-b-Cartier b-divisor $\mathbf{N}$ on $T$ such that $\bM. \sim_{\mathbb Q}f^* \mathbf N$.
\end{lemma}

\begin{proof}
Let $X^r \rar X$ be a resolution of $X$ such that $\bM.$ descends to $X^r$.
Let $X^f \rar T^f$ be a flattening of $X^r \rar T$, see \cite{raynaud_gruson}*{Th\'{e}or\`{e}me 5.2.2}, and let $\widetilde{T}$ be a resolution of $T^f$.
Let $\widetilde{X}$ be the normalization of the main component of $X^f\times_{T^f}\widetilde{T}$.
This ensures that $\widetilde{X}\to\widetilde{T}$ still has equidimensional fibers.
Since $\bM X.$ is torsion along the generic fiber of $X \rightarrow T$, the same holds true for $\bM \widetilde{X}.$ along the generic fiber of $\widetilde{X}\to\widetilde{T}$ by the negativity lemma.
Thus, by the proof of \cite{Wit17}*{Lemma 2.18}, under our assumptions,~\footnote{We note that \cite{Wit17}*{Lemma 2.18} holds even without assuming the projectivity of the varieties involved. In that case, using the same notation as in the statement of {\it op. cit.} it suffices to assume that the $\mathbb Q$-divisor $L$ is relatively nef.}
it follows that $\bM {\widetilde{X}}. \sim \subs \qq,\widetilde{T}. 0$.
Since $\mathbf M$ is the b-Cartier closure of $\mathbf M_{\widetilde{X}}$, the conclusion of the statement follows at once.
\end{proof}

\begin{remark}
Conditions (1), (3), and (5) of Definition~\ref{def p1 link} imply that $\Delta$ is vertical over $T$.
Moreover, up to replacing $\bM.$ with $\pi^*\mathbf{N}$ in its $\mathbb Q$-linear equivalence class as b-divisor,
the composition $\pi\vert_{D_2}^{-1} \circ \pi\vert_{D_1}$ induces an isomorphism of generalized klt pairs
\begin{align*}
(D_1, \Delta \subs D_1.,\mathbf{M}\vert_{D_1})/S \simeq (D_2, \Delta \subs D_1.,\mathbf{M}\vert_{D_2})/S.
\end{align*}
\end{remark}

\begin{definition}
\label{link.def}
Let $(X, B,\mathbf{M})/S$ be generalized dlt.
Assume that there is a morphism $\pi \colon X \rar T$ over $S$ such that $\K X. + B + M \sim \subs \qq,\pi. 0$.
Let $Z_1$, $Z_2$ be two generalized log canonical centers.
\begin{enumerate}
    \item 
We say that $Z_1$ and $Z_2$ are {\em directly $\pr 1.$-linked} if there is a generalized log canonical center $W$ (alternatively, $W=X$ itself) satisfying the following properties:
\begin{enumerate}
    \item 
$Z_i \subset W$, $i=1, 2$;

    \item 
$\pi (W) = \pi (Z_1) = \pi (Z_2)$; and

    \item
over a non-empty open subset of
$\pi(W)$, the generalized pair $(W,B_W+N_W)$ induced by generalized adjunction onto $W$ is birational to a standard $\pr 1.$-link, with the $Z_i$ mapping to the two horizontal sections of the $\pr 1.$-link structure.
\end{enumerate}
\item
We say that $Z_1$ and $Z_2$ are {\em $\pr 1.$-linked} if either  
$Z_1=Z_2$
or there exists a sequence of (distinct) generalized log canonical centers 
$Z_1', \ldots , Z_n'$ 
such that 
$Z_1 ' = Z_1$, 
$Z_n ' =Z_2$ 
and 
$Z_i'$ 
is directly 
$\pr 1.$-linked 
to 
$Z_{i+1}'$ 
for 
$i=i, \ldots , n-1$.
\end{enumerate}
\end{definition}

It is an immediate consequence of the previous definition that every $\pr 1.$-linking defines a birational map between $(Z_1,B \subs Z_1. + N \subs Z_1.)$ and $(Z_2, B \subs Z_2. + N \subs Z_2.)$.

\begin{remark}\label{selflink}
With the notation and assumptions of Definition~\ref{link.def}, if $Z$ is a generalized log canonical center, then $Z$ is never directly $\mathbb{P}^1$-linked to itself:
indeed, as $(X, B+M)$ is generalized dlt, by adjunction, any generalized log canonical center $Z_1$ directly $\mathbb{P}^1$-linked to $Z$ must be distinct from $Z$ itself.
Therefore, being directly $\pr 1.$-linked is not a reflexive relation.
Furthermore, considering the log Calabi--Yau surface $(\mathbb{P}^1\times \mathbb{P}^1, B)$, $B\coloneqq \{0\} \times \mathbb{P}^1+\{\infty\} \times \mathbb{P}^1 + \mathbb{P}^1 \times \{0\} + \mathbb{P}^1 \times \{\infty\}$, it is immediate to see that the relation is not transitive either.
The definition of $\mathbb P^1$-linkage provides instead an equivalence relation which is the smallest equivalence relation that includes that of direct $\mathbb P^1$-linkage.
\end{remark}

\begin{remark}\label{remark minimality}
A generalized log canonical center $Z_1$ for a generalized dlt pair $(X, B+ M)/S$ is minimal, with respect to inclusion, if and only if the generalized pair $(Z_1,B \subs Z_1., \mathbf{N})$ induced by adjunction along $Z_1$ is generalized klt, cf.~\cite{BZ16}*{Definition~4.7}.
If $(Z_1,B \subs Z_1. + N \subs Z_1.)$ is generalized klt and $Z_2$ is $\pr 1.$-linked to $Z_1$, then also $(Z_2,B \subs Z_2. + N \subs Z_2.)$ is generalized klt, cf.~\cite{Kol13}*{Corollary~4.35}.
In particular, $Z_1$ is a minimal generalized lc center if and only if so is $Z_2$.
\end{remark}

The following example shows that in Definition~\ref{link.def}.1.(c) the birational map between the generalized log canonical center $W$ and the standard $\mathbb P^1$-link may not be a morphism, that is, it may not be defined everywhere.

\begin{example}
We follow the notation of Definition~\ref{link.def}.
Let us take 
$X= \mathbb P^1 \times \mathbb P^1=W$, 
$S= \mathrm{Spec}(k)$,
$T= \mathbb P^1$
and the projection 
$pr_1 \colon \mathbb P^1 \times \mathbb P^1 \to \mathbb P^1$
onto the first copy
as 
$\pi$.
Moreover, set 
$B \coloneqq  H+V+C$,
$\mathbf M=0$,
where 
$H= \mathbb P^1 \times \{ 0 \}=Z_1$,
$V= \{ 0 \} \times \mathbb P^1$, 
and 
$C \in \vert \mathcal O_{\mathbb P^1 \times \mathbb P^1}(1, 1)\vert=Z_2$
is a general element.
Thus,
$pr_1(C)=pr_1(H)=pr_1(X)=T$, and $C$, $H$ are directly $\mathbb P^1$-linked since they are sections of $\pi$.
Nonethelss, any birational morphism to a standard $\mathbb P^1$-link cannot be taken to be everywhere defined since $C$ and $H$ intersect in $X$, whereas in the definition of standard $\mathbb P^1$-link the two sections are disjoint, see (0) in Definition~\ref{def p1 link}.
\end{example}

\subsection{Dual complexes for generalized log canonical pairs.}
\label{dual.compl.gen.ssect}
We recall the notion of dual complex of a simple normal crossing variety.

\begin{definition}  
\label{stratum.def}
Let $E$ be a simple normal crossing variety defined over a field $k$ with irreducible components $\lbrace E_i | i \in I \rbrace$.
A {\em stratum} $F$ of $E$ is any irreducible component $F$ of $\cap \subs i \in J. E_i$ for some $J \subset I$.
\end{definition}

Given a simple normal crossing variety $E= \bigcup_{i \in I} E_i$ and a stratum $F \subset \bigcap \subs i \in J. E_i$ of $E$, for any $j \in J$ there is a unique irreducible component $F_j$ of $\bigcap \subs i \in J \setminus \lbrace j \rbrace. E_i$ that contains $F$. 
Using this observation, it is possible to construct a regular $\Delta$-complex, in the sense of~\cite{Hat02}*{page 103}, that encodes the combinatorial structure of the strata of $E$.

\begin{definition}
\label{dual.compl.def}
Let $E$ be a simple normal crossing variety defined over a field $k$ with irreducible components $\lbrace E_i | i \in I \rbrace$.
The {\em dual complex of $E$}, denoted by $\mathcal{D}(E)$, is a CW-complex whose vertices are labeled by the irreducible components of $E$ and for every stratum $F \subset \bigcap \subs i \in J. E_i$ we attach a $(|J|-1)$-dimensional cell $C_F$ by attaching the facet corresponding to the inclusion $J \setminus \{j\}$ to the cell corresponding to $F_j$.
\end{definition}
Let $(X, B, \mathbf{M})/S$ be a generalized log canonical pair.
Consider a log resolution $f \colon X'  \rar X$ of $(X,B)$ where $\mathbf{M}$ descends.
In particular, the support $f_\ast^{-1}B + {\rm Exc}(f)$ is a simple normal crossing divisor and we can write
\begin{align*} 
K_{X'} + B' + \mathbf{M}_{X'}= f^\ast(K_X+B+M),
\end{align*}
Hence, we can define the dual complex $\mathcal{D}((B') \sups =1.)$ of the simple normal crossing variety $(B') \sups =1.$ as in Definition~\ref{dual.compl.def}.

\begin{definition} 
\label{def dual complex}
Let $(X, B, \mathbf{M})/S$ be a generalized log canonical pair, and let $f \colon X' \rar X$ be a log resolution of $(X,B)$ where $\mathbf{M}$ descends.
The {\em dual complex} $\mathcal{DMR}(X,B, \mathbf M)$ of $(X, B, \mathbf{M})/S$ is the PL-homeomorphism class of the $\Delta$-complex $\mathcal{D}((B') \sups =1.)$ constructed above.
\end{definition}

Let $g \colon X'' \rar X$ be a different log resolution of $(X,B)$ where $\mathbf{M}_{X''}$ descends, and write
\begin{align*}
K_{X''} + B'' + \mathbf{M}_{X''}= g^\ast(K_X+B+M).
\end{align*}
As for $(B') \sups =1.$, we can define the dual complex $\mathcal{D}((B'') \sups =1.)$ of $(B'') \sups =1.$.
By construction, the sub-pairs $(X',B')$ and $(X'',B'')$ are crepant birational to each other, as we have chosen both log resolutions $f$ and $g$ so that $\mathbf{M}$ descends to $X'$ and $X''$.
Hence, using the weak factorization theorem for the birational map $X' \dashrightarrow X''$, it is possible to prove that the complexes $\mathcal{D}((B') \sups =1.)$ and $\mathcal{D}((B'') \sups =1.)$ are PL-homeomorphic to each other, see~\cite{dFKX}*{Proposition~11} for full details of the argument.

\begin{remark} 
\label{remark dual complex gpair to pair}
In the setup of Definition~\ref{def dual complex}, if we further assume that $(X,B, \mathbf{M})$ is $\qq$-factorial and generalized dlt, then every log canonical place of $(X,B+M)$ is a log canonical place of $(X,B)$, and vice versa, cf. Definition~\ref{gen.dlt.def}.
Thus, in this case we have $\mathcal{DMR}(X,B)=\mathcal{DMR}(X,B, \mathbf M)$.
\end{remark}

\subsection{Dual complex of non-log canonical pairs} \label{section dual cplx non lc}

We extend the definition of dual complex to generalized pairs that are not necessarily generalized log canonical.
To this end, we will use dlt models of generalized log pairs, cf.~\S~\ref{gen.dlt.mod.ssect} and the notation defined there.

\begin{definition} 
\label{def dual complex nonlc}
Let $(X, B, \mathbf{M})/S$ be a generalized pair.
Assume that $(X, B, \mathbf{M})/S$ is not generalized log canonical.
The \emph{dual complex} $\mathcal{DMR}(X,B, \mathbf M)$ of $(X, B, \mathbf{M})/S$ is the simple homotopy equivalence class of the $\Delta$-complex $\mathcal{D}((B^{m})^{=1})$, where $(X^m,B^m,\bM.)$ is a generalized dlt model of $(X,B,\bM.)$.
\end{definition}

Let $(X,B, \mathbf{M})$ a generalized pair that is not generalized log canonical, we ought to show that Definition~\ref{def dual complex nonlc} is independent of the choice of a generalized dlt model of $(X,B, \mathbf{M})$, or, equivalently, that the dual complexes of any two generalized dlt models of $(X,B, \mathbf{M})$ are simple homotopy equivalent.

\begin{remark}
In Definition~\ref{def dual complex}, the dual complex of a generalized log canonical pair was defined by considering the PL-homeomorphism class of the dual complex.
In the non-log canonical case, we are bound to use the weaker notion of simple-homotopy equivalence class.
Nevertheless, this notion is good enough to discuss the collapsibility of $\mathcal{DMR}(X,B, \mathbf M)$ and to compute its cohomology.
\end{remark}

\begin{lemma}
\label{well.posed.def}
Assume the same notations and assumptions introduced above.
Consider two generalized dlt models of $(X,B,\bM.)$, denoted by $\rho_i \colon (X_i^m,B_i^m, \mathbf{M}) \rar X$, $i=1, 2$.
Then $\mathcal{D}((B_1^m)\sups =1.)$ and $\mathcal{D}((B_2^m)\sups =1.)$ are simple homotopy equivalent.
\end{lemma}

By Remark~\ref{remark valuations nef part}, considering the pair $(X,B, \mathbf{M} + \overline{H})$, where $H$ is a suitable ample divisor on $X$, $\mathbf{M} + \overline{H}$ is b-nef and b-big.
Hence, up to subsituting $\mathbf{M}$ with $\mathbf{M}+ \overline{H}$, we can assume that $\mathbf{M}$ is b-nef and b-big.
Now, let $\pi \colon X' \rar X$ be a log resolution of $(X,B)$ where $\mathbf{M}$ descends.
We will assume that $\pi$ factors through both $X_1^m$ and $X_2^m$.
Let $\pi_i \colon X' \rar X_i^m$ denote the corresponding morphisms for $i=1,2$.
\begin{lemma}
\label{lem:big.nef.b-div}
With the same notations and assumptions as above.
If $X'$ is a sufficiently high model of $X$, then there exists an effective $\qq$-divisor $E'$ such that the following conditions hold:
\begin{enumerate}
    \item 
$\Supp(E') \cup \Supp(B')$ is simple normal crossing; and

    \item 
for every sufficiently large positive integer $k \gg 1$, there exists a big and semi-ample $\qq$-divisor $A'_k$ such that $\mathbf{M} \subs X'. \sim_\qq A'_k + E'_k$, where $E'_k \coloneqq \frac{E'}{k}$.
\end{enumerate}
\end{lemma}
\begin{proof}
As $\mathbf{M} \subs X'.$ is nef and big and it descends to the model $X'$ fixed before the lemma,  there exists and effective divisor $E$ such that $\mathbf{M} \subs X'. \sim_\qq A_k+E_k$, where $A_k$ is ample for $E_k  \coloneqq  \frac{E}{k}$, see~\cite{KM98}*{Proposition~2.61}.
As $E$ may not be simple normal crossing, it suffices to replace $X'$ with a log resolution of $(X, E)$ and define $A'_k$, $E'$ to be their pull-back on the new model.
\end{proof}

For $i=1,2$, we define $B_i'$ by the log pull-back formula $\K X'. + B_i'=\pi_i^\ast (\K X_i^m. + B_i^m)$.
As $E'$ is effective, we may find $0 < \epsilon \ll 1$ so that $\Supp( (B_i')\sups \geq 1.)= \Supp((B_i' + \epsilon E') \sups \geq 1.)$ for $i=1,2$.
Thus, for $k \gg 1$, we have $\Supp( (B_i')\sups \geq 1.)= \Supp((B_i' + E'_k) \sups \geq 1.)$ for $i=1,2$.
Let us fix one such value of $k$ and let $A'_k$ be a generic effective divisor in its $\qq$-linear equivalence class.
For $i =1, 2$, we define
\begin{align*}
\Gamma' \coloneqq A'_k + E'_k, \; 
\Gamma \coloneqq \pi _* \Gamma', \; 
\Gamma_i^m \coloneqq \pi_{{i},*} \Gamma'.
\end{align*}

By construction, the following two conditions hold:
\begin{enumerate}
    \item[(i)]
    $a_F(X,B+M) \geq a_F(X,B+\Gamma)$ for every divisorial valuation $F$ centered on $X$; and
    \item[(ii)] a prime divisor $G \subset X'$ satisfies $a_G(X_i^m,B_i^m+\Gamma_i^m) \leq 0$ if and only if $a_G(X^m_i,B^m_i) \leq 0$.
\end{enumerate}

Both properties follow from the fact that $\mathbf{M} \subs X'. \sim_\qq A'_k + E'_k$ and $\Supp( (B_i')\sups \geq 1.)= \Supp((B_i' + E'_k) \sups \geq 1.)$, while $A'_k$ does not contribute to any singularity by Bertini's theorem.

By (i), it follows that $X_i^m \rar X$ extracts only divisors with non-positive log discrepancy for $(X, B+\Gamma)$.
Therefore, $B_i^m+\Gamma_i^m$ is effective for $i=1,2$.
As $X'$ is a log resolution of all the divisors involved, we can use it as input variety to construct a dlt model by means of a suitable MMP, as in Theorem~\ref{generalized dlt model}.
By running a suitable relative MMP over $X_i^m$, we can obtain a $\qq$-factorial dlt model $(X_i'',B_i''+\Gamma_i'')$ of $(X_i^m,B_i^m+\Gamma_i^m)$.

\begin{lemma}
\label{lem:equiv.dual.complex}
With the same notations and assumptions as above,
 $\mathcal{D}((B_i^m)^{=1})$ is simple homotopy equivalent to $\mathcal{D}((B_i''+\Gamma_i'')^{=1})$ for $i=1,2$.
\end{lemma}

\begin{proof}
Since by construction
\begin{align*}
\K X_i^m.+B_i^m+\Gamma_i^m \leq \pi_i^\ast  (\K X. + B + \Gamma),
\end{align*}
$(X_i'',B_i''+\Gamma_i'')$ is also a dlt model for $(X,B+\Gamma)$.
Therefore, by~\cite{Nak19}*{Proposition~2.14}, $\mathcal{D}((B_1''+\Gamma_1'')^{=1})$ and $\mathcal{D}((B_2''+\Gamma_2'')^{=1})$ are simple homotopy equivalent to each other, as they compute the dual complex of $(X,B+\Gamma)$.
\end{proof}

We are now ready to prove the proof that Definition~\ref{def dual complex nonlc} depends neither on the choice of the ample divisor $H$, nor on the choice of the representatives $A'_k, E'_k$ that we made in the course of the construction contained in this section.

\begin{proof}[Proof Lemma~\ref{well.posed.def}]
By construction, the exceptional divisors of $X_i'' \rar X_i^m$ appear as divisors on $X'$.
Furthermore, as already observed, a divisor $E \subset X'$ satisfies $a_E(X_i^m,B_i^m+\Gamma_i^m) \leq 0$ if and only if $a_E(X^m_i,B^m_i) \leq 0$.
Therefore, $X_i'' \rar X_i^m$ only extracts divisors $E$ with $a_E(X^m_i,B^m_i) \leq 0$.
Hence, the variety $X''_i$ provides a dlt model for the dlt pair $(X^m_i,B^m_i)$.
Hence, $\mathcal{D}((B_i''+\Gamma_i'')^{=1})$ computes $\mathcal{DMR}(X_i^m,B_i^m)$.
As $(X^m_i,B^m_i)$ is dlt, by~\cite{dFKX}*{Proposition~11}, $\mathcal{DMR}(X_i^m,B_i^m)$ is a well defined PL-homeomorphism class.
Therefore, $\mathcal{D}((B_i^m)^{=1})$ is PL-homeomorphic to $\mathcal{D}((B_i''+\Gamma_i'')^{=1})$ for $i=1,2$.
Thus, $\mathcal{D}((B_i^m)^{=1})$ is simple homotopy equivalent to $\mathcal{D}((B_i''+\Gamma_i'')^{=1})$ for $i=1,2$.
\end{proof}

\section{Connectedness for birational maps} \label{sect.bir.case}
In this section, we recall some results explaining how the structure of the non-klt locus of generalized pairs changes under birational maps.
The first result is an adaptation to the generalized pair case of
\cite{Kol92}*{Theoreom 17.4}.
It provides a partial relative birational version of the connectedness principle for generalized pairs.

\begin{proposition}[{cf.~\cite{Bir16a}*{Lemma~2.14}}] 
\label{p.conn.gen.pairs}
Let 
$(X, B, \mathbf{M})/Z$ 
be a generalized sub-pair, where 
$h \colon X  \rar Z$ 
is a projective contraction of normal varieties.
Assume that
\begin{enumerate}
\item $h_\ast B^{<0}=0$; and
\item $-(K_X+B+M)$ is $h$-nef and $h$-big.
\end{enumerate}
Let $g \colon W  \rar X$ be a log resolution of $(X, B)$ to which $\mathbf{M}$ descends. Let 
\begin{align*}
K_W + \mathbf{M}_W=g^\ast (K_X+B+M) + \sum a_i E_i.
\end{align*}
Let $F\coloneqq -\sum_{a_i \leq -1} a_iE_i$.
Then $(g \circ h) \vert_{\Supp (F)} \colon \Supp (F) \to Z$ has connected fibers.
\end{proposition}

\begin{remark}
The support of $F$ is exactly the non-klt locus of the generalized sub-pair $(W, -\sum a_iE_i, \mathbf{M})/Z$.
\end{remark}

\begin{remark}
When $h$ is birational, then it suffices to require that $-(K_X+B+M)$ is $h$-nef.
\end{remark}

\begin{proof}
Let $A\coloneqq \sum_{a_i > -1} a_iE_i$ and let $s\coloneqq g\circ h$.
Then,
\begin{align*}
\lceil A \rceil - \lfloor F \rfloor= K_W+ \mathbf{M}_W -g^\ast (K_X+B+M) + \{-A \} + \{F\}
\end{align*}
and, by relative Kawamata--Viehweg vanishing,
\begin{align*}
R^1s_* \strutt_Z(\lceil A \rceil - \lfloor F \rfloor) =0,
\end{align*}
since $\mathbf{M}_W -g^{*}(K_X+B+M)$ is $s$-big and $s$-nef and 
$\{-A \} + \{F\}$ is simple normal crossing.
\newline
Considering the following exact sequence
\begin{align*}
0 \to \strutt_X(\AF) \to \strutt_X(\lceil A \rceil) \to \strutt_{\lfloor F \rfloor}(\lceil A \rceil) \to 0
\end{align*} 
and applying $s_*$, we get that 
\begin{align*}
s_* \strutt_X(\lceil A \rceil) \to  s_* \strutt_{\lfloor F \rfloor}(\lceil A \rceil)
\end{align*}
is surjective. As $\lceil A \rceil$ is $s$-exceptional, then $s_* \strutt_X(\lceil A \rceil) 
\simeq \strutt_Z$.
Hence, $s_* \strutt_X(\lceil A \rceil)$ is locally principally generated in a neighborhood fo $z$. 
\newline
Assume by contradiction that $\Supp (F)$ is not connected over some point $z \in Z$. Then, $\strutt_{\lfloor F \rfloor}(\lceil A \rceil)$ is not locally principally generated over $z$, since it contains the push-forward 
of the sections vanishing on any but one of the components of the $\lfloor F \rfloor$. On the other hand, the locally principally generated sheaf $s_* \strutt_X(\lceil A \rceil) 
\simeq \strutt_Z$ surjects onto $\strutt_{\lfloor F \rfloor}(\lceil A \rceil)$. This leads to a contradiction.
\end{proof}

\begin{remark}
\label{r.num.comp}
By applying Proposition~\ref{p.conn.gen.pairs} in the case where $Z=X$ and $h$ is the identity map, we obtain that the number of connected components $\Nklt(X, B , \mathbf  M)$ is unchanged by passing to a log resolution where $\mathbf M$ descends.
\end{remark}

\begin{lemma}
\label{conn.div.contr.lemma}
Let $(X, B, \mathbf{M})/S$ be a $\mathbb Q$-factorial
generalized pair and let 
$f \colon X  \rar Y$ 
be a proper morphism of algebraic varieties over $S$. 
Let  
$\pi \colon X \to X_1$
be a divisorial contraction over $Y$, that is, we have the following commutative diagram
\begin{align*}
\xymatrix{
X \ar[dr]_f  \ar[rr]^\pi & &  X_1 \ar[dl]^{f_1}\\
& Y.
}
\end{align*}
Assume that 
\begin{itemize}
\item 
$\rho(X/X_1)=1$;

\item 
$K_X+B+M \sim_{\mathbb{R}, f} 0$;

\item 
$B \sups \geq 1.$ is $\pi$-ample; and 
    
\item 
$\Nklt(X, B , \mathbf M) \subset \Supp (B \sups \geq 1.)$.
\end{itemize}
Then $\Nklt(X_1, B_1 , \mathbf M) \subset \Supp(B_1 \sups \geq 1.)$, where $B_1  \coloneqq  \pi_\ast B$.
\newline
Moreover, for any point $y \in Y$, the number of connected components of $\Nklt(X_{1}, B_1, \mathbf M) \cap f_1^{-1}(y)$ 
is the same as the number of connected components of $\Nklt(X, B, \mathbf M) \cap f^{-1}(y)$.
\end{lemma}

\begin{proof}
Let $E$ be the divisor contracted by $\pi$.
As $K_X+B+M \sim_{\mathbb{R}, f} 0$, $\Nklt(X_1, B_1 , \mathbf M) = \pi(\Nklt(X, B , \mathbf  M))$. If $E \not \subset \Nklt(X, B, \mathbf M)$ then there is nothing to prove. 
Instead, if $E \subset \Nklt(X, B, \mathbf M)$, then $\mu_E (B) \geq 1$.
As $E$ is exceptional for $\pi$, then $E \cdot R < 0$, where $R$ is the extremal ray of $\overline{NE}(X/Y)$ corresponding to the contraction $\pi$. 
As $B^{\geq 1}$ is $\pi$-ample, there exists a component $G$ included in the support of $B^{\geq 1}$ such that $G \cdot R >0$.
In particular, $\pi_\ast G \supset \pi(E)$, which proves the claim, as $B_1^{\geq 1} = \pi_\ast (B^{\geq 1})$.
\newline
To prove the last assertion, it suffices to notice that, since we have
$\Nklt(X_1, B_1 , \mathbf M) = 
\pi(\Nklt(X, B , \mathbf  M))$, 
the number of connected components of 
$\Nklt(X_1, B_1 , \mathbf M) \cap f_1^{-1}(y)$ 
cannot be strictly larger than the number of those of 
$\Nklt(X, B , \mathbf M) \cap f^{-1}(y)$.
If the number of components were to actually decrease, then taking a common resolution 
\begin{displaymath}
\xymatrix{
& W \ar[dr]^p \ar[dl]_q& \\
X  \ar[rr]^\pi & & X_1,
}
\end{displaymath}
where $\mathbf{M}$ descends and applying Proposition~\ref{p.conn.gen.pairs} with $h=\pi$, $Z=X_1$, $g=q$, $p=g \circ h$ we would obtain a contradiction, since that situation would imply a lack of connectedness of the fibers of $p\vert_{\Supp F} \colon \Supp F \to X_1$, where $F$ is the divisor defined in the statement of Proposition~\ref{p.conn.gen.pairs}.
\end{proof}

\begin{lemma}
\label{flop.nklt.lem}
Let $(X, B, \mathbf{M})/S$ be a $\mathbb Q$-factorial
generalized pair and let $f \colon X  \rar Y$ be a proper morphism of algebraic varieties over $S$. 
Let $\pi \colon X \dashrightarrow X^+$ be a $(K_X+B+M)$-flop over $Y$, that is, we have the following commutative diagram
\begin{displaymath}
\xymatrix{X \ar@{-->}[rr]^{\pi}  \ar[rd]^{l}  \ar[rdd]_f& &\ar[ld]_{l^+} \ar[ldd]^{f^+} X^+\\
& Z \ar[d]^{f_Z} .& \\
&Y&}
\end{displaymath}
Assume that 
\begin{itemize}
    \item $B \sups \geq 1.$ is $g$-ample; and
    \item $\Nklt(X, B , \mathbf M) \subset \Supp(B \sups \geq 1.)$.
\end{itemize}
Then $\Nklt(X^+, B^+ , \mathbf M) \subset \Supp((B^{+})^{ \geq 1})$, where $B^+$ is the strict transform of $B$.
\newline
Moreover, for any point $y \in Y$, the number of connected components of 
$\Nklt(X^+, B^+, \mathbf M) \cap (f^{+})^{-1}(y)$ 
is the same as the number of connected components of 
$\Nklt(X, B, \mathbf M) \cap f^{-1}(y)$.
\end{lemma}
\begin{proof}
As $l$ is a flopping contraction, 
then
\begin{align*}
K_X+B+M \sim_{\mathbb{R}, l} 0, \text{ and }  K_{X^+}+B^+ +M^+ \sim_{\mathbb{R}, l^+} 0,
\end{align*}
where $M^+ \coloneqq \mathbf M \subs X^+.$.
This in turn implies that 
\begin{align*}
\Nklt(Z, B_Z+M_Z) = 
l(\Nklt(X, B, \mathbf M))=
l^+(\Nklt(X^+, B^+ , \mathbf M)),
\end{align*}
and $l_\ast (B^{\geq 1})=l^+_\ast (B^{+ \geq 1})$.
As $B \sups \geq 1.$ is $l$-ample, it follows that $-(B^{+})^{ \geq 1}$ is $l^+$-ample.
Since 
$-(B^{+})^{ \geq 1}$ 
is anti-effective, it follows that the 
$l^+$-exceptional locus is contained in 
$\Supp((B^{+})^{ \geq 1})$.
Thus, $\Nklt(X^+, B^+ , \mathbf M) \subset \Supp((B^{+})^{ \geq 1})$.
\newline
To prove the last assertion, it suffices to notice that, since $\pi$ is a flop over $Y$ and $l, l^+$  are the associated flopping contractions, then 
$l(\Nklt(X, B , \mathbf M)) = 
l^+(\Nklt(X^+, B^+ , \mathbf  M))$.
Taking a resolution of indeterminacies of $\pi$
\begin{displaymath}
\xymatrix{
& W \ar[dr]^p \ar[dl]_q& \\
X  \ar@{-->}[rr]^\pi & & X_1
}
\end{displaymath}
where $\mathbf{M}$ descends and applying Proposition~\ref{p.conn.gen.pairs} twice, first to $h=l$, $g=q$, and then to $h=l^+$, $g=p$, we see that the number of connected components of 
$\Nklt(X^+, B^+, \mathbf M) \cap (f^{+})^{-1}(y)$ (resp. $\Nklt(X, B, \mathbf M) \cap f^{-1}(y)$)
is the same as the number of connected components of 
$\Nklt(Z, B_Z, \mathbf M) \cap (f_Z)^{-1}(y)$, 
where 
$B_Z \coloneqq l_\ast B= l_\ast^+ B^+$, as otherwise that situation would imply a lack of connectedness of the fibers of $(l^+ \circ p)\vert_{\Supp F} \colon \Supp F \to Z$ 
(resp. $(l \circ p)\vert_{\Supp F} \colon \Supp F \to Z$), 
where $F$ is the divisor defined in the statement of Proposition~\ref{p.conn.gen.pairs}.
\end{proof}

\begin{remark} \label{remark nklt in gdlt model}
Let $(X, B, \mathbf{M})$ be a generalized pair.
Let $(X^m,B^m, \mathbf{M})$ be a generalized pair together with a dlt modification $f^m \colon X^m  \rar X$ of $(X,B+M)$, as in Theorem~\ref{generalized dlt model}. 
Let us denote by $B'$ the unique boundary supported on $\Supp(B^m)$ and defined by the identity
\begin{align*}
\K X^m. + B' + M^m = f \sups m,*.(\K X. + B + M), \; M^m  \coloneqq  \mathbf{M}_{X^m}.
\end{align*}
Then, every non-klt center of $(X^m,B'+M^m)$ is contained in $\Supp((B') \sups \geq 1.)$. 
Assume not, and fix a non-klt center $W \subset X^m$ not contained in $\Supp((B') \sups \geq 1.)$. 
Then, $B'=B^m= (B^m)\sups <1.$ near the generic point of $W$. Since $(X,B^m+M^m)$ is   generalized dlt, $(X,(B^m)\sups <1.+M^m)$ is generalized klt. 
Thus, $(X^m,B' + M^m)$ is also klt at the generic point of $W$, which gives a contradiction.
\end{remark}

\section{Generalized log Calabi--Yau pairs and their structure}
\label{dual.struct.sect}

In this section, we collect some technical results related to the structure of generalized log Calabi--Yau pairs which will be useful in the analysis of dual complexes for such class of pairs.

\subsection{A Kawamata--Viehweg type result for generalized klt pairs}

In this subsection, we prove the following vanishing result of Kawamata--Viehweg type that will be used to show contractibility of certain types of dual complex for generalized log Calabi--Yau pairs.

\begin{theorem} \label{kvv}
Let $(X,B, \mathbf{M})$ be a generalized pair with generalized klt singularities. Let $L$ be a Cartier divisor on $X$ such that $H \coloneqq L - (\K X. + B + M)$ is nef and big. Then, $H^i(X,L)=0$ for $i > 0$.
\end{theorem}
\begin{proof}
Let $\pi \colon X' \rar X$ be a log resolution for $(X,B)$ where $\mathbf{M}$ descends.
Thus, we may write
\begin{align*}
\K X'. + \pi_* \sups -1. (B) + E' - F' + M' = \pi^\ast (\K X. + B + M), \; M' \coloneqq  \mathbf{M}_{X'},
\end{align*}
where $E' \geq 0$, $F' \geq 0$, $E' \wedge F' = 0$, $E'-F'$ is $\pi$-exceptional, and $\pi_* \sups -1. (B) + E' - F'$ has simple normal crossing support.
Thus,
\begin{align*}
\pi^\ast L + \lceil F' \rceil = \K X'. + \pi_* \sups -1. (B) + E' + (\lceil F' \rceil - F') + M' + \pi^\ast H.
\end{align*}
As $M' + \pi^\ast H$ is nef and big, by Kawamata--Viehweg vanishing~\cite{KM98}*{Theorem~2.64}, we have $H^i(X',\pi^\ast L + \lceil F' \rceil)=0$ for $i>0$.
Similarly, by~\cite{HK10}*{Theorem~3.45}, we have $R^i\pi_* \O X'.(\pi^\ast L + \lceil F' \rceil)=0$ for $i>0$.
Since $\lceil F' \rceil$ is $\pi$-exceptional, the projection formula implies $\pi_* \O X'.(\pi^\ast L + \lceil F' \rceil)=\O X. (L)$.
Hence, we conclude that $H^i(X,L)=0$ for $i>0$.
\end{proof}

\subsection{Special birational models}
In this subsection, we show how to reduce the problem of studying the dual complex of a generalized log Calabi--Yau pair to the log Fano case.
The following result is a generalization of~\cite{KX16}*{Theorem~49}.

\begin{theorem} 
\label{main technical thm for complexes}
Let $(X,B,\mathbf{M})$ be $\qq$-factorial generalized dlt pair.
Assume that $\K X. + B + M \sim_\qq 0$. 
Then, there exist a crepant birational map $\phi \colon X \drar \overline{X}$, a generalized pair $(\overline{X}, \overline{B}, \mathbf M)$, and a morphism $q \colon \overline{X} \rar Z$ such that:
\begin{itemize}
    \item[(1)] $\overline{B}\sups =1.$ fully supports a $q$-ample divisor;
    \item[(2)] every generalized log canonical center of $(\overline{X},\overline{B},\mathbf{M})$ dominates $Z$;
    \item[(3)] $\overline{E} \subset \overline{B} \sups =1.$ for every $\phi \sups -1.$-exceptional divisor $\overline{E} \subset \overline{X}$; and
    \item[(4)] $\phi \sups -1.$ is an isomorphism over $\overline{X} \setminus \overline{B} \sups =1.$.
\end{itemize}
\end{theorem}

\begin{remark} 
\label{remark about technical thm}
Properties (3) and (4) in Theorem~\ref{main technical thm for complexes} imply that $\Nklt(\overline{X},\overline{B},\mathbf{M})= \overline{B} \sups =1.$  and that $\overline{B} \geq 0$.
\end{remark}

\begin{proof}
If $(X,B, \mathbf M)$ is generalized klt, then it suffices to take $Z=\overline{X}=X$.
Hence, we can assume that $\Nklt(X,B, \mathbf M) \neq \emptyset$.

\medskip

{\bf Step 0}. {\it In this step we reduce to the case when the dual complex is connected.} 
\newline
If $\Nklt(X,B, \mathbf M)$ is disconnected, the statement follows from Theorem~\ref{main theorem} and its proof.
More precisely, as $(X,B,\bM.)$ is already $\mathbb Q$-factorial generalized dlt, we may apply the algorithm in the proof of Theorem~\ref{main theorem} to $X$: 
we run a $(K_X+B^{<1}+\bM X.)$-MMP $X \dashrightarrow \overline{X}$, which will terminate, by construction, with a Mori fibre space $\overline{X} \rar Z$, the morphism whose existence is claimed in the statement of~\ref{main technical thm for complexes}.
By Theorem~\ref{main theorem}, items (1)-(2) in the statement of~\ref{main technical thm for complexes} hold.
By construction, $X \drar \overline{X}$ is a birational contraction, so item (3) is satisfied.
Lastly, this MMP is a $(-B^{=1})$-MMP, and thus by the results in \S~\ref{sect.bir.case} also item (4) is satisfied.
Therefore, we can assume that $\Supp(B \sups =1.)$ is connected and non-empty.\\

{\bf Step 1}. 
{\it 
In this step, we construct a birational map $X \drar X'$ and a fibration $X' \rar Z$ satisfying properties (1), (3), and (4) in the statement of the theorem.
}
\newline
We run a $(\K X. + B \sups <1. +M)$-MMP with scaling of an ample divisor
\begin{equation}
\label{eqn.MMP}
\xymatrix{
X =X_0 \ar@{-->}[r]^{\pi_1} &X_1 \ar@{-->}[r]^{\pi_2} &\dots \ar@{-->}[r]^{\pi_n} & X_n= X'  \ar[r]^q & Z,
}
\end{equation}
and we define $B_i \coloneqq \pi_{i \ast} B_{i-1}$, and $B_0 \coloneqq B$.
As $\K X. + B \sups <1. + M \sim_\qq -B \sups =1.$, each step of the $(\K X. + B \sups <1. +M)$-MMP in~\eqref{eqn.MMP} is $(B \sups =1.)$-positive; 
moreover, this run of the MMP terminates with a Mori fiber space $q \colon X' \rar Z$ and we have an induced generalized pair $(X', B', \mathbf{M})$, $B' \coloneqq B_n$, such that $K_{X'}+B' +\mathbf{M}_{X'} \sim_\mathbb{Q} 0$.
As each step of~\eqref{eqn.MMP} is $(B \sups =1.)$-positive, the indeterminacy locus of the inverse map $\pi_i^{-1} \colon X_{i} \dashrightarrow X_{i-1}$ is contained in $B_i \sups =1.$.
This implies that property (4) in the statement of the theorem holds for $(X', B', \mathbf{M})$;
property (1) is satisfied as $q$ is a Mori fiber space, while (3) holds, by construction, since $X \dashrightarrow X'$ is a birational contraction.
\newline
If property (2) holds for $(X', B', \mathbf{M})$ then the proof of the theorem is completed.
Hence, we shall assume that not all generalized log canonical centers of $(X',B', \mathbf{M})$ dominate $Z$.\\

{\bf Step 2}. 
{\it In this step  assuming property (2) does not hold for $(X', B', \mathbf{M})$, we construct a birational map $X \dashrightarrow \widetilde{X}$ over $Z$, a generalized pair $(\widetilde X, \widetilde B, \mathbf{M})$ crepant to $(X,  B, \mathbf{M})$ and a commutative diagram 
\begin{align*}
\xymatrix{
\widetilde X \ar[rr]^{\widetilde q} \ar[dr]& & \widetilde Z \ar[dl]\\
& Z &
}, 
\end{align*} 
where $\widetilde{q}$ is a Mori fibration, $\widetilde{Z} \rar Z$ is a birational contraction and all the lc centers of $(\widetilde X, \widetilde B, \mathbf{M})$ that do not dominate $Z$ are fully supported on the pull-back of a divisor on $\widetilde Z$.
Moreover, $(\widetilde X, \widetilde B, \mathbf{M})$ satisfies properties (3), (4) and $\widetilde{B} \sups =1.$ fully supports an effective divisor that is relatively ample over $\widetilde{Z} \setminus \widetilde{q}(\widetilde{B} \sups =1,v.)$.
}
\newline
By construction, $(X', B', \mathbf{M})$ is $\qq$-factorial and generalized dlt away from $B'^{=1}$.
We pass to a generalized dlt model $\psi \colon X'' \rar X'$, 
\begin{align*}
K_{X''}+ B'' + M'' = \psi^\ast(K_{X'}+B'+ \mathbf{M}_{X'}), \ M'' \coloneqq \mathbf{M}_{X''}.
\end{align*}
Property (3) holds on $X''$ by construction.
The morphism $\psi$ is an isomorphism over $X' \setminus B'^{=1}$, hence property (4) holds on $X''$ as well.
Lemma~\ref{lemma patch} implies that property (1) is also preserved.
\newline
Thus, up to substituting $X'$ with $X''$, we can assume that $(X',B', \mathbf{M})$ is $\mathbb Q$-factorial and generalized dlt and that any of its log canonical places is a log canonical place for the pair $(X',B')$.
\newline
To simplify the notation, we write $M' \coloneqq \mathbf{M}_{X'}$.\\

{\bf Claim 1}.
{\it
Up to substituting $(X', B', \bf M)$ with a crepant birational generalized pair on a higher model of $X'$, we can assume that any generalized log canonical center of $(X', B', \bf M)$ vertical over $Z$ is contained in $B'^{=1, v}$ and properties (1), (3), (4) still hold on $X'$.\\
In particular, we can assume that for any $0 \leq \delta < 1$, the generalized pair $(X',B'^{<1} + B'^{=1, h} + \delta B'^{=1, v} + \overline{M})$ has no vertical log canonical centers.
}
\begin{proof}[Proof of Claim 1]
Let $W$ be a log canonical center for  $(X',B')$ vertical over $Z$.
If $W \not \subset B'^{=1, v}$, then $W$ is the intersection of components of $B'^{=1,h}$.
Let $\mathcal{C}$ be the subcomplex
of $\mathcal{DMR}(X',B'^{=1})$ whose cells correspond to all components of $B'^{=1}$ together with the strata of $B'^{=1,h}$ that dominate $Z$ and all strata of $B'^{=1,v}$. 
The subcomplex $\mathcal{C}$ is closed inside $\mathcal{DMR}(X',B'^{=1})$, as it is closed within each simplex of $\mathcal{DMR}(X',B'^{=1})$: in fact, if a stratum $W'$ of $\mathcal{DMR}(X',B'^{=1, h})$ does not dominate $Z$, the same holds for any stratum $W'' \subset W'$.
Thus, up to substituting $(X', B')$ with the crepant birational pair constructed by~\cite{KX16}*{Lemma~57}, we can assume that $B'^{=1,v}\neq 0$ and that $(X',B'^{<1} + B'^{=1,h},\bM.)$ has no vertical log canonical centers.
By Lemma~\ref{lemma patch} and~\cite{KX16}*{Lemma~57}, this process preserves properties (1), (3), and (4).
\end{proof}
\noindent
By construction $B'^{=1}$ fully supports an effective $\mathbb{Q}$-divisor $H$ ample over $Z$,
hence, for $0 < \epsilon \ll1$, the support of $B'^{=1} -\epsilon H$ coincides with that of $B'^{=1}$, and moreover $\lfloor B'^{=1} -\epsilon H \rfloor =0$, so that
\begin{align}
\label{eqn:thm.4.3.1}
\K X'. + B'^{<1} + (B'^{=1} -\epsilon H)^h + \mathbf{M}_{X'} + \epsilon H \sim_\mathbb{Q} -F', \ F' \coloneqq (B'^{=1} -\epsilon H)^v,
\end{align}
and $(X', B'^{<1} + (B'^{=1} -\epsilon H)^h + \mathbf{M}_{X'} + \epsilon H)$ is generalized klt over $Z$.
\newline
By~\cite{BZ16}*{Lemma~4.4}, we can run a $(\K X'. +  B'^{<1} + (B'^{=1} -\epsilon H)^h + \mathbf{M}_{X'} + \epsilon H)$-MMP with scaling of an ample divisor over $Z$, that terminates with a relatively good minimal model/Z 
\begin{align}
\label{eqn:MMP.4.3}
\xymatrix{
X' \ar@{-->}[rr] \ar[drr]  & &\widetilde X \ar[rr]^{\widetilde q} \ar[d] & & \widetilde Z \ar[dll]\\
& & Z & &
}
\end{align}
We denote by $\widetilde B, \widetilde F, \widetilde H$ the strict transforms of $B', F', H$ on $\widetilde X$.
The morphism $\widetilde Z \rar Z$ is birational, as $\overline F$ is vertical$/Z$.
Since the MMP in~\eqref{eqn:MMP.4.3} is positive for $F'$ by construction, then $\widetilde F > 0$ and $\widetilde F \sim_{\mathbb{Q}, \widetilde q} 0$.
The map $X \drar \widetilde X$ satisfies properties (3) and (4) in view of~\cite{KX16}*{Lemma~55} and the fact that $\Supp(F') \subset \Supp(B'^{=1})$. \\
On the other hand, property (1) may no longer hold on $\widetilde X$.
Nevertheless, the following claim holds.\\

{\bf Claim 2}. 
{\it
$\widetilde{B} \sups =1.$ fully supports an effective divisor that is relatively ample over $\widetilde{Z} \setminus \widetilde{q}(\widetilde{B} \sups =1,v.)$.}

\begin{proof}[Proof of Claim 2]
As we modified $X'$ by means of an MMP for a divisor supported on $(B') \sups =1,v.$, property (1) still holds true over the complement of $q((B') \sups =1,v.)$.
\end{proof}

{\bf Step 3}. 
{\it 
In this step  we show that there exist generalized pairs $(\widetilde Z, B_{\widetilde Z}, \mathbf{M}_{\widetilde Z})$, $(\widetilde Z, \Delta_{\widetilde Z}, \mathbf{M}_{\widetilde Z})$ on $\widetilde Z$ such that $K_{\widetilde Z}+ B_{\widetilde Z}+ \mathbf{M}_{\widetilde Z} \sim_\mathbb{Q} 0$ and $-(K_{\widetilde Z}+ \Delta_{\widetilde Z}+ \mathbf{M}_{\widetilde Z})$ is pseudoeffective.
}
\newline
The generalized pair $(\widetilde{X},\widetilde B \sups <1. + \widetilde B \sups =1,h. + \epsilon \widetilde H^v, \mathbf{M})$ has no log canonical centers vertical over $\widetilde Z$, and
\begin{align*}
\label{eqn.hv}
\K \widetilde X. + \widetilde B \sups <1. + \widetilde B \sups =1,h. + \epsilon \widetilde H^v + \mathbf{M}_{\widetilde X} \sim \subs \qq,\widetilde q. 0,
\end{align*}
by~\eqref{eqn:thm.4.3.1}.
Theorem~\ref{generalized canonical bundle formula} implies the existence of generalized pairs $(\widetilde Z, B_{\widetilde Z}, \mathbf{M}_{\widetilde Z})$, $(\widetilde Z, \Delta_{\widetilde Z}, \mathbf{M}_{\widetilde Z})$ such that 
\begin{eqnarray}
\label{eqn:cbf.tildez}
0\sim_\mathbb{Q} \K \widetilde X. + \widetilde B + \mathbf{M}_{\widetilde X} &\sim \subs \qq.  \widetilde q ^\ast (\K \widetilde Z. + B \subs \widetilde Z. + \mathbf M \subs \widetilde Z.), \\
\nonumber
\K \widetilde X. + \widetilde B \sups <1. + \widetilde B \sups =1,h. + \epsilon \widetilde H^v + \mathbf{M}_{\widetilde X} &\sim \subs \qq.  \widetilde q^\ast (\K \widetilde Z. + \Delta \subs \widetilde Z. + \mathbf M \subs \widetilde Z.).
\end{eqnarray}
The two generalized pairs in~\eqref{eqn:cbf.tildez} have the same moduli b-divisor $\mathbf{M}_{\widetilde{Z}}$, since
\begin{align*}
(\K \widetilde X. + \widetilde B + \mathbf{M}_{\widetilde X}) \vert_{\widetilde{X}_\eta}= (\K \widetilde X. + \widetilde B \sups <1. + \widetilde B \sups =1,h. + \epsilon \widetilde H^v + \mathbf{M}_{\widetilde X})\vert_{\widetilde{X}_\eta},
\end{align*}
along the generic fiber $\widetilde{X}_\eta$ of $\widetilde q$.
Furthermore, 
\begin{align*}
B \subs \widetilde Z. - \Delta \subs \widetilde Z. \geq 0 \ 
{\rm and} \
\widetilde q ^\ast (B \subs \widetilde Z. - \Delta \subs \widetilde Z.)= \widetilde F,
\end{align*}
since $\widetilde B \sups =1,v. = \epsilon \widetilde H^v  + \widetilde F$.
As no log canonical center of $(\widetilde{X},\widetilde B \sups <1. + \widetilde B \sups =1,h. + \epsilon \widetilde H^v, \mathbf{M})$ is vertical over $\widetilde Z$, the generalized pair $(\widetilde Z, \Delta \subs \widetilde Z. + \mathbf M \subs \widetilde Z.)$ is generalized klt, cf.~\cite{Amb99}*{proof of Proposition 3.4}.
Hence,
\begin{eqnarray}
\label{eqn.zhat}
\K \widetilde Z. + \Delta \subs \widetilde Z. + \mathbf M \subs \widetilde Z.&\sim_\qq \Delta \subs \widetilde Z. - B \subs \widetilde Z.,\\
\nonumber
K \subs \widetilde X. + \widetilde B \sups <1. + \widetilde B \sups =1,h. + \epsilon \widetilde H^v + \mathbf{M}_{\widetilde X} &\sim \subs \qq. \widetilde{q}^\ast(\Delta \subs \widetilde Z. - B \subs \widetilde Z.),
\end{eqnarray}
which concludes the proof of this step.\\

{\bf Step 4}. 
\emph{
In this step  we obtain a birational contraction $\widetilde{Z} \drar \widehat Z$ together with a Mori fiber space $\widehat Z \rar W$.
Furthermore, we show that we can lift the birational contraction $\widetilde{Z} \drar \widehat Z$ to a birational contraction $\widetilde X \drar \widehat X$ and $\widehat X$ is endowed with a morphism $\widehat X \rar \widehat Z$.
}
\newline
By construction, $(\widetilde X,\widetilde B \sups <1. + \widetilde B \sups =1,h. + \epsilon \widetilde H^v,\bM.)$ satisfies the following properties:
\begin{itemize}
    \item[(a)] it is generalized dlt;
    \item[(b)] all of its log canonical centers dominate $\widetilde Z$; and
    \item[(c)] $\widetilde B \sups =1,h.$ fully supports a divisor $\widetilde H ^h$ that is ample over $\widetilde Z \setminus \Supp(B \subs \widetilde Z. - \Delta \subs \widetilde Z.)$.
\end{itemize}
{\bf Claim 3}. 
{\it
The following properties hold:
 \begin{itemize}
     \item[(d)] 
$(\widetilde X,\widetilde B \sups <1. + \widetilde B \sups =1,h. - \delta \widetilde H ^h + \epsilon \widetilde H^v,\bM.)$ is generalized klt, where $0 < \delta \ll 1$;

\item[(e)]
for a general element $\widetilde D \in |\widetilde H ^h/\widetilde Z|_\qq$, the generalized pair $(\widetilde X,\widetilde B \sups <1. + \widetilde B \sups =1,h. - \delta \widetilde H ^h + \epsilon \widetilde H^v + \delta \widetilde D,\bM.)$ is generalized klt, and
\begin{align}
\K \widetilde X. + \widetilde B \sups <1. + (\widetilde B \sups =1,h. - \delta \widetilde H ^h) + \epsilon \widetilde H^v + \delta \widetilde D + \bM \widetilde X. \sim_{\mathbb{Q}, \widetilde{Z}} \widetilde{q}^\ast (\Delta \subs \widetilde Z. - B \subs \widetilde Z.).    
\end{align}
\end{itemize}
}
\begin{proof}[Proof of Claim 3]

\begin{itemize}
    \item[(d)] 
This follows at once from properties (a-b) stated right before Claim 3.

    \item[(e)]
As $\widetilde H ^h$ is ample over $\widetilde Z \setminus \Supp(B \subs \widetilde Z. - \Delta \subs \widetilde Z.)$ by Claim 2, the relative $\qq$-linear series $|\widetilde H ^h/\widetilde Z|_\qq$ is free over $\widetilde Z \setminus \Supp(B \subs \widetilde Z. - \Delta \subs \widetilde Z.)$.
In particular, a sufficiently general member $\widetilde{D} \in |\widetilde H ^h/\widetilde Z|_\qq$ avoids the generic point of every log canonical center of $(\widetilde X,\widetilde B \sups <1. + \widetilde B \sups =1,h. + \epsilon \widetilde H^v,\bM.)$, as these all dominate $\widetilde Z$.
The generalized log canonical centers of $(\widetilde X,\widetilde B \sups <1. + \widetilde B \sups =1,h. - \delta \widetilde H ^h + \epsilon \widetilde H^v + \delta \widetilde D,\bM.)$ are a subset of the ones of  $(\widetilde X,\widetilde B \sups <1. + \widetilde B \sups =1,h. + \epsilon \widetilde H^v,\bM.)$, cf.~\cite{KM98}*{Corollary~2.33}.
Thus, as $\widetilde D$ avoids the log canonical centers that dominate $\widetilde Z$, the conclusion follows.
\end{itemize}
This concludes the proof.
\end{proof}
We run a $(\K \widetilde Z. + \Delta \subs \widetilde Z. + M \subs \widetilde Z.)$-MMP with scaling of an ample divisor 
\begin{align}
\label{eqn:MMP.Z.4.3}
\xymatrix{
\widetilde X \ar[d]^{\widetilde q} \\
\widetilde Z \coloneqq \widetilde{Z}_0 \ar@{-->}[r]^{\widetilde{\psi}_{1}} &
\widetilde{Z}_1 \ar@{-->}[r]^{\widetilde{\psi}_{2}} &
\dots \ar@{-->}[r]^{\widetilde{\psi}_{n-1}} &
\widetilde{Z}_{n-1} \ar@{-->}[r]^{\widetilde{\psi}_{n}} &
\widetilde{Z}_n \eqqcolon \widehat Z \ar[d]^r 
\\
& & & & W 
}
\end{align}
that terminates with a Mori fiber space $r \colon \widehat Z \rar W$.
We define
\begin{align*}
&\Delta \subs \widetilde{Z}_{0}. \coloneqq \Delta \subs \widetilde{Z}., \
M \subs \widetilde{Z}_{0}. \coloneqq M \subs \widetilde Z.,
\widetilde F \subs 0. \coloneqq \widetilde{F},
\\
& \Delta \subs \widetilde{Z}_{i+1}. \coloneqq \widetilde{\psi}_{i+1} \Delta \subs \widetilde{Z}_{i}., \
M \subs \widetilde{Z}_{i+1}. \coloneqq \widetilde{\psi}_{i+1} M \subs \widetilde{Z}_{i}. .
\end{align*}

{\bf Claim 4.}
{\it Let us define $(X_0, B_0, \mathbf M)\coloneqq(\widetilde{X}, \widetilde{B}, \mathbf{M})$.
For any $0\leq i \leq n$ there exists a $\mathbb{Q}$-factorial generalized lc pair $(\widetilde{X}_{i}, \widetilde{B}_i, \mathbf{M})$ with $\widetilde{B}_i \neq 0$, and morphisms $m_i \colon X_i \rar Z_i$ such that the following diagram commutes
\begin{align*}
\label{eqn:MMP.Z.4.3b}
\xymatrix{
\widetilde{X}_0
\ar@{-->}[r]^{\widetilde{\pi}_{1}} \ar[d]_{m_0 \coloneqq \widetilde{q}}&
\widetilde{X}_1 \ar@{-->}[r]^{\widetilde{\pi}_{2}} \ar[d]^{m_1} &
\dots \ar@{-->}[r]^{\widetilde{\pi}_{n-1}} &
\widetilde{X}_{n-1} \ar@{-->}[r]^{\widetilde{\pi}_{n}} \ar[d]^{m_{n-1}}&
\widetilde{X}_n \eqqcolon \widehat X \ar[d]^{m_n \eqqcolon \widehat{q}} \\
\widetilde{Z}_0 \ar@{-->}[r]^{\widetilde{\psi}_{1}} &
\widetilde{Z}_1 \ar@{-->}[r]^{\widetilde{\psi}_{2}} &
\dots \ar@{-->}[r]^{\widetilde{\psi}_{n-1}} &
\widetilde{Z}_{n-1} \ar@{-->}[r]^{\widetilde{\psi}_{n}} &
\widetilde{Z}_n \eqqcolon \widehat Z \ar[d]^r
\\
& & & & W 
}
\end{align*}
each $\widetilde{\pi}_i$ is a birational contraction, $\widetilde{m}_i$ is a contraction satisfying properties (a-e) in Claim 3, 
Moreover, for any $0 \leq i \leq n$, $\Nklt(\widetilde{X}_{i}, \widetilde{B}_i, \mathbf{M}) = \Supp(\widetilde B \sups =1.\subs i.)$
}.
\begin{proof}
We use induction on $i \in \{0, 1, \dots, n\}$.
We define $\widetilde F \subs i.\coloneqq \widetilde \pi \subs i \ast . \widetilde F \subs i-1.$.\\
For $i=0$, there is nothing to prove.
Thus, we can assume that the statement of the claim holds up to $i-1$ and we shall prove that it holds for $i$ as well.
\\
Hence it suffices to show that the diagram 
\begin{eqnarray}
\label{eqn:incom.diag}
\xymatrix{
\widetilde{X}_{i-1} \ar[d]^{m_i}
\\
\nonumber
\widetilde{Z}_{i-1} \ar@{-->}[r]^{\widetilde{\psi}_{i}} & 
\widetilde{Z}_{i}
}
\end{eqnarray}
can be completed to a diagram 
\begin{eqnarray}
\label{eqn:incom.diag2}
\xymatrix{
\widetilde{X}_{i-1} \ar[d]^{m_{i-1}} \ar@{-->}[r]^{\widetilde{\pi}_{i}} &
\widetilde{X}_{i} \ar[d]^{m_i}
\\
\widetilde{Z}_{i-1} \ar@{-->}[r]^{\widetilde{\psi}_{i}} & 
\widetilde{Z}_{i}
}
\end{eqnarray}
\noindent
satisfying the conditions in the statement of the claim.
If $\widetilde{\psi}_i \colon \widetilde{Z}_{i-1} \rar \widetilde{Z}_i$ is a divisorial contraction, then for a general choice of $\widetilde{D}_{i-1} \in |\widetilde{H}^h_{i-1}/\widetilde{Z}_{i}|_\mathbb{Q}$, in view of property (e) of Claim 3, the generalized pair $(\widetilde{X}_{i-1}, \widetilde B \sups <1. \subs i-1. + (\widetilde B \sups =1,h. \subs i-1. -\delta \widetilde H \subs i-1. \sups h.) + \epsilon \widetilde H^v_{i-1}+\delta \widetilde{D}_{i-1}, \mathbf{M})$ is generalized klt and
\begin{equation}
\label{eqn:lin.eq.flip.claim}    
\K \widetilde{X}_{i-1}.+ \widetilde B \sups <1. \subs i-1. + (\widetilde B \sups =1,h. \subs i-1. -\delta \widetilde H \subs i-1. \sups h.) + \epsilon \widetilde H^v_{i-1} + \delta \widetilde{D}_{i-1}+ \mathbf{M} \subs \widetilde{X}_{i-1}. \sim_{\widetilde{Z}_i, \mathbb{Q}} m \subs i-1. \sups \ast. (\widetilde{\Delta}_{\widetilde{Z} \subs i-1.} - B_{\widetilde{Z} \subs i-1.}).
\end{equation}
By~\cite{BZ16}*{Lemma~4.4}, we can run a $(\K \widetilde{X}_{i-1}.+ \widetilde B \sups <1. \subs i-1. + (\widetilde B \sups =1,h. \subs i-1. -\delta \widetilde H \subs i-1. \sups h.) + \epsilon \widetilde H^v_{i-1}+ \delta \widetilde{D}_{i-1}+ \mathbf{M} \subs \widetilde{X}_{i-1}.)$-MMP over $Z_i$ with scaling of an ample divisor and this terminates, as $\widetilde D \subs i-1.$ is big over $Z_{i}$.
Thus, we can define $\widetilde{X}_i$  in~\eqref{eqn:incom.diag2} to be a good minimal model for this $(\K \widetilde{X}_{i-1}.+ \widetilde B \sups <1. \subs i-1. + (\widetilde B \sups =1,h. \subs i-1. -\delta \widetilde H \subs i-1. \sups h.) + \epsilon \widetilde H^v_{i-1} + \delta \widetilde{D}_{i-1}+ \mathbf{M} \subs \widetilde{X}_{i-1}.)$-MMP over $\widetilde{Z}_i$.
As $\widetilde F \subs i-1. = m^\ast \subs i-1. (B_{\widetilde{Z}_{i-1}} - \Delta_{\widetilde{Z}_{i-1}})$,~\eqref{eqn:lin.eq.flip.claim} implies that each step of the MMP above is $\widetilde F \subs i-1.$-positive.
Hence, the exceptional locus of $\widetilde \pi \sups -1. \subs i.$ must be contained in $\widetilde F \subs i.$, and $\widetilde \pi \subs i.$ is an isomorphism above $\Supp(B \subs \widetilde{Z}_i. - \Delta \subs \widetilde{Z}_i.)$, so that properties (a-e)
also hold for the morphism $m_i$.
\\
Similarly, if $\widetilde{\psi}_i$ is a flipping contraction
\begin{align*}
\xymatrix{
\widetilde{Z}_{i-1} \ar@{-->}[rr]^{\widetilde{\psi}_i} \ar[dr]_{p_i} & & \widetilde{Z}_i \ar[dl]^{p'_i} \\
& Z_i &
}
\end{align*}
for a general choice of $\widetilde{D}_{i-1} \in |\widetilde{H}^h_{i-1}/ Z_{i}|_\mathbb{Q}$ the generalized pair $(\widetilde{X}_{i-1}, \widetilde B \sups <1. \subs i-1. + (\widetilde B \sups =1,h. \subs i-1. -\delta \widetilde H \subs i-1. \sups h.) + \epsilon \widetilde H^v_{i-1}+ \delta \widetilde{D}_{i-1}, \mathbf{M})$ is klt and
\begin{align}
\label{eqn:lin.eq.flip.claim2}
\K \widetilde{X}_{i-1}.+ \widetilde B \sups <1. \subs i-1. + (\widetilde B \sups =1,h. \subs i-1. -\delta \widetilde H \subs i-1. \sups h.) + \epsilon \widetilde H^v_{i-1}+ \delta \widetilde{D}_{i-1}+ \mathbf{M} \subs \widetilde{X}_{i-1}. \sim_{ Z_i, \mathbb{Q}} m \subs i-1. \sups \ast. (\widetilde{\Delta}_{\widetilde{Z} \subs i-1.} - B_{\widetilde{Z} \subs i-1.}).
\end{align}
By~\cite{BZ16}*{Lemma~4.4}, we can run a $(\K \widetilde{X}_{i-1}.+ \widetilde B \sups <1. \subs i-1. + (\widetilde B \sups =1,h. \subs i-1. -\delta \widetilde H \subs i-1. \sups h.) + \epsilon \widetilde H^v_{i-1}+ \delta \widetilde{D}_{i-1}+ \mathbf{M} \subs \widetilde{X}_{i-1}.)$-MMP over $Z_i$ with scaling of an ample divisor and this terminates, as $\widetilde D \subs i-1.$ is big over $Z_{i}$.
Thus, we can define $\widetilde{X}_i$  in~\eqref{eqn:incom.diag2} to be a good minimal model for this $(\K \widetilde{X}_{i-1}.+ \widetilde B \sups <1. \subs i-1. + (\widetilde B \sups =1,h. \subs i-1. -\delta \widetilde H \subs i-1. \sups h.) + \epsilon \widetilde H^v_{i-1}+ \delta \widetilde{D}_{i-1}+ \mathbf{M} \subs \widetilde{X}_{i-1}.)$-MMP over $Z_i$, since
$\widetilde{Z}_i= {\rm Proj}_{\mathcal{O}_{Z_i}}(\bigoplus_
{m \geq 0} p_{i\ast}'\mathcal{O}_{Z_i}(m(\Delta_{\widetilde{Z}_i} - B_{\widetilde{Z}_i})))$ and by~\eqref{eqn:lin.eq.flip.claim2}.
As $\widetilde F \subs i-1. = m^\ast \subs i-1. (B_{\widetilde{Z}_{i-1}} - \Delta_{\widetilde{Z}_{i-1}})$,~\eqref{eqn:lin.eq.flip.claim2} implies that each step of the MMP above is $\widetilde F \subs i-1.$-positive.
Hence, the exceptional locus of $\widetilde \pi \sups -1. \subs i.$ must be contained in $\widetilde F \subs i.$, and $\widetilde \pi \subs i.$ is an isomorphism above $\Supp(B \subs \widetilde{Z}_i. - \Delta \subs \widetilde{Z}_i.)$, so that properties (a-d) also hold for the morphism $m_i$.
\end{proof}
\noindent
We define divisors $\widehat F, \widehat H$ to be the strict transforms on $\widehat X$ of $F, H$, respectively;
thus, $\widehat F = \widehat q ^\ast  (B \subs \widehat Z. - \Delta \subs \widehat Z.)$.
By construction, $\widetilde Z \drar \widehat Z$ is an isomorphism outside of $\Supp (B \subs \widehat Z. - \Delta \subs \widehat Z.)$; thus,~\cite{KX16}*{Lemma~54}\footnote{We note that the cited result is proven for $\mathbb Q$-factorial dlt pairs, while we use it in the context of $\mathbb Q$-factorial generalized klt pairs. Since, by our constructions, the generalized klt pair has a relatively big boundary, then by standard perturbation arguments, the generalized klt pair is equivalent over the base to a klt pair.} implies that the rational map $\widetilde X \drar \widehat X$ is an isomorphism over the complement of $\Supp (B \subs \widehat Z. - \Delta \subs \widehat Z.)$.
As $\Supp(\widehat F) = \Supp (\widehat B \sups =1,v.)$, $X \drar \widehat X$ satisfies properties (3) and (4).\\

{\bf Step 5}. \emph{In this step we modify $\widehat X$ so that property (1) is achieved for $\widehat X \rar W$.}
\newline
By construction, $\widehat B \sups =1.$ fully supports a divisor $\widehat H$ that is $\widehat q$-ample over $\widehat Z \setminus \Supp (B \subs \widehat Z. - \Delta \subs \widehat Z.)$.
As $r$ is a $(B \subs \widehat Z. - \Delta \subs \widehat Z.)$-Mori fiber space, then $\Supp (B \subs \widehat Z. - \Delta \subs \widehat Z.)$ fully supports an $r$-ample divisor $H \subs \widehat Z.$.
Since $\widetilde q \sups -1. (\Supp(B \subs \widetilde Z. - \Delta \subs \widetilde Z.)) = \Supp(\widetilde B \sups =1,v.)$, by construction of the birational contractions $\widetilde Z \drar \widehat Z$, $\widetilde X \drar \widehat X$ and $\qq$-factoriality of all varieties involved, it follows that
$\widehat q \sups -1. (\Supp (B \subs \widehat Z. - \Delta \subs \widehat Z.)) = \Supp (\widehat B \sups =1,v.)$.
Thus, both $\widehat H$ and $\widehat q ^\ast  H \subs \widehat Z.$ are supported on $\widehat B \sups =1.$.
Furthermore, since $\widehat H$ is fully supported on $\widehat B \sups =1.$, it follows that $\widehat H + m \widehat q ^\ast  H \subs \widehat Z.$ is fully supported on $\widehat B \sups =1.$ for every $m \geq 0$.
Similarly, since $\widehat H ^h$ is fully supported on $\widehat B \sups =1,h.$ and $\widehat q ^\ast  H \subs \widehat Z.$ is fully supported on $\widehat B \sups =1,v.$, $\widehat H ^h + m \widehat q ^\ast  H \subs \widehat Z.$ is fully supported on $\widehat B \sups =1.$ for every $m > 0$.
\newline
If $\widehat H$ or $\widehat H ^h$ are $\widehat q$-ample, then $\widehat H + m \widehat q ^\ast  H \subs \widehat Z.$ is $r \circ \widehat q$-ample for $m \gg 0$, and property (1) holds on $\widehat X$.
\newline
Thus, we can assume that $\widehat H ^h$ is not $\widehat q$-ample: we will construct a new birational model of $\widehat{X}$ where relative ampleness of $\widehat H ^h$ over $\widehat{Z}$ is achieved.
In Step 4 we showed that there exists an effective $\widehat{D}\sim \subs \qq,\widehat{Z}. \widehat{H}^h$ such that
\begin{itemize}
    \item[(f)] for $0 < \delta \ll 1$, $(\widehat X, \widehat B \sups <1. + (\widehat B \sups =1,h. -\delta \widehat{H}^h) + \epsilon \widehat H^v + \delta \widehat{D}, \mathbf{M})$ is generalized klt; and
    \item[(g)] by (f) for $0< \sigma \ll \delta \ll 1$, then $(\widehat X, \widehat B \sups <1. + (\widehat B \sups =1,h. -\delta \widehat{H}^h) + \epsilon \widehat H^v + \delta \widehat{D}+\sigma\widehat{D}, \mathbf{M})$ is generalized klt; moreover,
    \begin{align} 
\label{eqtn_in_blu}
\K \widehat{X}. + \widehat B \sups <1. + (\widehat B \sups =1,h. - \delta \widehat H ^h) + \epsilon \widehat H^v + \delta \widehat D +\sigma\widehat{D} + \bM \widehat{X}. \sim \subs \qq,\widehat{Z}. \sigma \widehat{H}^h,
    \end{align}
and $\widehat{H}^h$ is relatively ample over $\widehat{Z} \setminus \Supp(B \subs \widehat{Z}. - \Delta \subs \widehat{Z}.)$.

\end{itemize}

By~\eqref{eqtn_in_blu} and~\cite{BZ16}*{Lemma~4.4}, if we run a relative MMP with scaling of an ample divisor over $\widehat{Z}$ for $(\widehat X, \widehat B \sups <1. + (\widehat B \sups =1,h. - \delta \widehat H ^h) + \epsilon \widehat H^v + \delta \widehat D +\sigma\widehat{D}, \mathbf{M})$,
this must terminate with a relative good minimal model $\check q ' \colon (\check X ', \check B ', \bM.) \rar \widehat{Z}$ over $\widehat{Z}$; 
we denote by $\check q \colon (\check X, \check B,\bM.) \rar \widehat{Z}$ the corresponding relative ample model.
The MMP that we have just described is an isomorphism over $\widehat{Z} \setminus \Supp(B \subs \widehat{ Z}. - \Delta \subs \widehat {Z}.)$.
The following properties hold:
\begin{itemize}
    \item 
    $\widehat X \drar \check X '$ and $\widehat X \drar \check X$ are isomorphisms over $\widehat Z \setminus \Supp (B \subs \widehat Z. - \Delta \subs \widehat Z .)$, hence, properties (3) and (4) are preserved;
    \item 
    $(\check X ', \check B ', \bM.)$ and $(\check X , \check B , \bM.)$ are generalized log canonical, and the former has $\qq$-factorial singularities.
    Furthermore, by Step 4 and the fact that $\widehat{X} \drar \check{X}'$ is an isomorphism over $\widehat{Z} \setminus \Supp(B \subs \widehat{Z}. - \Delta \subs \widehat{Z}.)$, $(\check X ', \check B ', \bM.)$ is generalized dlt over $\widehat{Z} \setminus \Supp(B \subs \widehat{Z}. - \Delta \subs \widehat{Z}.)$;
    \item
    $(\check B ')\sups =1,h.$ fully supports a divisor $(\check H ')^h$ that is relatively big and semi-ample over $\widehat{Z}$, and its push-forward to $\check X$, denoted by $\check H ^h$, is relatively ample over $\widehat{Z}$.
    As $r$ is a $(\Delta \subs \widehat{Z}. - B \subs \widehat{Z}.)$-Mori fiber space, $\Supp(B \subs \widehat{Z}.-\Delta \subs \widehat{Z}.)$ fully supports an $r$-ample divisor $H \subs \widehat{Z}.$; 
    as $\widehat{q}\sups -1.(\Supp(B \subs \widehat{Z}.-\Delta \subs \widehat{Z}.))=\Supp((\check B ')\sups =1,v.)$, then $(\check B ')\sups =1.$ fully supports a divisor that is big and semi-ample over $W$, which we denote by $\check H '$.
    Similarly, $\check B \sups =1.$ fully supports a divisor that is ample over $W$, which we denote by $\check H$.
    Furthermore, by construction, $\check H '$ is relatively ample away from $\Nklt(\check X ', \check B ', \bM.)$:
    namely, $\check X ' \drar \mathrm{Proj}_W \bigoplus_m (r \circ \check q ') _*(\O \check X '. (m \check H '))=\check X$ is a birational morphism whose exceptional locus is contained in $\Nklt(\check X ', \check B ', \bM.)$.
\end{itemize}

{\bf Claim 5}.
{\it
By construction, we have $\Nklt(\check X ', \check B ', \bM.) = \Supp((\check B')\sups =1.)$.
Furthermore, we may take a generalized dlt model $(\check X ^m, \check B ^m, \bM.) \rar (\check X ', \check B ', \bM.)$ such that properties (3) and (4) are preserved, and $\Supp((B^m)\sups =1.)$ fully supports a divisor that is relatively big and semi-ample over $\widehat Z$, and that is relatively ample away from $\Nklt(\check X ^m, \check B ^m, \bM.)$.
}
\begin{proof}
By construction, $\check X '$ is $\qq$-factorial.
By property (f), $(\widehat X, \widehat B \sups <1. + \widehat B \sups =1,h. + \epsilon \widehat H^v + \delta (\widehat{D}-\widehat{H}^h), \mathbf{M})$ is generalized klt and its generalized log canonical divisor is relatively trivial over $\widehat{Z}$.
Thus, it follows that $(\check X ', (\check B ') \sups <1. + (\check B ')\sups =1,h. + \epsilon (\check H ')^v + \delta (\check{D}'-(\check{H}')^h), \bM.)$ is generalized klt.
Since $\check X '$ is $\qq$-factorial, $(\check X ', (\check B ') \sups <1., \bM.)$ is generalized klt.
Thus, by Lemma~\ref{lemma supporto}, if we take a generalized dlt model $(\check X ^m, \check B ^m, \bM.)$ of $(\check X ', \check B ', \bM.)$, every divisor extracted has positive coefficient in the pull-back of $(\check B ')\sups =1.$.
Then, since $\Supp((\check B')\sups =1.)$ fully supports a divisor that is relatively big and semi-ample over $\widehat{Z}$ and that is relatively ample away from $\Nklt(\check X ' , \check B ', \bM.)$, then so does $\Supp((B^m)\sups =1.)$.
Indeed, every divisor extracted by $\check X ^m \rar \check X '$ appears with positive coefficient in the pull-back of $\check H '$.
Furthermore, passing to the model $\check X ^m$ preserves properties (3) and (4).
\end{proof}
{\bf Step 6}. \emph{Conclusion}.
\newline
By construction, the generalized pair $(\check X, \check B,\mathbf{M})$ together with the morphism $r \circ \check q \colon \check X \rar W$ satisfies the properties (1), (3) and (4).
If property $(2)$ holds for $\check X \rar W$, then we can stop.
Otherwise, we perform the following procedure:
\begin{enumerate}
\item[(i)]
we replace $(\check X, \check B + \check M)$ with a generalized dlt model $(\check X ^m, \check B ^m, \bM.)$ as at the end of Step 5;

\item[(ii)] 
we repeat the arguments in Steps (2-5) for $(\check X ^m, \check B ^m , \bM.)$ and $\check X ^m \rar W$, thus obtaining a new birational model $\check X'' \rar V$ satisfying properties (1), (3), (4) with $\dim V < \dim W$;

\item[(iii)]
if property $(2)$ holds for $\check X'' \rar V$, then we can stop; otherwise, we restart from (i) with input data $(\check X'', \check B'' , \bM.)$ and $\check X'' \rar V$;

\end{enumerate}
The procedure described in (i)-(iii) must terminate after finitely many iterations as at each time the dimension of the base of the fibration drops:
the final outcome will then satisfy property (2).
\\
To conclude our proof, we need to show that, in the procedure (i-iii) just introduced, whenever in (i) we take a generalized dlt model of $(\check X, \check B, \mathbf{M})$, then this does not constitute an issue with respect to property (1).
That is, if we replace $(\check X, \check B, \mathbf{M})$ with a generalized dlt model, property (1) may not hold anymore.
On the other hand, we will argue that a suitable weaker version $(1')$ of property (1) still holds and that this property is enough to run the proof.
\\
Let $(\check X ^m, \check B ^m, \bM.)$ be the generalized dlt model of $(\check X , \check B , \bM.)$ introduced in Step 5.
Recall that the morphism $\check X ^m \rar \check X$ preserves (3) and (4) for the rational map $X \drar \check X ^m$.
Recall that
$(\check B ^m)\sups =1.$ fully supports a divisor $\check H ^m$ that is big and semi-ample relatively to $W$.
Furthermore, $\check H ^m$ is relatively ample away from $\Nklt(\check X ^m, \check B ^m, \bM.)$.
\newline
Now, consider the following weaker version of property (1):
\begin{itemize}
    \item[($1'$)] $(\check B ^m)\sups =1.$ fully supports a divisor that is relatively big and semi-ample over $W$. Furthermore, this divisor is ample over $W$ away from $\Nklt(\check X ^m, \check B ^m, \bM.)$.
\end{itemize}
Notice that ($1'$) is stable under the extraction of  valuations of generalized log discrepancy $0$ for $(X, B, \mathbf{M})$.
Thus, we are free to replace $\check X '$ with a higher generalized dlt model.
In particular, ($1'$) is stable under the operations performed in Claim 1.
Furthermore, for $0 < \epsilon \ll 1$ $(\check X ^m, (\check B ^m)\sups < 1. + ((\check B ^m)\sups =1. - \epsilon \check H ^m)\sups h. +  \epsilon \check H ^m,\bM.)$ is generalized klt, and satisfies the hypotheses of~\cite{BZ16}*{Lemma~4.4}.
In particular, we can repeat Step 2, and produce new birational models $\widetilde{X}'$ and $\widetilde{W}$ of $\check X '$ and $W$, respectively.
These two models satisfy the conclusions of Step 2, besides that the conclusion of Claim 2 is replaced by its weaker version with property ($1'$).
\newline
Notice that we can reproduce Step 3 with the varieties $\widetilde{X}'$ and $\widetilde{W}$ with no significant chages.
\newline
Now, we would like to reproduce Step 4.
Notice that the key facts for Step 4 to hold are properties (d-e).
Indeed, these guarantee that we can apply~\cite{KX16}*{Lemma~54}.
By construction, properties (a-b) are still satisfied by $\widetilde{X}'$ and $\widetilde{W}$, while (c) does not hold, as (c) is equivalent to the conclusion of Claim 2.
On the other hand, since a weaker version of Claim 2 where ampleness is replaced by bigness and semi-ampleness holds, the corresponding weaker version of (c) holds.
Now, notice that this weakening of (c) is sufficient in the proof of Claim 3, as it is enough to be able to apply a Bertini-like argument.
In particular, this implies that we can replicate Step 4 with $\widetilde{X}'$ and $\widetilde{W}$.
Hence, we produce a variety $\check X ''$ which is birational to $X$, and a variety $V$ with $\dim(V) < \dim(W)$.
\newline
Finally, we can reproduce Step 5 in its entirety.
More precisely, Step 5 can be reproduced replacing property (1) with property ($1'$).
Then, the conclusions of Step 5 still hold.
Therefore, as indicated in (iii) in Step 6, we can iterate this argument until property (2) is satisfied.
Since after each iteration the dimension of the base drops, the algorithm has to terminate, and the final outcome will satisfy properties (1-4).
\end{proof}

The following is a generalization of~\cite{KX16}*{Corollary~58}.
We will use the same notation as in Theorem~\ref{main technical thm for complexes}.

\begin{corollary} 
\label{corollary reduction to fibers}
Let $(X,B,\mathbf{M})$ be a generalized log canonical pair.
Assume that $\K  X. + B + \mathbf{M}_X \sim_\qq 0$.
Then, there exist a $\qq$-factorial, generalized dlt pair $(\tilde X, \tilde B , \mathbf{M})$ which is crepant birational to $(X, B, \mathbf{M})$, and crepant birational maps
\begin{align*}
\psi \coloneqq g \sups -1. \circ \phi \colon X \drar \overline{X} \leftarrow \tilde{X},
\end{align*}
together with a morphism $\tilde q \colon \tilde X \rar Z$, such that:
\begin{itemize}
    \item[(1)] every generalized log canonical center of $(\tilde{X},\tilde {B} + \tilde {M})$ dominates $Z$;
    \item[(2)] $\tilde B \sups =1. = g \sups -1. (\overline{B}\sups =1.)$;
    \item[(3)] $g \sups -1. (\overline{B}\sups =1.)$ fully supports a $\tilde q$-big and $\tilde q$-semi-ample divisor;
    \item[(4)] every prime divisor $\tilde E \subset \tilde X$ that does not dominate $Z$ has non-empty intersection with $g \sups -1. (\overline{B}\sups =1.)$; and
    \item[(5)] $\psi \sups -1.$ is a crepant, birational contraction and $\tilde E \subset \tilde B \sups =1.$ for every $\psi \sups -1.$-exceptional divisor.
\end{itemize}
\end{corollary}

\begin{proof}
Let $(\hat X, \hat B, \mathbf M)$ be a generalized dlt model of $(X,B, \mathbf M)$.
Let $\overline q \colon (\overline{X}, \overline B , \mathbf M)\rar Z$ be an outcome of Theorem~\ref{main technical thm for complexes} applied to $(\hat X, \hat B , \mathbf M)$.
Write $\phi \colon X \drar \overline{X}$.
By construction, every $\phi \sups -1.$-exceptional divisor is contained in $\overline{B} \sups =1.$.
Let $(\tilde X, \tilde B , \mathbf M)$ be a generalized dlt model of $(\overline{X},\overline{B} , \mathbf{M})$, and set $g \colon \tilde{X} \rar \overline{X}$.
By Theorem~\ref{main technical thm for complexes}, $(\overline{X}, \overline{B} , \mathbf{M})$ satisfies (1), and so does $(\tilde X, \tilde B , \mathbf M)$.
Since $(\tilde X, \tilde B , \mathbf M)$ is a generalized dlt model of $(\overline{X}, \overline{B} , \mathbf{M})$, we have $g \sups -1.(\overline{B} \sups =1.) \subset \tilde B \sups =1.$.
By Remark~\ref{remark about technical thm}, we have the reversed inclusion, and (2) holds.
By construction, $\overline{B} \sups =1.$ fully supports a $\overline{q}$-ample divisor.
As $\Supp(g^\ast \overline{B} \sups =1.)=g \sups -1. (\overline{B} \sups =1.)$, then (3) holds.
Furthermore, any divisor $\tilde E \subset \tilde X$ that does not dominate $Z$ and is not $g$-exceptional intersects $g \sups -1. (\overline{B} \sups =1.)$, as $\overline{B} \sups =1.$ fully supports a $\overline{q}$-ample divisor.
Then, as every $g$-exceptional divisor is in $g \sups -1. (\overline{B} \sups =1.)$, (4) follows.
Finally, by construction, we either used Theorem~\ref{main technical thm for complexes} or generalized dlt models to produce new varieties.
Hence, (5) holds true.
\end{proof}

\begin{remark}
If $\dim \tilde X - \dim Z \geq 2$, then property (4) in the statement of Corollary~\ref{corollary reduction to fibers} can be further strengthened.
In fact, under such assumption, every prime divisor $\tilde E \subset \tilde X$ has non-empty intersection with $g \sups -1.(\overline{B}\sups =1.)$.
\\
On the other hand, if $\dim Z = \dim \tilde X -1$, this stronger statement could fail.
Indeed, to construct an example showing this cannot be achieved, it suffices to fix an elliptic curve $E$ and define $Z\coloneqq E$, $\tilde X \coloneqq \pr 1. \times E$, $M'=\tilde M =0$, and $\tilde B$ to be the union of the two disjoint sections of $\tilde X \rar Z$.
This example is an admissible case of Corollary~\ref{corollary reduction to fibers}, and $\tilde B$ has empty intersection with other sections of $\tilde X \rar Z$.
\end{remark}

Theorem~\ref{main technical thm for complexes} and Corollary~\ref{corollary reduction to fibers} make it possible to reduce the study of the dual complex of a generalized pair $(X,B,\mathbf{M})$ to the study of the dual complex of lower-dimensional pairs via the construction of suitable fibrations on the variety $X$.
The next result furthers this strategy by establishing a comparison between the dual complex of a variety and the dual complex of the generic fiber of a morphism.

\begin{lemma}
\cite{KX16}*{Lemma~28-30}
\label{lemma KX}
Let $E = \bigcup \subs i \in I. E_i$ be a simple normal crossing variety over a field $k$. 
\begin{enumerate}
\item[(1)] Let $K/k$ be a Galois extension with Galois group $G$. 
Then $G$ acts on $\mathcal{D}(E_K)$ and $\mathcal{D}(E_k)=\mathcal{D}(E_K)/G$.
\end{enumerate}
Let $q \colon E \rar Z$ be a morphism.
\begin{enumerate}
\item[(2)] The generic fiber $E \subs k(Z).$ is a simple normal crossing variety over the function field $k(Z)$ and $\mathcal{D}(E \subs k(Z).)$ is a subcomplex of $\mathcal{D}(E)$.
Furthermore, if every stratum dominates $Z$, then $\mathcal{D}(E \subs k(Z).)=\mathcal{D}(E)$.
\item[(3)] Assume that every stratum of $E$ dominates $Z$.
Let $z \in Z$ be a general point and $E_z$ the fiber over $z$.
Then $E_z$ is a simple normal crossing variety and there is a finite group $G$ acting on $\mathcal{D}(E_z)$ such that $\mathcal{D}(E)= \mathcal{D}(E_z)/G$. 
\end{enumerate}
\end{lemma}

\section{The dual complex of generalized log Calabi--Yau pairs} \label{sect.dual.cplx}

In this section, we collect some technical and partial results that will be used to prove our main theorem on the structure of the dual complex for generalized log Calabi--Yau pairs.

\subsection{Dual complex and MRC fibration}
\label{dual.MRC.sect}

Firstly, we aim to show that, to study the dual complex of generalized log canonical pairs of log Calabi--Yau type, it suffices to consider the case when the underlying variety is rationally connected.

\begin{proposition} \label{prop mrc}
Let $(X, B, \mathbf{M})$ be a generalized dlt pair. 
Assume that $\K X. + B + \mathbf{M}_X \sim \subs \qq. 0$. 
Let $g \colon X \drar Z$ be a dominant rational map to a non-uniruled variety $Z$.
Assume that there exists a non-empty open set of the base $U \subset Z$ such that the restriction of $g$ to $g^{-1}(U)$ is a morphism with connected fibers over $U$.
Then, every irreducible component $D$ of $B \sups =1.$ dominates $Z$.
\end{proposition}

\begin{proof}
We follow the strategy of the proof of~\cite{KX16}*{Proposition~19}.
\newline
Let us assume by contradiction that some component $D$ of $B \sups =1.$ does not dominate $Z$.
Up to replacing $Z$ birationally, we can assume that $Z$ is smooth and projective and that every prime divisor in $X$ either dominates $Z$ or dominates a prime divisor in $Z$.
In particular, $D$ dominates a prime divisor $P \subset Z$.
Let $X''$ be the normalization of the closure of the graph of $g$, and let $(X'',B''+M'')$ denote the log pull-back of $(X, B, \mathbf M)$ on $X''$.
Under these assumptions, the divisor $B''$ may be non-effective.
Let $h \colon X'' \rar Z$ and $\pi \colon X'' \rar X$ denote the induced morphisms.
Let $C \subset X$ be a sufficiently general complete intersection curve.
Thus, $D \cdot C > 0$ and we may assume that $C$ avoids any prescribed set of codimension at least 2.
In particular, we may assume that $C$ avoids the indeterminacy locus of $g$ and the exceptional locus of $\pi$, and we may identify $C$ with its strict transform in $X''$.
\\

{\bf Claim 1.}
{\it For every prime divisor $R \subset Z$ with $R \cap h(C) \neq \emptyset$, $\Supp(h^\ast (R))$ contains a prime divisor that is not exceptional for $X'' \rar X$.
In particular, $\Supp(h^\ast (R))$ contains a divisor that has non-negative coefficient in $B''$ and that dominates $R$.
}
\begin{proof}
Fix a prime divisor $R$ so that $R \cap h(C) \neq \emptyset$.
Since $Z$ is smooth, $R$ is Cartier, so $h^\ast(R)$ is well defined and $\Supp(h^\ast (R))$ is purely divisorial.
Since $R \cap h(C) \neq \emptyset$, there is a prime divisor $O \subset \Supp(h^\ast (R))$ so that $O \cap C \neq \emptyset$.
Since $X'' \rar X$ is an isomorphism along $C$, it is an isomorphism along the generic point of $O$.
In particular, $O$ is the strict transform of a prime divisor on $X$.
Since $B \geq 0$, it follows that $B''$ has non-negative coefficient along $O$.
Finally, by assumption on $Z$, every prime divisor on $X$ either dominates $Z$, or it dominates a prime divisor in $Z$.
Thus, $O$ has to dominate $R$.
\end{proof}

\noindent Since $X''$ and $Z$ are normal and generically $h$ has connected fibers, it follows that $h$ has connected fibers everywhere.
By construction, we have $\K X''. + B'' + M'' \sim_\qq 0$.
Thus, we can apply Theorem~\ref{generalized canonical bundle formula} to $h$, and induce a generalized sub-pair $(Z,B_Z,\mathbf N)$ such that
\begin{align*}
\K X''. + B'' + M '' \sim_\qq h^\ast (\K Z. + B_Z + \mathbf N _Z).
\end{align*}
Moreover, $\K Z. \cdot h(C) \geq 0$, by~\cite{MM86}*{Corollary~3} since $C$ is a complete intersection of ample divisors, $g$ is dominant, and $Z$ is not uniruled.
In addition, as $\mathbf N _Z$ is pseudo-effective, $\mathbf N _Z \cdot h(C) \geq 0$.
By the generalized canonical bundle formula, as $D$ dominates $P$ and $\mu_D(B)=1$, then $\mu_P (B_Z) > 0$.
Furthermore, if $\mu_Q (B_Z)<0$ for some prime divisor $Q \subset Z$, it follows that for every prime divisor $\Gamma \subset X''$ that dominates $Q$ we have $\mu_\Gamma (B'') < 0$.
Thus Claim 1 implies that $h(C)$ is disjoint from any prime divisor $Q$ with $\mu_Q (B_Z)<0$.
As $D \cdot C > 0$, then $P \cdot h(C)>0$.
This implies that
\begin{align}
\label{eqn.inters.C}
    0 < B_Z^{\geq 0} \cdot h(C) = B_Z \cdot h(C) \leq (\K Z. + B_Z + \mathbf N _Z) \cdot h(C) = 0 \cdot h(C) =0,
\end{align}
which leads to the required contradiction.
\end{proof}

Proposition~\ref{prop mrc} allows a first interesting reduction in the study of the dual complex of a generalized pair of log Calabi--Yau type.

Any normal and proper variety $X$ admits a birational contraction called a \emph{maximal rationally chain connected} (in short, mrcc) fibration $g \colon X \drar Z$, see~\cite{Kol96}*{Theorem~IV.5.2} for the details of the construction. 
Roughly speaking, the mrcc fibration is characterized by the following properties:
\begin{itemize}
    \item the fibers of $g$ are rationally chain connected; and
    \item almost every rational curve is contained in a fiber of $g$: namely, for a very general $z \in Z$, any rational curve intersecting $X_z$ is contained in $X_z$.
\end{itemize}
The mrcc fibration is uniquely determined up to birational equivalence.
By~\cite{Kol96}*{Theorem~IV.4.17}, the birational contraction $g$ is a well-defined morphism with connected fibers over a non-empty open set of the base $Z$.
If the variety $X$ is smooth, then the general fiber of $g$ will be rationally connected; 
in this case, the morphism $g$ is called the \emph{maximal rationally connected} (mrc) fibration.
Work of Hacon and M\textsuperscript{c}Kernan,~\cite{HM07}, shows that the same actually holds in the case of dlt pairs.
Therefore, for our purposes, we can always consider the mrc fibration of a dlt model of a generalized log canonical pair.

By work of Graber, Harris, and Starr, cf.~\cite{GHS03}*{Corollary~1.4}, the base $Z$ of an mrc fibration is not uniruled.
Therefore, Proposition~\ref{prop mrc} applies to the mrc fibration of a generalized dlt pair $(X,B, \mathbf{M})$ with $\K X. + B + M \sim_\qq 0$.
In particular, by Lemma~\ref{lemma KX}, we can reduce the study of the dual complex to the case when $X$ is a rationally connected variety.

\subsection{Reduction to the ample case} \label{red.ample.sect}
Let $(X, B, \mathbf{M})$ be a generalized log canonical pair.
Assume that $\K X. + B + M \sim_\qq 0$.
We are interested in studying $\mathcal{DMR}(X,B, \mathbf M)$.
By the results of~\S~\ref{dual.MRC.sect}, we can assume that $X$ is a rationally connected $\qq$-factorial klt variety.

As $\K X. + B +M \sim_\qq 0$, we aim to use vanishing theorems, cf. Theorem~\ref{kvv}, to show that $H^i(X,\O X.)=0$, for $i > 0$;
the vanishing of the higher cohomology of the structure sheaf is known to imply the vanishing of the cohomology of the dual complex, see~\cite{KX16}*{\S~4}.
On the other hand, since we are interested in studying $\mathcal{DMR}(X,B, \mathbf M)$, we can assume that the generalized pair $(X,B+M)$ is not generalized klt.
In general, the vanishing theorems are not expected to hold for purely lc pairs, without imposing some conditions on the positivity of $B+M$.
Nonetheless, in Corollary~\ref{corollary reduction to fibers} we showed that we can assume that $X$ is endowed with a fibration $X \to Z$ such that $B$ fully supports an effective divisor which is big and semi-ample over $Z$.

Let $\eta$ be the generic point of $Z$. 
It follows from Lemma~\ref{lemma KX} that 
\begin{align*}
\mathcal{DMR}(X,B, \mathbf M) \simeq \mathcal{DMR}(\tilde X _\eta,\tilde B _\eta , \mathbf M | \subs \tilde X _\eta.).
\end{align*}
Furthermore, Lemma~\ref{lemma KX} implies the existence of a finite group $G$ such that
\begin{align*}
\mathcal{DMR}(X,B, \mathbf M) \simeq \mathcal{DMR}(\tilde X _z, \tilde B _z , \mathbf M | \subs \tilde X _z.)/G,
\end{align*}
for a general closed point $z \in Z$.
As we are interested in showing that the identity $H^i(\mathcal{DMR}(X,B, \mathbf M),\qq)=0$ holds, for $i >0$, it suffices to show that the corresponding identity $H^i(\mathcal{DMR}(\tilde X _z, \tilde B _z , \mathbf M | \subs \tilde X _z.),\qq)=0$ holds, for $i >0$.

Therefore, invoking Corollary~\ref{corollary reduction to fibers}, we can assume that $(X,B+M)$ is $\qq$-factorial generalized dlt, and that $B \sups =1.$ fully supports a big and semi-ample effective $\mathbb{Q}$-divisor $H$ with $B \geq H$.
Then, the generalized pair $(X,B-H,\mathbf{M})$ is generalized klt.
Hence,
\begin{align*}
0 \sim_\qq \K X. + (B-H) + H + M
\end{align*}
and $(X, B-H, \mathbf{M} + \overline{H})$ is generalized klt.
Theorem~\ref{kvv} then implies that 
\begin{align*}
H^i(X,\O X.) = 0, \ {\rm for } \ i > 0.
\end{align*}

\subsection{Reduction to the classical case}
In this subsection, we show that if the dual complex of a generalized log Calabi--Yau pair $(X, B, \mathbf{M})$ is not collapsible to a point, then we can reduce the study of $H^i(\mathcal{DMR}(X,B, \mathbf M),\qq)$ to the study of the cohomology of the dual complex of a log pair $(\tilde X, \tilde B)$ with $K_{\tilde X}+ \tilde B \sim_{\mathbb{Q}} 0$; the latter case was studied in detail in~\cite{KX16}.

We will use the notation introduced in~\S~\ref{red.ample.sect}.

\begin{corollary} \label{corollary reduce to pairs}
Let $(X, B, \mathbf{M})$ be a generalized log canonical pair. Assume that $\K X. + B + M \sim \subs \qq. 0$.
Let $(\tilde X, \tilde B , \mathbf M)$ be the generalized pair constructed in Corollary~\ref{corollary reduction to fibers} together with the morphism $\tilde q \colon \tilde X \rar Z$ whose existence is claimed in the corollary.
Let $\eta$ be the generic point of $Z$.
Assume that $\mathcal{DMR}(X,B, \mathbf M)$ is not collapsible to a point.
Then, $\lfloor \tilde{B}_\eta \rfloor= \tilde B _\eta$ and $\mathbf M | \subs \tilde X _\eta. \equiv 0$.
\end{corollary}

\begin{proof}
Set $\tilde M \coloneqq \mathbf M \subs \tilde X.$.
Assume that $\tilde B \sups =1._\eta \not \equiv \tilde B _\eta + \tilde M _\eta$.
Then, $\K \tilde X. + \tilde B \sups =1.$ is not pseudo-effective over $Z$.
Thus, we may run a $(\K \tilde X. + \tilde B \sups =1.)$-MMP with scaling over $Z$, which terminates with a Fano contraction $p \colon \hat X \rar Y$.
The final model $(\hat X, \hat B \sups =1.)$ of this MMP is dlt, where $\hat B$ is the strict transform of $\tilde B$ on $\hat X$.
As $B \sups =1.$ fully supports a divisor that is big and semi-ample over $Z$, then $\hat B \sups =1.$ dominates $Y$.
Therefore, by~\cite{KX16}*{Proposition~24}, $\mathcal{DMR}(\hat X, \hat B \sups =1.)$ is collapsible to a point.
\newline
Let $h \colon (\check X,\check B, \mathbf M) \rar (\hat X, \hat B, \mathbf M)$ be a generalized dlt model of $(\hat X , \hat B, \mathbf M)$.
By Corollary~\ref{corollary reduction to fibers}, $\check B \sups =1.$ fully supports a big and mobile divisor over $Z$, so that, $\check B \sups =1.$ dominates $Z$ and we have $\check B \sups =1. = h \sups -1.(\hat B \sups =1.)$, cf~\cite{KX16}*{\S~22}.
By~\cite{dFKX}*{Theorem~3}, $\mathcal{DMR}(\check X,\check B , \mathbf M)$ collapses to $\mathcal{DMR}(\hat X, \hat B \sups =1.)$.
\end{proof}

\begin{remark} \label{remark reduction}
Using the result from the previous subsection, Corollary~\ref{corollary reduce to pairs} implies that the study of the cohomology of the dual complex of generalized log Calabi--Yau pairs reduces to the case of log Calabi--Yau pairs.
Indeed, if $\mathcal{DMR}(X,B, \mathbf M)$ is collapsible, there is nothing to prove.
Otherwise, we can consider the relation $\mathcal{DMR}(X,B, \mathbf M) \simeq \mathcal{DMR}(\tilde X _z, \tilde B _z , \mathbf M | \subs \tilde X _z.)$ discussed in \S~\ref{red.ample.sect}, where $\tilde X _z$ denotes a general fiber of the morphism $X \to Z$ whose existence is stated in Corollary~\ref{corollary reduction to fibers}.
Then, by Corollary~\ref{corollary reduce to pairs}, we have $\tilde M _z \equiv 0$.
This implies that $(\tilde X _z, \tilde B _z , \mathbf M | \subs \tilde X _z.)$ is a pair.
Therefore, to study $\mathcal{DMR}(\tilde X _z, \tilde B _z , \mathbf  M | \subs \tilde X _z.)$ we can use the results in~\cite{KX16}.
\end{remark}

We conclude this section by showing that, in the setup of this work, the dual complex of a generalized pair is equidimensional. 

\begin{theorem} \label{dual complex equidimensional}
Let $(X, B, \mathbf{M})$ be a generalized log canonical pair.
Assume that $\K X. + B + M \sim \subs \qq. 0$.
Then, $\mathcal{DMR}(X,B, \mathbf M)$ has the same dimension at every point.
\end{theorem}

\begin{proof}
Let $(\overline{X},\overline{B}, \mathbf{M})$ be a generalized dlt model of $(X,B+M)$.
By Remark~\ref{remark dual complex gpair to pair}, we have $\mathcal{DMR}(X,B, \mathbf M)=\mathcal{DMR}(\overline{X},\overline{B})$.
Furthermore, by~\cite{KM98}*{Theorem~2.44}, we may compute $\mathcal{DMR}(\overline{X},\overline{B})$ directly by $\mathcal{D}(\overline{B}\sups =1.)$.
Theorem~\ref{Theoremp1 link} implies that $\mathcal{D}(\overline{B}\sups =1.)$ gives rise to an equidimensional CW-complex.
\end{proof}

\section{Proof of the theorems}
The following proposition will be used in the proof of Theorem~\ref{main theorem}.

\begin{proposition} \label{key prop}
Let $(X,B,\mathbf{M})/S$ be a generalized pair.
Assume that $X$ is a generalized dlt model for $(X,B,\mathbf M)/S$.
Let $f \colon X \rar S$ be a projective morphism such that $\K X. + B + \mathbf{M}_X \sim \subs \qq,f. 0$. Fix $s \in S$ and assume that $f \sups -1. (s)$ is connected but $f \sups -1. (s) \cap \Nklt(X,B,\mathbf{M})$ is disconnected (as $k(s)$-schemes).
Then, a component of $B \sups \geq 1.$ dominates $S$.
\end{proposition}

\begin{proof}
We divide the proof into steps, for the reader's convenience.
\\

{\bf Step 0}. 
{\it 
In this step, we show that it suffices to show that the proposition holds for the contraction in the Stein factorization of $f$.}
\newline
By assumption, $X$ is $\qq$-factorial, and $(X,B \wedge \Supp(B),\mathbf M)/S$ is generalized dlt.
Let $\phi \colon X \rar Y$ denote the Stein factorization of $f$, and let $\psi \colon Y \rar S$ the induced finite morphism.
Since the fiber of $f$ over $s$ is connected, and $\phi$ has geometrically connected fibers, there exists a unique point $y \in Y$ mapping to $s \in S$.
Furthermore, as $f \sups -1. (s) \cap \Nklt(X,B,\mathbf{M})$ is disconnected, then so is $\phi \sups -1. (y) \cap \Nklt(X,B,\mathbf{M})$;
hence, the hypotheses of the statement apply to the morphism $\phi \colon X \rar Y$ as well.
Furthermore, $B \sups \geq 1.$ dominates $S$ if and only if it dominates $Y$.
\\
Therefore, up to substituting $S$ (resp. $s$, $f$) with $Y$ (resp. $y$, $\phi$), from now on, we shall assume that the morphism $f$ satisfies the additional property $f_\ast \O X. = \O S.$.
\\

{\bf Step 1}. {\it In this step we make further reductions and explain the strategy of proof}.
\newline
Remark~\ref{remark nklt in gdlt model} implies that $\Nklt(X,B,\mathbf M)= {\rm Supp} B \sups \geq 1.$.
Moreover, by~\cite{Kol13}*{4.38}, passing to an \'etale neighborhood of $s \in S$ we may assume that 
\begin{center}
{\bf ($\ast$)} {\it different connected components of $\Nklt(X,B,\mathbf M)\cap f \sups -1. (s)$ are contained\\ 
in different connected components of $\Nklt(X,B,\mathbf M)$}.
\end{center}
To prove the proposition, we shall argue by contradiction: 
namely, we shall assume that no component of $B \sups \geq 1.$ dominates $S$.
Given two distinct connected components $D_1$ and $D_2$ of $B \sups \geq 1.$ intersecting $f \sups -1.(y)$ we will obtain the sought contradiction by using the fact that the components of $B^{\geq 1}$ are vertical over S to show that $D_1$ and $D_2$ actually intersect.
\\

{\bf Step 2}.
{\it In this step we define a generalized pair $(S,B_S + M_S)$ on $S$ using the canonical bundle formula.
We introduce an auxiliary divisor $\Sigma'$ on a higher model $S' \to S$; $\Sigma'$ is only needed to treat the case when $(X,B,\bM.)$ is not generalized log canonical.}
\newline
By~\cite{Fil18}*{Theorem 1.4}, there exists  $(S,B_S + M_S)$ a generalized pair induced by the generalized canonical bundle formula on $S$.
Thus,~\cite{Fil18}*{Proposition 4.16} implies that there is some generalized non-klt center of $(Y, B_S + M_S)$ containing $s$.
Let $\alpha \colon S' \rar S$ be a generalized dlt model for $(S, B_S + M_S)$, and let $(S', B \subs S'. + M \subs S'.)$ denote the trace of $(S,B_S+M_S)$ on $S'$.
Let $X'$ be the normalization of the main component of $X \times_S S'$.
Let $\beta \colon X' \rar X$ and $g \colon X' \rar S'$ denote the induced morphisms.
The assumption that $\Nklt(X, B, \mathbf{M})$ does not dominate $S$ implies that $(X',B', \mathbf{M})$ is generalized klt over the generic point of $S'$, where $(X',B', \mathbf{M})$ is the sub-pair induced by log pull-back on $X'$.
\newline
Let us define the divisor $\Sigma' \coloneqq (B \subs S'. \wedge \Supp(B \subs S'.))-B \subs S'.$ on $S'$. 
Then $\Sigma'$ is the only divisor supported on $\Supp (B \subs S'. \sups \geq 1.)$ such that $B \subs S'. \wedge \Supp(B \subs S'.)$ is the boundary part on $S'$ for the generalized canonical bundle formula applied to  $K_X'+B'+g^\ast\Sigma'+\mathbf M$.
By definition, we have $\Sigma' \leq 0$;
moreover, $\Sigma' = 0$ if $(X,B,\mathbf M)/S$ is generalized log canonical.
By inversion of adjunction for the generalized canonical bundle formula, cf.~\cite{Fil18}*{Proposition 4.16}, for every irreducible component $D'$ of $\Supp(B \subs S'. \sups \geq 1.)$ there is a divisorial valuation $E'$ over $X'$ such that $a \subs E'.(X',B'+g^\ast\Sigma',\mathbf M) = 0$ and $c \subs X'.(E')$ dominates $D'$.
Furthermore, the same results imply that the generalized sub-pair $(X',B'+g^\ast\Sigma', \mathbf M)$ is generalized sub-log canonical.
\\

{\bf Step 3.} {\it Let $\pi \colon X'' \rar X'$ be a log resolution of $(X',\Supp(B') + g^\ast\Supp(B \subs S'.))$ where $\mathbf{M}$ descends.
In this step we define a divisor $F''$ on $X''$, cf.~\eqref{eqn.def.F''}, supported on those components of $B'' +(g \circ \pi)^\ast\Sigma'$ of coefficient in $ (0, 1)$ whose image in $S'$ is contained in $B^{\geq 1} \subs S'.$, and we discuss the properties of $F''$.}
\newline
Let $\pi \colon X'' \rar X'$ be a log resolution of $(X',\Supp(B') + g^\ast\Supp(B \subs S'.))$ where $\mathbf{M}$ descends, and let $B''$ denote the sub-boundary induced on $X''$.
Up to passing to a higher smooth birational model, we may assume that for every prime component $D'$ of $\Supp(B \subs S'. \sups \geq 1.)$ the corresponding divisorial valuation $E'$ over $X'$, defined at the end of the previous step, is extracted.
In particular, for every prime divisor $D' \subset \Supp(B \subs S'. \sups \geq 1.)$ there exists always a component of $B''+(g\circ \pi)^\ast\Sigma'$ of coefficient 1 that dominates $D'$.
Let us define
\begin{align}
\label{eqn.def.F''}
F'' \coloneqq& \sum_{E_i'' \in J} E_i'', \\ 
\nonumber
J \coloneqq \{ E_i'' \subset X'' \ \text{prime divisor} \ \vert \ 
\mu \subs E_i''. ( B'' &+(g \circ \pi)^\ast\Sigma' ) \sups \geq 0. < 1 \ 
\text{and} \ 
g \circ \pi (E''_i) \subset \Supp (B \subs S'. \sups \geq 1.)\}. 
\end{align}
In order to reach the sought contradiction, we will show that it is possible to contract $F''$ by means of suitable runs of the MMP without altering the number of connected components of $\Nklt(X,B,\bM.)$ around $s$, in such a way that on the final model $\Nklt(X,B,\bM.)$ is connected around $s$.
\newline
As observed at the end of the previous step, by the definition of $\Sigma'$ and inversion of adjunction for fiber spaces, the generalized subpair $(X'',B''+(g \circ \pi)^\ast\Sigma',\bM.)$ is generalized sub-log canonical.
Hence,
\begin{align*}
(B'' +(g \circ \pi)^\ast\Sigma' ) \sups \geq 0. = (B'' +(g \circ \pi)^\ast\Sigma' ) \sups \geq 0. \wedge \Supp(( B'' +(g \circ \pi)^\ast\Sigma' ) \sups \geq 0.),
\end{align*}
and any prime divisor $Q'' \subset X''$ such that $\mult \subs Q''. ( B'' +(g \circ \pi)^\ast\Sigma' ) \sups \geq 0.=1$ is mapped into $\Supp(B \subs S'. \sups \geq 1.)$.
For brevity, we define
\begin{align*}
\Delta'' &\coloneqq 
(B'' +(g \circ \pi)^\ast\Sigma' ) \sups \geq 0..
\end{align*}
Fix a rational number $0 < \epsilon \ll 1$.
Then, as $X''$ is a log resolution of $(X',\Supp(B') + g^\ast\Supp(B \subs S'.))$, the pair $(X'',\Delta'' + \epsilon F'')$ is dlt, and
\begin{align}
    \label{eqn.def.G''}
\K X''. + \Delta'' + \epsilon F'' + \mathbf M \subs X''. \sim \subs \qq,g \circ \pi. 
(B''  +(g \circ \pi)^\ast\Sigma')\sups \leq 0. + \epsilon F'' \eqqcolon G''.
\end{align}
By construction, the divisor $G''$ is effective. Moreover, the following claim holds.

\medskip

{\bf Claim 1}. {\it If $\Supp G''$ dominates $S'$, then $G'' \not \in \overline{{\rm Mov}}(X''/X')$, where $\overline{{\rm Mov}}(X''/X')$ is the closure of the cone of relatively movable divisors}.
\begin{proof}[Proof of Claim 1]
Since $B'$ is effective over the generic point of $S'$, as observed in Step 2, and since $F''$ is vertical over $S'$, it follows that $\Supp(G'')$ is $\pi$-exceptional over generic point of $S'$.
Let $U \subset S'$ be a non-empty open subset such that all the vertical components of $\Supp G''$ are mapped into $S' \setminus U$.
Then, we set $X'_U \coloneqq X' \times \subs S'. U$ and $X'' \coloneqq X'' \times \subs S'. U$.
Also, we let $G''_U$ be the pull-back of $G''$ to $X''_U$.
\newline
By definition of relatively movable divisor, if $G'' \in \overline{{\rm Mov}}(X''/X')$, then $G''_U \in \overline{{\rm Mov}}(X''_U/X'_U)$.
Thus, it suffices to show that $G''_U \not \in \overline{{\rm Mov}}(X''_U/X'_U)$.
By construction, $G''_U$ is effective and exceptional for $X''_U \rar X'_U$.
Thus, $G_U''$ is degenerate in the sense of \cite{Lai11}*{Definition 2.9}.
Hence, by \cite{Lai11}*{Lemma 2.10}, it follows that $G''_U \not \in \overline{{\rm Mov}}(X''_U/X'_U)$.
\end{proof}

{\bf Step 4.} 
{\it 
In this step we run a relative $(\K X''. + \Delta'' + \epsilon F'' + \mathbf M \subs X''.)$-MMP over $X'$ and we show that this MMP contracts those components in $F''$ that dominate $S'$.}
\newline
As observed in the proof of Claim 1, any component of $G''$ that dominates $S'$ is $\pi$-exceptional by construction.
If such components exist on $X''$, then we observed in Claim 1 that $\K X''. + \Delta'' + \epsilon F'' + \mathbf M \subs X''.\not \in \overline{{\rm Mov}}(X''/X')$. 
To contract those components, we run a $(\K X''. + \Delta'' + \epsilon F'' + \mathbf M \subs X''.)$-MMP relative to $X'$ with scaling of an ample divisor. 
We can run this MMP as the pair is generalized dlt.
\\
By~\cite{Fuj11}*{Theorem 2.3}, whose proof extends to generalized dlt pairs, after finitely many steps the run of the MMP terminates to yield a model $X'''$ over which $(\K X'''. + \Delta''' + \epsilon F''' + \mathbf M \subs X'''.)\in \overline{{\rm Mov}}(X'''/X')$.
Let $\sigma \colon X''' \rar X'$ and $\rho \coloneqq g\circ \sigma \colon X''' \rar S'$ be the induced morphisms and let $F'''$ (resp. $G'''$) be the strict transform of $F''$ (resp. $G''$) on $X'''$.
By~\eqref{eqn.def.G''}, $G''' \in \overline{{\rm Mov}}(X''/X')$.
Hence, over the generic point of $S'$, $G''$ is supported on divisors that are exceptional for $X'' \rar X'$.
By \cite{Lai11}*{Lemma 2.10}, these divisors need to be contracted for $G'''$ to be limit of movable divisors relatively to $X'$.
Therefore, $G'''$ is vertical over $S'$ as desired.
\\
When $G''$ does not dominate $S'$, we do not need to run any MMP, and in the rest of the proof we have $X'' = X'''$.
\\

{\bf Step 5.} {\it In this step we run $(\K X'''. + \Delta''' + \epsilon F''' + \mathbf M \subs X'''.)$-MMP relative to $S'$ to contract $G'''$.
}
\newline
By~\cite{BZ16}, we can run a $(\K X'''. + \Delta''' + \epsilon F''' + \mathbf M \subs X'''.)$-MMP over $S'$.
We need to show that this MMP terminates.
First, we check that it terminates over a big open set of $S'$.
Let $P'''$ be a prime component of $G'''$.

\medskip

{\bf Claim 2}.
{\it Assume that $\rho(P''')=D'$, for a prime divisor $D' \subset S'$. 
Then $P'''$ is of insufficient fiber type over $S'$, cf.~\cite{Lai11}*{Definition~2.9}}.

\begin{proof}[Proof of Claim 2]
We first assume that $D' \subset \Supp (B \subs S'.\sups \geq 1.)$.
Then, by construction, there is another prime divisor $Q''' \subset X'''$ not contained in $\Supp (G''')$ but mapping to $D'$, see the end of Step 2 or, alternatively, the start of Step 3.
\\
Instead, if $D' \not \subseteq \Supp( B \subs S'. \sups \geq 1.)$, then $D'$ is not $\alpha$-exceptional, by construction of $\alpha \colon S' \to S$, cf. Step 2.
Furthermore, $B'' = B'' + (g \circ \pi)^\ast\Sigma'$ over the generic point of $D'$.
Since our claim can be verified over the generic point of $D'$, we can assume $\Sigma'=0$ in the rest of the proof of this claim.
Similarly, as $F''$ is mapped into $\Supp( B \subs S'. \sups \geq 1.)$, we can assume $F''=0$ for the purpose of this claim.
Thus, by these two observations, it follows that $P'''$ is a component of $(B''')\sups \leq 0.$.
Now, since $D'$ is not $\alpha$-exceptional, $S' \rar S$ is an isomorphism at the generic point of $D'$.
Thus, by definition of $X'$, $X' \rar X$ is an isomorphism over the generic point of $D'$.
For this reason, as $B \geq 0$, it follows that $B' \geq 0$ over the generic point of $D'$.
Since $P'''$ is a component of $(B''')\sups \leq 0.$, it follows that $P'''$ is exceptional for $X''' \rar X'$.
In particular, as $P'''$ is a component of $\sigma^*(g^* D')$, there exists a prime divisor $Q''' \subset X'''$ that is not $\sigma$-exceptional and that dominates $D'$. 
\end{proof}
If $G'''$ has any prime component dominating a prime divisor $D' \subset S'$, then by Claim 2 and~\cite{Lai11}*{Lemma~2.10}, $G''' \not \in  \overline{{\rm Mov}}(X'''/S')$; thus, by~\cite{Fuj11}*{Theorem 2.3} the $(\K X'''. + \Delta''' + \epsilon F''' + \mathbf M \subs X'''.)$-MMP over $S'$ with ample scaling terminates at the generic point of $D'$ and it contracts those components of $G'''$ that are of insufficient fiber type over $D'$.
Hence, after finitely many steps of running the $(\K X'''. + \Delta''' + \epsilon F''' + \mathbf M \subs X'''.)$-MMP relative to $S'$, we reach a model $X'''' \to S'$ such that no component of $G''''$ dominates a divisor in $S'$, where $G''''$ is the strict transform of $G'''$ on $X''''$.
If $G'''' = 0$, we stop.
Otherwise, as the image of $G'''$ in $S'$ has codimension $\geq 2$,   by~\cite{Lai11}*{Lemma~2.10}, there is a component $P'''' \subset \Supp(G'''')$ such that $P'''' \subset \mathbf{B}_-(G''''/S')$.
Therefore, $G''''\not \in \overline{{\rm Mov}}(X''''/S')$ and a further run of the MMP contracts $P''''$.\\

{\bf Step 6.} {\it In this step we list all the properties of the model $\overline{X}$ that we obtain after contracting the divisors in $G'''$
and we reach the sought contradiction to conclude the proof.}
\newline
After finitely many steps, we reach a model $\overline{X}$ with morphism $\tau \colon \overline X \rar S'$ on which $G''''$ has been contracted.
In particular, any prime divisor $\overline P \subset \overline{X}$ such that $\tau(\overline{P}) \subset \Supp ( B \subs S'. \sups \geq 1.)$ satisfies $\mult \subs \overline{P}. \overline{B} \geq 1$.
Indeed, the components of $(B'' + (g \circ \pi)^\ast\Sigma') \sups < 1.$ mapping to $\Supp ( B \subs S'. \sups \geq 1.)$ were supported on $\Supp((F'')\sups \geq 0.)$, which is contracted on the model $\overline{X}$.
Furthermore, the fact that $\Sigma' \leq 0$ guarantees that if $\mu \subs Q''.(B'')<1$ then $\mu \subs Q''.(B'' + (g \circ \pi)^\ast\Sigma')<1$.
Because of this and the fact that the support of $(B'' + (g \circ \pi)^\ast\Sigma')<0$ has been contracted on $\overline X$, as that was in the support of $G''$, cf.~\ref{eqn.def.G''}, it follows that $\overline{B}\geq 0$, where $\overline B$ is the strict transform of $B''$ on $\overline{X}$.
By Proposition~\ref{p.conn.gen.pairs} and Remark~\ref{remark nklt in gdlt model}, there is a unique connected component $\Omega'$ of $\Supp(B \subs S'. \sups \geq 1.)$ such that $y \in \alpha(\Omega')$.
By construction, $\Nklt(\overline{X},\overline{B},\mathbf{M})$ is connected over $\Omega'$, as any prime divisor $\overline P \subset \overline{X}$ such that $\tau(\overline{P}) \subset \Omega'$ satisfies $\mu \subs \overline{P}. \overline{B} \geq 1$.
\\
Let $\hat X$ be a common resolution of $X$ and $\overline X$.
Then, by Proposition~\ref{p.conn.gen.pairs} applied to $\phi \colon \hat X \rar X$ and $\overline{\phi} \colon  \hat X \rar \overline{X}$, the connected components of $\Nklt(X,B,\mathbf M)$ and of $\Nklt(\overline{X},\overline{B}, \mathbf{M})$ are in bijection, as these two generalized pair are crepant to each other; in fact, the connected components of $\Nklt(X,B,\mathbf M)$ and of $\Nklt(\overline{X},\overline{B}, \mathbf M)$ are in  1-1 correspondence with the components of $\Nklt(\hat X, \hat B, \bM.)$, where $(\hat X, \hat B, \bM.)$ denotes the trace of $(X,B, \bM.)$ on $\hat X$.
This leads to the sought contradiction.
Indeed, as $\alpha \colon S' \rar S$ has connected fibers, $\alpha \sups -1.(s) \subset \Omega '$ is connected.
As also $\tau \colon \overline{X} \rar S'$ has connected fibers, then $\tau \sups -1.(\alpha \sups -1.(s))$ is connected and, by construction, $\tau \sups -1.(\alpha \sups -1.(s)) \subset \Supp(\overline{B}\sups \geq 1.)$.
Applying Proposition~\ref{p.conn.gen.pairs} to $\overline \phi \colon \hat X \rar \overline{X}$, then $\overline{\phi} \sups -1. (\tau \sups -1.(\alpha \sups -1.(s))) \cap \Nklt(\hat X, \hat B ,\bM.)$ is connected.
On the other hand, by a similar argument,  it follows that $\phi \sups -1. (f \sups -1.(s)) \cap \Nklt(\hat X,\hat B,\bM.)$ is disconnected, since $f \sups -1.(s) \cap \Nklt(X,B,\bM.)$ is disconnected, by Proposition \ref{p.conn.gen.pairs} applied to $\phi \colon \hat X \rar X$.
\end{proof}

\begin{proof}[Proof of Theorem~\ref{main theorem}]
By Theorem~\ref{p.conn.gen.pairs}, we may replace $X$ with a generalized dlt model $f^m \colon X^m  \rar X$; 
thus, we can assume that the generalized pair $(X, B \wedge \Supp(B), \mathbf{M})$ is a $\mathbb{Q}$-factorial generalized dlt pair and that $K_X+B+M \sim_{\qq,f} 0$. 
Remark~\ref{remark nklt in gdlt model} implies that $\Nklt(X,B , \mathbf M)= {\rm Supp} (B \sups \geq 1.)$.
Moreover, by~\cite{Kol13}*{\S~4.38}, passing to an \'etale neighborhood of $s \in S$ we can assume that 
\begin{center}
{\bf ($\ast$)} {\it different connected components of $\Nklt(X,B, \mathbf M)\cap \pi \sups -1. (s)$ are contained\\ 
in different connected components of $\Nklt(X,B, \mathbf M)$}.
\end{center}
Under these assumption, we shall show that $X$ is a $\mathbb{P}^1$-link over $S$.
Moreover, by Proposition~\ref{key prop} we can assume that at least one component of $B \sups \geq 1.$ dominates $S$. 
Hence, $\K X. + B - \epsilon B \sups \geq 1. + M$ is not pseudo-effective over $S$, for any $\epsilon >0$. 
As in addition $(X, B \sups <1. + M)$ is generalized klt, we can run a $(K_X+ B \sups <1. +M)$-MMP over $S$
\begin{equation}
\label{mmp.diag.proof1}
\xymatrix{
X =X_0 \ar@{-->}[r]^{\pi_1}  \ar[rrd]_{f_0(=f)} &X_1 \ar@{-->}[r]^{\pi_2}  \ar[rd]^{f_1}&\dots \ar@{-->}[r]^{\pi_n} & X_n  \ar[r]^h \ar[ld]_{f_n} & Z \ar[lld]^{g}\\
& & S & &
}
\end{equation}
which terminates with a Mori fiber space, $h\colon X_n  \rar Z$ over $S$, cf.~\cite{BZ16}*{Lemma~4.4}.
At each step of the MMP in~\eqref{mmp.diag.proof1}, we define $B_k \coloneqq  \pi_{k \ast} B_{k-1}$ and $M_k \coloneqq  \mathbf{M}_{X_k}$, where $B_0=B$; 
hence, $K_{X_k}+B_k+M_k \sim_{\qq,f_k} 0$. 
Applying Lemmata~\ref{conn.div.contr.lemma} and~\ref{flop.nklt.lem} at a given step $\pi_k$ of~\eqref{mmp.diag.proof1}, the number of connected components of $\Nklt(X_{k-1}, B_{k-1}, \mathbf M)$ in a neighborhood of $f^{-1}_{k-1}(s')$, $s' \in S$, is the same as the number of connected components of $\Nklt(X_{k}, B_{k}, \mathbf M)$ around $f^{-1}_{k}(s')$. 
Moreover, while for $k>0$ the support of $\Nklt(X_k,B_k, \mathbf M)$ does not necessarily coincide anymore with $B_k \sups \geq 1.$, it still holds that $\Supp(B_k \sups \geq 1.) \subset \Nklt(X_k,B_k, \mathbf M)$ and every irreducible component of $\Nklt(X_k,B_k, \mathbf M)$ contains at least one component of $B_k \sups \geq 1.$, since at each step of this run of the MMP $B_k \sups \geq 1.$ has positive intersection with the contracted extremal ray.
\newline
Hence, if $\Nklt(X, B, \mathbf M)$ is disconnected in a neighborhood of the fiber $f^{-1}(s)$, then so is $\Nklt(X_n, B_n, \mathbf M)$ in a neighborhood of $f_{n}^{-1}(s)$. 
As $h \colon X_n \rar Z$ is a Mori fiber space and $B_n \sups \geq 1.$ is ample over $Z$, there exists at least one component $\widetilde{D}$ of $B_n \sups \geq 1.$ ample over $Z$; 
thus, $\widetilde{D}$ dominates both $Z$ and $S$. 
Let $D'$
be any other 
component of $B_n \sups \geq 1.$ in a neighborhood of $f_{n}^{-1}(s)$. 
As $\widetilde D$ is ample over $Z$, in particular, it is horizontal over $Z$, hence, as $h$ is a Mori fiber space, then also $D'$ must be ample
over $Z$, otherwise, $D' \supseteq f^{-1}(s)$ and $D' \cap \widetilde{D} \cap f_n^{-1}(s) \neq \emptyset$ which would prompt a contradiction.
Then, $D'$ dominates $Z$ and it is $h$-ample.
Hence, we may argue as in~\cite{Kol13}*{proof of Proposition~4.37}.
In particular, all the reduced fibers of $h$ are smooth rational curves, and $\tilde D$ and $\tilde D '$ are disjoint sections of $h$.
Thus, as $(\tilde D +\tilde{D}') \cdot F=2$ for a general fiber $F$ of $h$, it follows that $B_n$ has to have coefficient one along $\tilde D$ and  $\tilde{D}'$.
Furthermore, every other component of $B_n$ is vertical for $h$.
Since we are assuming that $\Nklt(X_n,B_n,\bM.) \cap f_n^{-1}(s)$ is disconnected, the vertical components of $B_n \sups \geq 1.$ have to be disjoint from $f_n^{-1}(s)$.
Hence, up to shrinking around $s \in S$, $B_n \sups \geq 1.=\tilde D +\tilde{D}'$.
\newline
By construction of $X \drar X_n$ and the fact that $X_n$ is $\qq$-factorial, it follows that $(X_n,B_n \sups <1.)$ is a klt pair.
Thus, it follows from~\cite{Kol13}*{Proposition 4.37} and its proof that $(X_n,B_n) \rar Z$ is a standard $\pr 1.$-link.
Thus, since $\K X_n. + B_n \sim \subs \qq,h. 0$, it follows that $M_n \sim \subs \qq,h. 0$.
In particular, conditions (0), (1), (3), and (5) of Definition~\ref{def p1 link} are satisfied.

In order to show that also condition (4) in Definition~\ref{def p1 link} is satisfied,
it suffices to show that $\tilde D$ and $\tilde D '$ are the only generalized log canonical centers of $(X_n,B_n,\bM.)$.
Assume by contradiction that it is not the case.
Then, as $(X_n,B_n)$ is plt with two log canonical centers, there exists $\alpha \in (0,1]$ so that $(X_n,B_n, \alpha \bM.)$ is generalized log canonical and has three or more generalized log canonical centers.
What we have shown so far in particular implies that the only divisorial components of $\Nklt(X_n,B_n,\bM.)$ are $\tilde D$ and $\tilde D '$.
Therefore, since $X_n \rar Z$ has relative dimension 1, it follows that $\tilde D$ and $\tilde D '$ are the only log canonical centers of $(X_n,B_n, \alpha \bM.)$ that dominate $Z$.
Let $\phi \colon X_n' \rar X_n$ be a generalized dlt model for $(X_n,B_n, \alpha \bM.)$, and let $(X_n',B_n'+\alpha M_n')$ denote its trace on $X_n'$.
Since $M_n \sim \subs \qq,h. 0$, we have $\K X_n . + B_n + \alpha M_n \sim \subs \qq,h.0$.
Hence, we have $\K X_n'. + B_n' + \alpha M_n' \sim \subs \qq,Z. 0$.
By construction, we have $B_n' \geq 0$, as $\phi$ only extracts divisors with generalized log discrepancy 0.
Now, let $E'$ be a component of $(B_n') \sups =1.$ that is not $\tilde D$ nor $\tilde D '$; such divisor exists by the absurd assumption.
Since $0 < \alpha \leq 1$, we have $\Nklt(X_n,B_n,\alpha \bM.) \subset \Nklt(X_n,B_n,\bM.)$.
Since $\tilde D$ and $\tilde D '$ belong to disjoint connected components of $\Nklt(X_n,B_n,\bM.)$,
at least one among $\tilde D$ and $\tilde D '$ belongs to a connected component of $\Nklt(X'_n,B'_n,\alpha \bM.)$ that is disjoint from the component containing $E'$.
Up to swapping the roles, we may assume $E' \cap \tilde D = \emptyset$ and that these belong to different connected components of $\Nklt(X_n',B_n',\alpha \bM.)$, where we identify $\tilde D$ with its strict transform on $X_n'$.
Thus, $\K X_n'. + (B_n'- \tilde D - E')+ \alpha M_n' \sim \subs \qq,Z. -\tilde D - E'$ is not pseudo-effective over $Z$, as $\tilde D$ dominates $Z$.
Then, by~\cite{BZ16}*{Lemma 4.4}, we may run a $(\K X_n'. + (B_n'- \tilde D - E')+ \alpha M_n')$-MMP over $Z$, which terminates with a Mori fiber space $\hat X _n \rar \hat Z$ over $Z$.
Since $X_n \rar Z$ has relative dimension 1, $\hat Z \rar Z$ is birational.
Arguing as in the first part of the proof, we know that distinct connected components of $\Nklt(X'_n,B'_n,\alpha \bM.)$ have to remain disjoint after the run of the MMP.
Therefore, as the MMP is positive for $\tilde D + E'$ and these divisors need to remain disjoint, these two divisors are not contracted by $X_n' \drar \hat X _n$.
Call $\hat D$ and $\hat E$ the corresponding strict transforms on $\hat X$.
Since $\hat X _n \rar \hat Z$ is a Mori fiber space and $\hat D$ is relatively big, it is relatively ample.
Since $\hat X _n \rar \hat Z$ has relative dimension 1, $\hat D$ is horizontal and $\hat D \cap \hat E = \emptyset$, then $\hat E$ has to be horizontal over $\hat Z$.
This is absurd, as its corresponding center on $X$ does not dominate $Z$.

Finally, Lemma~\ref{condition.2.lemma} implies that condition (2) in Definition~\ref{def p1 link} holds.
\end{proof}

\begin{proof}[Proof of Theorem~\ref{Theoremp1 link}]
We follow the proof of~\cite{Kol13}*{Theorem~4.40} and divide the proof into 2 steps.
\\

\textbf{Step 1:} 
{\it 
In this step  we prove the statement of the theorem over an \'etale neighborhood $(s' \in S') \rar (s \in S)$ such that $k(s) \simeq k(s')$.}
\\
We proceed by induction on $\dim (X)$ and $\dim (Z)$.
If $f \sups -1.(s) \cap \lfloor B \rfloor$ is disconnected, then, by Theorem~\ref{main theorem}, after an \'etale base change $(s' \in S') \rar (s \in S)$ there are exactly two generalized log canonical centers intersecting the fiber over $s'$, and they are $\pr 1.$-linked.
Thus, the claim follows.
\newline
Now, we can assume that $f \sups -1. (s) \cap \lfloor B \rfloor$ is connected.
Write $\lfloor B \rfloor = \sum D_i$, where each $D_i$ is a prime Weil divisor.
Then, up to an \'etale base change that does not change $k(s)$~\cite{Kol13}*{\S~4.38}, we can assume that each $D_i$ has connected fibers over $s$, and that every generalized log canonical center of $(X,B+M)$ intersects $f \sups -1. (s)$.
By the connectedness of $f \sups -1.(s) \cap \lfloor B \rfloor$, up to reordering, we can assume that $Z \subset D_1$, $W \subset D_r$, and $f \sups -1. (s) \cap D_i \cap D \subs i+1. \neq \emptyset$ for $i=1,\ldots , r-1$.
\newline
By induction on the dimension, we may apply Theorem~\ref{Theoremp1 link} to $D_1 \rar S$ with $Z$ as minimal generalized log canonical center and $D_1 \cap D_2$ as the other center.
It follows that there is a generalized log canonical center $Z_2 \subset D_1 \cap D_2$ that is $\pr 1.$-linked to $Z$.
By Remark~\ref{remark minimality}, $Z_2$ is also minimal with respect to inclusion among the generalized log canonical centers of $(X,B+M)$ that intersect $f \sups -1.(s)$.
Let $(D_1,B_1 + M_1)$ and $(D_2,B_2+M_2)$ denote the generalized pairs induced by generalized adjunction on $D_1$ and $D_2$ respectively.
Notice that $Z_2$ is a generalized log canonical center of $(D_1,B_1+M_1)$.
Then, by the generalized dlt property and generalized adjunction~\cite{Bir16a}{\S~3.1}, it follows that $Z_2$ is a generalized log canonical center also for $(X,B+M)$ and $(D_2,B_2+M_2)$.
To conclude, we apply this argument inductively to consecutive prime component $D_i$ and $D \subs i+1.$, until we have $i+1 = r$.
This process produces a minimal generalized log canonical center $Z_r \subset D_r$ which is $\pr 1.$-linked to $Z$.
Since $Z_r$ may not be contained in $W$, we apply the inductive hypothesis to the morphism $D_r \rar S$ with $Z_r$ and $W$ as the centers involved.
This process produced a new generalized log canonical center $Z_W \subset W$ with the claimed properties.
\\

\textbf{Step 2:} 
{\it We prove that the \'etale base change is not necessary}.
\newline
Let $g \colon X \rar T$ be the Stein factorization of $f$, and let $t \in T$ be the unique preimage of $s$ in $T$.
Let $Z_1, \ldots, Z_k$ be the minimal log canonical centers with respect to inclusion such that $s \in f(Z_i)$.
Generalized log canonical centers commute with \'etale base change.
Thus, by the previous step, all the $Z_i$ are $\pr 1.$-linked to each other after a suitable base change.
Therefore, there is a unique subvariety $V \subset T$ such that $g(Z_i) = V$ for every $i$.
\newline
Let $v \in T$ be the generic point of $V$.
Since $g$ has connected fibers, we can apply the Step 1 to $g \colon (X,B+M) \rar T$ and $v$.
Thus, we get an \'etale base change $\pi \colon (v' \in T') \rar (v \in T)$ that induces an isomorphism
\begin{equation} 
\label{equation canonical iso}
\pi \colon (g') \sups -1.(v') \simeq g \sups -1.(v).
\end{equation}
Thus, each $Z_i$ is canonically isomorphic to a minimal generalized log canonical center $Z'_i \subset X \times_T T'$.
The centers $Z'_i$ are $\pr 1.$-linked to each other by Step 1.
By~\eqref{equation canonical iso}, the $\pr 1.$-links descend to $\pr 1.$-links between the $Z_i$.
\end{proof}

\begin{proof}
[Proof of Theorem~\ref{dual.main.thm}]
By Theorem~\ref{dual complex equidimensional}, we know that $\mathcal{DMR}(X,B, \mathbf M)$ is equidimensional.
If the dual complex is contractible to a point, there is nothing to prove.
Otherwise, the observation in Remark~\ref{remark reduction} implies that the result follows from the analogous result for log pairs proved in~\cite{KX16}.
\end{proof}

\begin{proof}
[Proof of~\ref{thm non lc}]
Without loss of generality, we can assume that $X$ is $\qq$-factorial, and that $(X,\Delta,\mathbf M)$ is generalized dlt, where $\Delta \coloneqq B \wedge \Supp(B)$.
As $(X,B,\mathbf M)$ is not generalized log canonical, then $B - \Delta > 0$ and $\K X. + \Delta + \mathbf M _X$ is not pseudo-effective.
Hence, we can run a $(\K X. + \Delta + \mathbf M _X)$-MMP with scaling of an ample divisor $H$
\begin{align}
\label{eqn.proof.1_7}
\xymatrix{
X=X_0 \ar@{-->}[r] & X_1 \ar@{-->}[r] & \dots \ar@{-->}[r]& X_{n-1} \ar@{-->}[r] & X_n
}, 
\end{align}
which terminates with a Mori fiber space $g \colon X_n\rar Z$.
This MMP is also a $-(B - \Delta)$-MMP.
We shall denote by $\Gamma_i$ the strict transform on $X_i$ of a divisor $\Gamma$ on $X$.
\newline
Let $R_i$ be the extremal ray corresponding to the $i$-th step $X_{i-1} \drar X \subs i.$ of~\eqref{eqn.proof.1_7}.
Thus, we have $(B_i - \Delta_i) \cdot R_i > 0$ and there exists a prime component $D_i$ of $\Supp(B_i \sups > 1.) \subset \Supp (B_i \sups \geq 1.)$ satisfying $D_i \cdot R_i > 0$.
\\

{\bf Claim}. {\it For all $i$, $\mathcal{D}(B_i \sups \geq 1.)$ and $\mathcal{D}(B \subs i+1. \sups \geq 1.)$ are simple homotopy equivalent.}
\begin{proof}[Proof of Claim]
As the MMP in~\eqref{eqn.proof.1_7} terminates, there exists $0 < \epsilon \ll 1$ such that MMP is also a run of the $(K_X+\Delta+(\mathbf{M}_X+\epsilon {H}))$-MMP.
In particular, at each step of~\eqref{eqn.proof.1_7}, $H_i$ is a big divisor and $\mathbb{B}_+(H_i)$ does not contain any generalized log canonical center of $(X_i, \Delta_i, \mathbf{M})$.
Since the generalized dlt property is preserved under the steps of the MMP in~\eqref{eqn.proof.1_7}, see Definition~\ref{def dual complex nonlc}, for all $i$, $\mathcal{DMR}(X_i,B_i,\mathbf M) = \mathcal D (B_i \sups \geq 1.)$.
Hence, the claim is a direct consequence of Lemma~\ref{aux.dlt.lemma1} and~\cite{dFKX}*{Theorem~19}, since, as we noted above, there exists a prime component $D_i$ of $\Supp(\Delta_i)$ such that $D_i \cdot R_i >0$.
\end{proof}
\noindent
The claim implies that 
\begin{align*}
\mathcal{DMR}(X,B,\mathbf M) = \mathcal{DMR}(X_n,B_n,\mathbf M).
\end{align*}
As $\mathbf M \subs X_n.$ is pseudo-effective and $g$ is a Mori fiber space, then $-(\K X_n. + \Delta_n)$ is $g$-ample; 
as $\mathcal D (B_n \sups \geq 1.)= \mathcal D (\Delta_n \sups =1.)$, we can apply~\cite{Nak19}*{Lemma~3.1} to conclude that $\mathcal D (\Delta_n \sups =1.)=\mathcal{DMR}(X,B,\mathbf M)$ is contractible.
\end{proof}
\begin{lemma}
\label{aux.dlt.lemma1}
Let $(X,B,\bM.)$ be a $\qq$-factorial generalized dlt pair.
Let $H$ be a big divisor such that $\mathbb{B}_+(H)$ does not contain any generalized lc center of $(X,B,\bM.)$.
Then, for any $0<\epsilon \ll 1$, there exists an effective divisor $\Gamma_\epsilon$ such that $(X,B+\Gamma_\epsilon)$ is a dlt pair,
\begin{align*}
\K X. + B + \bM X. + \epsilon H \sim_\qq \K X. + B + \Gamma_\epsilon,
\end{align*}
and $\Supp(B \sups =1.)=\Nklt(X,B,\bM.)=\Nklt(X,B+\Gamma_\epsilon)$.
Furthermore, the dual complexes of the pair $(X,B+\Gamma_\epsilon)$ and of the generalized pair $(X,B,\bM.)$ agree. 
\end{lemma}
\begin{proof}
As $\mathbb{B}_+(H)$ does not contain any generalized log canonical center of $(X,B,\bM.)$, we can write $H \sim_\mathbb{Q} A+E$, where $A$ is ample and $E$ effective in such a way that for $0 < \epsilon \ll 1$, $(X,B+\epsilon E,\bM.+\epsilon\overline{A})$ is still generalized dlt, its non-klt locus coincides with that of $(X, B, \bM.)$ and its dual complex coincides with that of $(X, B, \bM.)$.
In view of this, thus, it suffices to show that the conclusion of the lemma holds if we substitute $H$ with $A$.
Hence, we shall assume that $H$ is an ample divisor.
\\
As $(X,B,\bM.)/\cc$ is generalized dlt, $\bM .$ descends to a neighborhood of the generic point of each generalized log canonical center.
In particular, there exists a closed subset $Z \subset X$ so that $\bM.$ descends to $X \setminus Z$ and no generalized log canonical center is contained in $Z$.
By~\cite{KM98}*{Lemma~2.45}, we can find a birational morphism $\pi \colon X' \rar X$ from a normal birational model $X'$ so that $\pi$ is an isomorphism over $X \setminus Z$ and $\bM.$ descends to $X'$.
We write $(X',B',\bM.)$ for the trace of $(X,B,\bM.)$ on $X'$.
\\
As $X$ is $\qq$-factorial, we can find an ample divisor $H'$ on $X'$ so that $H \coloneqq \pi_*(H')$ is ample on $X$ and $\pi^\ast (H)-H'=F' \geq 0$ is $\pi$-exceptional\footnote{Let $A'$ be an ample divisor on $X'$, and let $D$ be an ample divisor on $X$.
Then, for every $\delta>0$, $\pi^\ast (D)+\delta A'$ is ample.
As ampleness is an open condition, $D+\delta \pi_*(A')$ is ample for $0 < \delta \ll 1$.
We fix an ample divisor $H'$ on $X'$ such that $H\coloneqq \pi_*(H')$ is ample.
By the negativity lemma~\cite{KM98}*{Lemma~3.39}, $\pi^\ast (H)-H'=F' \geq 0$, where $F'$ is $\pi$-exceptional.
As $\bM X'.$ is nef and $H'$ is ample on $X'$, then $\bM X'. + \epsilon H'$ is ample.}.
As $\pi(\Supp(F')) \subset Z$ and $Z$ does not contain any generalized log canonical center of $(X,B,\bM.)$, $\Supp(F')$ does not contain any generalized log canonical center of $(X',B',\bM.)$.
In particular, for $0 < \epsilon \ll 1$, we have 
\begin{equation} \label{eq_nklt}
\Nklt(X',B') = \Nklt(X',B',\bM.) = \Nklt(X',B'+\epsilon F',\bM.)=\Nklt(X',B'+\epsilon F'),
\end{equation}
where the first and last equalities comes from the fact that $\bM.$ descends to $X'$.
As the equalities in~\eqref{eq_nklt} also hold for each generalized log canonical place, then adding $\epsilon F$ does not introduce new generalized log canonical places and preserves the generalized sub-log canonical property, so that 
\begin{align*}
\mathcal{D}((B') \sups = 1.) = \mathcal{D}((B' +\epsilon F') \sups = 1.)
\end{align*}
Let us fix $0 < \epsilon \ll  1$ so that the properties just discussed hold.
As $\bM X'. + \epsilon H'$ is ample, there exists $\Gamma'_\epsilon \sim_\qq \bM X'. + \epsilon H'$ such that $\Nklt(X',B'+\epsilon F')=\Nklt(X',B'+\Gamma'_\epsilon + \epsilon F')$, and adding $\Gamma'_\epsilon$ does not introduce any new log canonical places.
Moreover, by construction,
\begin{align*}
\K X'. + B' + \Gamma'_\epsilon + \epsilon F' \sim_\qq \pi^\ast (\K X. + B + \bM X. + \epsilon H), \ \text{and}\\
\mathcal{D}((B') \sups = 1.) = 
\mathcal{D}((B' +\epsilon F') \sups = 1.) =
\mathcal{D}((B' +\Gamma_\epsilon'+\epsilon F') \sups = 1.).
\end{align*}
Defining $\Gamma_\epsilon \coloneqq \pi_*(\Gamma'_\epsilon) = \pi_*(\Gamma'_\epsilon + \epsilon F')$, the pair $(X,B+\Gamma_\epsilon)$ satisfies the claims of the statement.
\end{proof}

\end{document}